\documentclass[moor]{informs1}              

\usepackage{natbib}
 \NatBibNumeric
 \def\bibfont{\small}%
 \bibpunct[, ]{[}{]}{,}{n}{}{,}%

\usepackage[colorlinks=true,breaklinks=true,bookmarks=true,urlcolor=blue,
     citecolor=blue,linkcolor=blue,bookmarksopen=false,draft=false]{hyperref}

\usepackage{amsmath,amssymb, amsfonts,subcaption,caption}
\usepackage[noabbrev,capitalise,nameinlink]{cleveref}
\usepackage{float}
\usepackage{graphicx}
\usepackage{epstopdf}
\usepackage{multirow} 
\usepackage{algorithm}
\usepackage{algorithmic}
\usepackage{enumitem}
\usepackage{chemarr}

\def\bfx{{\bf x}}

\def\bfb{{\bf b}}

\def\bfd{{\bf d}} 
\def\bfu{{\bf u}}
\def\bfw{{\bf w}}
\def\bfv{{\bf v}} 
\def\bfy{{\bf y}} 
\def\bfg{{\boldsymbol g}} 
\def\bfxi{{\boldsymbol \xi}} 
\def\L{{\boldsymbol \Lambda}}

\def\ts{$\tau$-stationary point}

\def\beq{\begin{eqnarray}}
\def\eeq{\end{eqnarray}}

\def\ve{{\rm vec}}
\def\nscp{{\tt SNSCO}} 
\def\rel{{\tt RelA}} 
\def\reg{{\tt RegAC}}
 
\def\M{{\cal M}}
\def\N{{\cal N}}
\def\K{{\cal K}}
\def\J{{\cal V}}
\def\S{{\cal S}}
\def\CF{{\cal F}}
\def\CE{{\cal C}}
\def\CD{{\cal D}}

\def\A{{\bf A}}
\def\B{{\bf B}}
\def\G{{\bf G}}
\def\Z{{\bf Z}}
\def\W{{\bf W}}
\def\F{{\bf F}}
\def\H{{\bf H}}

\def\P{{\bf P}}
\def\Q{{\bf Q}}

\def\R{{\mathbb R}}
\def\T{{\mathbb T}}

\def\rT{{\rm T}}
\def\rN{{\rm N}}

\def\mz{\Z^{\max}}
\def\zs{\Z^{\downarrow}_s}
\def\bpi{{\rm \Pi}}
\def\rmo{{\rm \Omega}}
\def\eqspace{\arraycolsep=1pt\def\arraystretch}
 
\graphicspath{ {./Figures/} }

\def\EMAIL#1{\href{mailto:#1}{#1}}
\def\URL#1{\href{#1}{#1}}         

\TheoremsNumberedThrough     

\EquationsNumberedThrough    

\begin{document}
\flushbottom


\RUNAUTHOR{S. Zhou, L. Pan, N. Xiu, G. Li}

\RUNTITLE{$0/1$ Constrained Optimization Solving SAA for CCP}

 \TITLE{\Large A $0/1$ Constrained Optimization Solving Sample Average Approximation for Chance Constrained Programming }

\ARTICLEAUTHORS{%
\AUTHOR{Shenglong Zhou}
\AFF{School of Mathematics and Statistics, Beijing Jiaotong University, China, \EMAIL{slzhou2021@163.com}  \URL{}}
\AUTHOR{Lili Pan}
\AFF{Department of Mathematics, Shandong University of Technology, China, \EMAIL{panlili1979@163.com}  \URL{}}
\AUTHOR{Naihua Xiu}
\AFF{School of Mathematics and Statistics, Beijing Jiaotong University, China, \EMAIL{nhxiu@bjtu.edu.cn}  \URL{}}
\AUTHOR{Geoffrey Ye Li}
\AFF{ITP Lab, Department of EEE, Imperial College London, UK, \EMAIL{geoffrey.li@imperial.ac.uk}  \URL{}}
} 

\ABSTRACT{%
Sample average approximation (SAA) is a tractable approach for dealing with chance constrained programming, a challenging stochastic optimization problem. The constraint of SAA is characterized by the $0/1$ loss function which results in considerable complexities in devising numerical algorithms. Most existing methods have been devised based on reformulations of SAA, such as binary integer programming or relaxed problems. However, the development of viable methods to directly tackle SAA remains elusive, let alone providing theoretical guarantees. In this paper, we investigate a general $0/1$ constrained optimization, providing a new way to address SAA rather than its reformulations. Specifically, starting with deriving the Bouligand tangent and Fr$\acute{e}$chet normal cones of the $0/1$ constraint, we establish several optimality conditions. One of them can be equivalently expressed by a system of equations, enabling the development of a semismooth Newton-type algorithm. The algorithm demonstrates a locally superlinear or quadratic convergence rate under standard assumptions, along with nice numerical performance compared to several leading solvers. 
}%


\KEYWORDS{$0/1$ constrained optimization; sample average approximation;  chance constrained programming;     optimality conditions;  semismooth Newton method; 	locally  convergence rate}

\maketitle

\section{Introduction}
 Chance constrained programming (CCP) is an efficient tool for decision-making in uncertain environments to hedge risk and thus has been extensively studied recently \cite{adam2016nonlinear,geletu2017inner,curtis2018sequential,
 sun2019Convergence, pena2020solving}.
It has a wide range of applications, such as supply chain management \cite{lejeune2007efficient}, optimization of chemical processes \cite{henrion2003optimization}, surface water quality management \cite{takyi1999surface}, just naming a few. A simple version of CCP problem takes the form of
\begin{equation}\label{CCP}
\underset{{\bfx\in{ \Omega}}}{\min}~ f(\bfx),~~{\rm s.t.}~ \mathbb{P}\{\bfg(\bfx,\bfxi)\leq {\bf0}\}\geq 1-\alpha, \tag{CCP}
\end{equation} 
where $f:\R^{K}\rightarrow\R$ is continuously differentiable, $\bfg(\bfx,\bfxi)=(g_1(\bfx,\bfxi),\ldots,g_M(\bfx,\bfxi))^\top:\R^{K}\times \Xi\rightarrow\R^M$ with $g_m(\cdot,\bfxi)$ being continuously differentiable on $\R^{K}$ for each $m$,  $\bfxi$ is a random vector with a probability distribution supported on set $\Xi\subset \R^D$,  ${\bf0}$ is the zero vector,    $\alpha$ is a confidence parameter chosen by the decision maker, typically near zero, e.g. $\alpha=0.01$ or $\alpha=0.05$, { and  $\Omega\subseteq \R^{K}$ is a closed and convex set.} 
We note that $\{\bfg(\bfx,\bfxi)\leq {\bf0}\}$ represents the feasible region described by a group of constraints subject to uncertainty $\bfxi$. The constraint is called single chance constraint if $M=1$ \cite{charnes1958cost} and joint chance constraint if $M>1$ \cite{miller1965chance}. Some general theory can be found in \cite{prekopa1973contributions,prekopa2013stochastic} and the references therein.

\subsection{Related work}

Problem \eqref{CCP} is difficult to solve numerically in general for two reasons. The first reason is the hardness of computing quantity $\mathbb{P}\{\bfg(\bfx,\bfxi)\leq {\bf0}\}$ for a given $\bfx$, since it requires multi-dimensional integration. The second reason is the non-convexity of the feasible set even if $\bfg(\cdot,\bfxi)$ is convex. Therefore, to solve the problem, one common strategy is to make some assumptions on the distributions of $\bfxi$.
For example, for the case of $\bfg(\bfx)=\A(\bfxi)\bfx+\bfb(\bfxi)$ and the single chance constraint, the feasible set is convex if $\alpha<0.5$ and $\bfxi$ has a nondegenerate multivariate normal distribution \cite{kataoka1963stochastic}, or if
$(\A(\bfxi);\bfb(\bfxi))$ has a symmetric log-concave density
\cite{lagoa2005probabilistically}. Moreover, if $\bfxi$ has an elliptically symmetric distribution, then the feasible set can be expressed as a second-order cone constraint  \cite{henrion2007structural}.  Furthermore,  the gradient of the linear joint chance constraint can be derived in an explicit formula under a multivariate  Gaussian distribution \cite{henrion2012gradient}.   However, when not making assumptions on $\bfxi$, many approaches leverage sampling to approximate the probabilistic constraint. 

  {\bf a)  SAA approaches.}  The fundamental idea of SAA to address CCP lies in using the empirical distribution function to approximate the true distribution function. More specifically, let $\bfxi_1,\ldots,\bfxi_N$ be independent and identically distributed samples of $N$ realizations of random vector $\bfxi$. As the constraint in \eqref{CCP} is equivalent to {$\mathbb{P}\{{\max}_{m=1,2,\ldots,M} g_m(\bfx,\bfxi) >0 \}\leq \alpha$},  through its empirical distribution function, SAA  takes the form of
\begin{equation}\label{SAAP} 
\begin{array}{l}
\underset{{\bfx\in{ \Omega}}}{\min}~ f(\bfx),~~{\rm s.t.}~ \frac{1}{N}\sum_{n=1}^N \ell_{0/1}  ( {\max}_{m=1,2,\ldots,M} g_m(\bfx,\bfxi_n)   )\leq \alpha,
 \end{array} \tag{SAA}\end{equation}
where  
  $\ell_{0/1}(t)$ is the $0/1$ loss function  \cite{friedman1997bias, LL2007, hastie2009elements} defined as
\begin{eqnarray}\label{0/1-loss-func}
\ell_{0/1}(t):=\left\{\begin{array}{ll}
1,&~t>0,\\[1ex] 
0,&~t\leq0.
\end{array}\right.
\end{eqnarray}
{The function is also known as the (Heaviside) step function \cite{osuna1998reducing, evgeniou2000regularization}}.  We point out that \eqref{SAAP}  with $\alpha>0$ is always non-convex, but it gains popularity since it requires relatively few assumptions on the structure of \eqref{CCP} or the distribution of $\bfxi$  \cite{luedtke2014branch}. Therefore,  there is an impressive body of work on developing numerical algorithms \cite{pagnoncelli2009sample,pena2020solving,shapiro2021lectures} and establishing asymptotic convergence \cite{luedtke2008sample,sun2014asymptotic,sun2019Convergence} for \eqref{SAAP}. 
 \begin{itemize}[leftmargin=17pt]
 \item \textit{Binary integer programming (BIP).} 
 One way to solve \eqref{SAAP} is to reformulate it as  BIP \cite{ahmed2008solving}.  
For example,   a single chance constrained problem has been investigated in \cite{adam2016nonlinear} and the integer variables in the BIP were relaxed as continuous ones. In \cite{luedtke2010integer}, the authors have solved a linear \eqref{SAAP} problem (i.e., $f$ and $g$ are linear) by solving its BIP reformulation. Moreover,   a branch-and-cut decomposition method \cite{luedtke2014branch} and a branch-and-bound approach \cite{beraldi2010exact} have been proposed to deal with the BIP reformulations.  
\item \textit{Nonlinear programming (NLP).}   
Alternative approaches tackle \eqref{SAAP} by formulating it as NLP problems. For instance, the author in \cite{curtis2018sequential} proposed a cardinality-constrained NLP problem solved by a sequential algorithm. The algorithm comprised computing quadratic optimization subproblems with linear cardinality constraint and was proven to converge to a stationary point of a novel penalty function.  In \cite{pena2020solving}  the constraint in \eqref{SAAP} was converted to a quantile constraint, resulting in an NLP problem to be solved by a trust-region method.
 
 \end{itemize}

It is noted that most of the aforementioned work focused on surrogates of \eqref{SAAP}, without providing thorough optimality analysis or directly developing algorithms for \eqref{SAAP}.

{\bf b) Scenario approaches.} Differing from SAA relaxing the probabilistic constraint by the empirical distribution function, scenario approaches aim at solving the following problem,
\begin{equation}\label{SAS}
 \min_{{\bfx\in{ \Omega}}}~ f(\bfx),~~{\rm s.t.}~    g_m(\bfx,\bfxi_n)   \leq 0,~ m=1,\ldots,M,~ n=1,\ldots,N.\end{equation}
In fact, this model is a special case of \eqref{SAAP}, corresponding to $\alpha=0$, and is quite strict. It {has been shown} in \cite{de2004constraint,calafiore2005uncertain,erdougan2006ambiguous}
  that when {the sample size} $N$ is large enough,  the constraints in \eqref{SAS} can ensure the satisfaction of the chance constraint {with a high probability. However, such approaches suffer from several drawbacks \cite{hong2011sequential}. Moreover, one of our main results show that an optimal solution to \eqref{CCP} might be infeasible to \eqref{SAS}, which has been verified by empirical experiments in \cite{luedtke2008sample}. } 
  
{\bf c)  Other reformulations.} To avoid solving mixed-integer programming, some other approximation methods have been proposed.    For instance, in cases where both $f$ and $\bfg(\cdot,\bfxi)$ are convex, the authors \cite{nemirovski2007convex} introduced a general class of convex conservative approximations for the chance constraints. Additionally, the authors \cite{hong2011sequential} employed the difference-of-convex (DC) functions to approximate the step function, deriving a DC approximation of problem \eqref{SAAP} solved by a gradient-based Monte Carlo method. For non-linear CCP,  a smooth approximation approach has been proposed to relax the CCP problem by two parametric NLP problems that can be tackled using NLP solvers\cite{geletu2017inner}. 
Moreover, a large body of work has been dedicated to developing approaches for nonlinear CCP based on the spherical radial decomposition of elliptically distributed random vectors. These approaches leverage specific information about the distributions, thereby leading to significantly reduced variance when estimating values and gradients of probability functions. Typical distributions include  Gaussian or Gaussian-like distributions \cite{van2014gradient,van2017sub,hantoute2019subdifferential}, elliptically symmetric distributions \cite{van2018sub}, and log-normal and Student's t-distributions \cite{van2017second}.  One can refer to    \cite{gotzes2016quantification, gonzalez2017joint, van2019generalized, heitsch2019probabilistic, van2020gradient, farshbaf2020optimal, berthold2022algorithmic, van2022generalized} for more details.

\subsection{The main model}
To simplify the constraint in \eqref{SAAP}, we define {a measure} by
\beq\label{heaviside}
\begin{array}{l}
\|\Z\|^+_0:= \sum_{n=1}^N \ell_{0/1}\left(\mz_{:n}\right),
\end{array}
\eeq
where  matrix $\Z\in\R^{M\times N}$ and $\mz_{:n}$ denotes the maximum entry of the $n$th column  of $\Z$, namely,
\begin{eqnarray*} 
  \arraycolsep=1pt\def\arraystretch{1.5}
\begin{array}{lcllll}
\mz_{:n}&:=&\max_{m=1,\ldots,M}Z_{mn}.
\end{array}
\end{eqnarray*}  
One can observe that $\|\Z\|^+_0$ counts the number of columns in $\Z$ with positive maximum values. Motivated by \eqref{SAAP}, in this paper, we study  $0/1$ (or step) constrained optimization (SCO):
\begin{equation}\label{SCP}
{ \min_{\bfx\in{ \Omega}}~ f(\bfx),~~{\rm s.t.}~ \bfx\in\CF:=\{\bfx\in\R^K:\|\G(\bfx)\|^+_0\leq s\},}\tag{SCO}
\end{equation}
where $\G(\bfx) :\R^{K}\rightarrow\R^{M\times N}$ is continuously differentiable  and $s\ll N$ is a given positive integer. Hereafter, the entry in the $m$th row and $n$th column of $\G(\bfx)$ is denoted by $G_{mn}(\bfx)$, {and we always assume that $\Omega\cap\CF\neq\emptyset$.} The above problem is NP-hard due to the discrete nature of the $0/1$ loss function. {However, it can be applied to deal with various applications. For example, if let $$\G(\bfx)=(\bfg(\bfx,\bfxi_1),\bfg(\bfx,\bfxi_2), \ldots, \bfg(\bfx,\bfxi_N))$$ and $s=\lceil \alpha N\rceil$, the minimal integer no less than $\alpha N$, then model \eqref{SCP} turns to \eqref{SAAP}. We emphasize that besides addressing \eqref{SAAP}, model \eqref{SCP} can also handle many other deterministic problems, such as support vector machine  \cite{CV95, wang2021support} and one-bit compressed sensing \cite{BB08,zhou2022computing} by taking $\G(\bfx)=\A\bfx+\bfb$, $M=1$, and $\Omega=\R^K$, where $\A$ and $\bfb$ are a deterministic matrix and vector without involving the realizations of random variable $\bfxi$. In this regard, (SCO) is more general than (SAA).}

\subsection{Contributions}
To the best of our knowledge, this is the first paper to directly address \eqref{SAAP} with 0/1 constraints and to provide thorough theoretical guarantees along with a viable numerical algorithm. The main contributions of this paper are threefold.

\begin{itemize}[leftmargin=17pt]

\item[1)] \emph{Variational properties and optimality conditions.} Despite the challenges posed by the $0/1$ loss function, we manage to build some theoretical properties. We begin by calculating the projection of a point onto set $\S$ and deriving the Bouligand tangent and Fr$\acute{e}$chet normal cones of sets $\S$ and $\CF$, see  Propositions \ref{pro-1} and \ref{tangent-S}, where 
\begin{eqnarray}
\label{l01-s}
\S&:=&\left\{\Z\in\R^{M\times N}:~\|\Z\|^+_0\leq s\right\}.
\end{eqnarray}
The properties on $\CF$ further allow us to conduct comprehensive optimality analysis. Specifically, we introduce a KKT point and define a $\tau$-stationary point,  revealing their relationships to local minimizers of problem \eqref{SCP}. The relationships are shown in \eqref{relation}, where conditions and cases are outlined in  Corollary \ref{coro-relation}. Furthermore,   we also introduce a binary KKT (BKKT) point to investigate the optimality conditions of a binary integer programming, an equivalent reformulation of problem \eqref{SCP}. The relationships among the BKKT point and the other points are also illustrated in \eqref{relation}, from which we can conclude that a $\tau$-stationary point is a better solution than a KKT or BKKT point. 
 \begin{equation} \label{relation}
  \fbox{\text{$
 \eqspace{1.35}
 \begin{array}{lcl}
 &\text{Local  minimizer}& \\
&\scriptsize{(\textbf{Cond 1})}\downharpoonleft \upharpoonright \scriptsize{(\textbf{Cond 2})}&~~~ \searrow\scriptsize{(\textbf{Cond 3})} \\
\tau\text{-stationary point}~~\xrightleftharpoons[(\textbf{Case 1}~or~\textbf{Case 2})]{}
~~&\text{KKT point}&~~\longrightarrow~~\text{BKKT point}\\
\end{array}$} }\end{equation}

\item[2)]   \emph{A semismooth Newton-type method with locally quadratic convergence.} 
An advantageous property of a $\tau$-stationary point is its equivalence to a system of equations, which allows us to leverage the smoothing Newton method for solving \eqref{SCP}, dubbed as \nscp. We then prove that the proposed algorithm has a locally superlinear or quadratic convergence rate under the standard assumptions, as outlined in Theorem \ref{the:quadratic}. The endeavour of attaining this result underscores the non-triviality of the proof. 

\item[3)] \emph{A nice numerical performance.}
 We compare \nscp\ with several selected algorithms and {\tt GUROBI} for solving norm optimization problems. It is capable of delivering solutions of comparable quality to these algorithms. In particular, for instances of large sizes, it runs much faster than the others without compromising solution quality.
 
\end{itemize}

\subsection{Organization}
The paper is organized as follows. In Section \ref{sec:sta}, we establish several statistical properties of \eqref{SCP} when it reduces to \eqref{SAAP}.  Section \ref{sec:sets} involves the calculations of the projection of one point onto set $\S$, as well as the determination of the Bouligand tangent and normal cones of $\CF$ and $\S$. In Section \ref{sec:opt},  we define KKT points and $\tau$-stationary points, both of which serve as optimality conditions of \eqref{SCP}.  Moreover, we also introduce BKKT points to investigate the optimality condition for a BIP reformulation of \eqref{SAAP}. Furthermore, the relationships among these three kinds of points and local minimizers are revealed. 
In Section \ref{sec:Newton}, by equivalently rewriting the $\tau$-stationary point as a system of $\tau$-stationary equations, we develop a semismooth Newton method, \nscp, to solve the equations and establish its locally superlinear and quadratic convergence rates.  In Section \ref{sec:numerical}, we implement \nscp\ to solve norm optimization problems and compare it with several leading solvers.  Concluding remarks are given in the last section.

\subsection{Notation}
We end this section by defining some notation. Denote 
\begin{eqnarray*}
\begin{array}{cll}
\M:=\{1,2,\ldots,M\}~~\text{and}~~\N:=\{1,2,\ldots,N\}.
\end{array}
\end{eqnarray*}
Given a subset $T\subseteq \N$, its cardinality and complement  are $|T|$ and $\overline T:=\N\setminus T$. For a scalar $a\in\R$, $\lceil a\rceil$ represents the smallest integer no less than $a$. For two matrices  $\A=(A_{mn})_{M \times N}\in\R^{M \times N}$ and $ \B=(B_{mn})_{M \times N}\in\R^{M \times N}$, we  denote
\begin{eqnarray}\label{notation-v}
\begin{array}{cll}
  \A ^+ := (\max\{A_{mn},0\})_{M \times N}~~\text{and}~~
   \A ^- := (\min\{A_{mn},0\})_{M \times N}.
\end{array}
\end{eqnarray}
and their inner product is $\langle \A,\B\rangle:=\sum_{mn}A_{mn}B_{mn}$. Moreover, let
${\bf0}\geq \A \perp \B \geq {\bf0}$ stand for $0\geq A_{mn}, B_{mn} \geq 0$, and $A_{mn}B_{mn}=0$ for any $m\in\M,n\in\N$. We use $\|\cdot\|$ to represent the Euclidean norm for vectors and Frobenius norm for matrices. A positive definite matrix is written as $\A\succ  {\bf0}$. The neighbourhood of $ \bfx\in\R^{K} $ with a radius $\epsilon>0$ is 
\begin{equation*}
{\mathbb N}( \bfx ,\epsilon) := \{\bfv\in\R^{K}: \|\bfv- \bfx \|<\epsilon\}.
\end{equation*} 
For matrix $\Z\in\R^{M\times N}$, let $\Z_{\J}$ be the sub-part indexed by $\J\subseteq \M\times {\N}$, namely $\Z_{\J}:=\{Z_{mn}\}_{(m,n)\in \J}$. Particularly,  let $\Z_{:T}$ stand  for the sub-matrix containing columns indexed by $T\subseteq {\N}$, and $\Z_{:n}$  represent the $n$th column of $\Z$.  In addition,  we define the following useful index  sets:
\begin{eqnarray}\label{notation-z}
  \arraycolsep=1pt\def\arraystretch{1.5}
\begin{array}{lcllll}
 \Gamma_+&:=&\left\{n\in{\N}:~\mz_{:n} >0  \right\},\\
 \Gamma_0&:=&\left\{n\in{\N}:~ \mz_{:n} = 0 \right\},\\
 \Gamma_-&:=&\left\{n\in{\N}:~ \mz_{:n} < 0 \right\},\\
 V_{\Gamma}&:=&\left\{(m,n)\in \M\times \Gamma:~{Z}_{mn} = 0\right\},~~\Gamma\subseteq {\N}.
\end{array}
\end{eqnarray}
We point out that $\Gamma_+, \Gamma_0, \Gamma_-$, and $V_{\Gamma}$  depend on $\Z$, but we will drop their dependence if no additional explanations are provided. Recalling \eqref{heaviside}, the above definitions indicate  that
\begin{eqnarray} \label{0+norm}
 \begin{array}{lll}
\|\Z\|^+_0=|\Gamma_+|.
\end{array}
\end{eqnarray}
Let $\Z^\downarrow_r$ be the $r$th largest element of $\{\|(\Z_{:1})^+\|,\|(\Z_{:2})^+\|,\ldots,\|(\Z_{:N})^+\|\}$. Then
\begin{eqnarray}
\zs \left\{\begin{array}{lll}
 =0, &~\text{if}~\|\Z\|^+_0<s,\\[1ex]
 >0, &~\text{if}~ \|\Z\|^+_0\geq s.
\end{array}\right.
\end{eqnarray}
For $\G(\bfx):\R^{K} \rightarrow\R^{M\times N}$ and an index set $\J\subseteq\M\times\N$, we denote   $\nabla_\J \G(\bfx)\in\R^{K\times |\J|}$ as a matrix with columns consisting of $\nabla G_{mn}(\bfx)$ with $(m,n)\in \J$, namely,
$$\nabla_\J \G(\bfx) : =\{\nabla G_{mn}(\bfx):~(m,n)\in \J \}\in\R^{K\times |\J|}.$$
For notational simplicity, we write
\begin{eqnarray*}
\begin{array}{lll}
\sum G_{mn}:=\sum_{(m,n)\in\M\times\N}G_{mn},\qquad \sum_\J G_{mn}:=\sum_{(m,n)\in\J}G_{mn}.
\end{array}
\end{eqnarray*}
Let $\varphi:\R^n\to\R^m$ be  a locally Lipschitz continuous function. According to Rademacher's Theorem, $\varphi$ is differentiable almost everywhere. Denote the set of points at which $\varphi$ is differentiable by $D_\varphi$. 
 Then the  Clarke generalized Jacobian \cite{clarke1990optimization} of $\varphi$ at $\bfx\in\R^n$ is
 $$
 \begin{array}{l}
 \partial \varphi(\bfx)={\rm co}\{\lim_{\bfx^k\in D_\varphi, \bfx^k\to\bfx} \nabla\varphi (\bfx^k)\},
 \end{array}$$
 where ${\rm co}(\Omega)$ stands for the convex hall of set $\Omega$. 
Let ${\cal O} \subseteq\R^n$ be an open set  and $\varphi:{\cal O} \to\R^m$ be a locally Lipschitz
continuous function. We say that $\varphi$ is semismooth at $\bfx\in{\cal O}$ if it is directionally differentiable at $\bfx$ and   for any  $\triangle\bfx$ with $\triangle\bfx\to0$, and   $\H\in \partial \varphi(\bfx+\triangle\bfx)$,
\begin{equation}
\label{def-semismooth}
\varphi(\bfx+\triangle\bfx)- \varphi (\bfx)- \H \triangle\bfx=o(\|\triangle\bfx\|).
\end{equation}
Furthermore, if the above equation is replaced by
\begin{equation}
\label{def-strong-semismooth}
\varphi(\bfx+\triangle\bfx)- \varphi (\bfx)- \H \triangle\bfx=O(\|\triangle\bfx\|^2),
\end{equation}
then $\varphi$ is said to be  strongly semismooth at $\bfx$. 

 \section{Statistical Properties}\label{sec:sta}
In this section, we shall see how to set $s$ to guarantee an optimal solution to problem \eqref{CCP}  to be feasible to problem \eqref{SCP} when \eqref{CCP} reduces to \eqref{SAAP} (namely when taking $G_{mn}(\bfx)=g_m(\bfx,\bfxi_n)$). To proceed with that, let 
\begin{equation*} 
\begin{array}{lll}
Y(\bfxi;\bfx) :=  - {\max}_{m\in\M} g_m(\bfx,\bfxi). 
\end{array}\end{equation*}
We then define the probability density and empirical distribution function of $Y(\bfxi;\bfx)$ by
\begin{equation*} 
\begin{array}{lll}
F(t;\bfx) := \mathbb{P}\{ Y(\bfxi;\bfx) < t\}, \qquad F_N(t;\bfx):=\frac{1}{N} |\{n\in\N: Y(\bfxi_n;\bfx)< t\}|.
\end{array}\end{equation*}
Consequently, we have
\begin{equation} \label{F0-FN0}
\eqspace{1.35}
\begin{array}{lll}
F(0;\bfx) &=&  \mathbb{P}\{ {\max}_{m\in\M}  g_m(\bfx,\bfxi) > 0 \}, \\
 F_N(0;\bfx)&=&\frac{1}{N}\sum_{n=1}^N\ell_{0/1} ( {\max}_{m\in\M}  g_m(\bfx,\bfxi_n)  ) = \frac{1}{N} \|\G(\bfx)\|_0^+.
\end{array}\end{equation}
Let $\bfx^*$ be an  optimal solution to (\ref{CCP}) and define 
$$\alpha_*:=\mathbb{P}\{ {\max}_{m\in\M}  g_m(\bfx^*,\bfxi)  >0\} = F(0;\bfx^*).$$
Using these notation, we derive the following results.
\begin{theorem}\label{s-lower-bd}  Suppose that an  optimal solution  $\bfx^*$ to (\ref{CCP})  satisfies  $\alpha_*>0$. Then 
 \begin{equation} \label{away-asy-n}
\mathbb{P} \left\{ \|\G(\bfx^*)\|_0^+ > (1-\sqrt{\nu})\alpha_* N  \right\} \geq 1- 2 {\rm exp}(-2 \nu \alpha_*^2 N),
\end{equation}
for some $\nu\in(0,1)$. If   $s  \geq (\sqrt{\nu}+1)\alpha_* N$, then  
 \begin{equation} \label{feasible-asy-n}
\mathbb{P} \left\{  \|\G(\bfx^*)\|_0^+ \leq s \right\} \geq 1- 2 {\rm exp}(-2   \nu \alpha_*^2 N).
\end{equation}
\end{theorem}
\proof{Proof} 
It follows from \cite[Theorem 1]{massart1990tight} that
\begin{equation} \label{F0-FN-UPP-BD}
\begin{array}{lll}
\mathbb{P} \{ \sqrt{N } \sup_t |F_N(t;\bfx^*)  -F(t;\bfx^*)|\geq \lambda  \} \leq 2{\rm exp}(-2\lambda^2)
\end{array}\end{equation}
for any given $\lambda >0$, which by letting $\lambda= \alpha_* \sqrt{\nu N}$ with $\nu\in(0,1)$  allows us to  obtain
\begin{eqnarray*} 
\eqspace{1.35}
\begin{array}{lll}
\mathbb{P} \{ \|\G(\bfx^*)\|_0^+ >  (1-\sqrt{\nu})\alpha_* N \} &=&
\mathbb{P} \{F_N(0;\bfx^*)   > (1-\sqrt{\nu})\alpha_*  \}\\
 &=&
\mathbb{P} \{ \sqrt{N} (F(0;\bfx^*) - F_N(0;\bfx^*) )< \alpha_* \sqrt{\nu N} \}   \\
 &=&
1-\mathbb{P} \{ \sqrt{N} (F(0;\bfx^*) - F_N(0;\bfx^*) ) \geq \alpha_* \sqrt{\nu N} \}   \\
&\geq&1-2 {\rm exp}(-2  \nu \alpha_*^2 N),
\end{array}\end{eqnarray*}
where the  first two equalities are from \eqref{F0-FN0} and $\alpha_*=F(0;\bfx^*)$, and the last inequality holds because of   \eqref{F0-FN-UPP-BD}.  Similar reasoning yields that
\begin{eqnarray*} 
\eqspace{1.35}
\begin{array}{lll}
\mathbb{P} \{ \|\G(\bfx^*)\|_0^+ > s \}&=&
\mathbb{P} \{ \|\G(\bfx^*)\|_0^+ \geq s+1 \}\\
 &=&
\mathbb{P} \{ F_N(0;\bfx^*) \geq (s+1)/N \}   \\
&\leq&\mathbb{P} \{ F_N(0;\bfx^*) \geq  (\sqrt{\nu}+1) \alpha_* \} \\
&=&\mathbb{P} \{ \sqrt{N}   (F(0;\bfx^*) - F_N(0;\bfx^*)  ) \geq    \alpha_*\sqrt{\nu N} \} \\
&\leq&2 {\rm exp}(-2 \nu \alpha_*^2 N),
\end{array}\end{eqnarray*}
where  the first inequality is from  $s  \geq (\sqrt{\nu}+1)\alpha_* N $. The proof is finished.
\Halmos\endproof
\begin{remark}From Theorem \ref{s-lower-bd}, when $\alpha_*>0$ for an optimal solution $\bfx^*$ to \eqref{CCP}, it holds $\|\G(\bfx^*)\|_0^+ >  (1-\sqrt{\nu}) \alpha_* N>0$ in high probability, which indicates that $G_{mn}(\bfx^*)=g_m(\bfx^*,\bfxi_n) \leq 0$ may not be satisfied for any $m\in\M$ and $n\in\N$, thereby $\bfx^*$ is unlikely feasible to problem \eqref{SAS}. In addition, if we set $s$ to be away from $0$,   optimal solution $\bfx^*$ could be feasible to \eqref{SCP}. In this regard, solving \eqref{SCP} with a positive $s$ admits its advantage in contrast to solving \eqref{SAS}. 

It is noted that similar results have been well established in \cite{luedtke2008sample,pagnoncelli2009sample}. For instance,  \cite[Theorem 3]{luedtke2008sample} and its proof respectively show that $ \widehat{f}\leq f(\bfx^*)$  and $\|\G(\bfx^*)\|_0^+ \leq s$ hold with   probability ${1- {\rm exp}(-2  (s/N-\alpha)^2 N)}$ if $s/N>\alpha $, where $\widehat{f}$ is the optimal function value of the SAA problem.   However, \eqref{feasible-asy-n} can also allow us to conclude that $ \widehat{f}\leq f(\bfx^*)$ holds with high probability due to $\mathbb{P} \{ \widehat{f}\leq f(\bfx^*)\} \geq \mathbb{P} \{  \|\G(\bfx^*)\|_0^+ \leq s \}$. Differing from \cite[Theorem 3]{luedtke2008sample},  the probability bound in Theorem \ref{s-lower-bd} is established based on $\alpha_*$ (which depends on optimal solution $\bfx^*$) rather than $\alpha$. Moreover, Theorem \ref{s-lower-bd} also provides a positive lower bound of $\|\G(\bfx^*)\|_0^+$. 

 \end{remark}

\section{Properties of Sets $\S$ and $\CF$}\label{sec:sets}
In this section, we direct our attention to feasible set $\CF$ in \eqref{SCP} and $\S$ in \eqref{l01-s}, with the aim of deriving the projection of a point onto $\S$, as well as determining tangent and normal cones of these two sets. Before proceeding, we introduce a key set that remains significant throughout the paper.  Let $\Gamma_+, \Gamma_0$, and $\Gamma_-$ be given by \eqref{notation-z}. We order indices in $\Gamma_+ :=\{i_1,i_2,\ldots,i_{|\Gamma_+|}\}$ as follows,
\begin{eqnarray} 
\begin{array}{l}
 \|(\Z_{:i_1})^+\| \geq  \|(\Z_{:i_2})^+\| \geq \ldots \geq  \|(\Z_{:i_{r}})^+\| \geq  \ldots \geq  \|(\Z_{:i_{|\Gamma_+|}})^+\|>0.
\end{array} 
\end{eqnarray} 
Here, $ r:=\min\{s, |\Gamma_+|\}$. We note that $\|(\Z_{:n})^+\|=0$ for any $n\in\Gamma_0\cup\Gamma_-$, so the definition of $\Z_r^{\downarrow}$ implies that $\|(\Z_{:i_{r}})^+\|=\Z_r^{\downarrow}$.   Based on this, we define the following set
\begin{eqnarray}
\label{T-z}
\T(\Z;s):=\Big\{(\Gamma_+\setminus\Gamma_{s})\cup\Gamma_0 :~ 
  \Gamma_{s}\subseteq  \Gamma_+,~ |\Gamma_{s}|=r,~\min\limits_{n\in \Gamma_{s}} \|(\Z_{:n})^+\|\geq \max\limits_{n\notin \Gamma_{s}}\|(\Z_{:n})^+\|
\Big\}.~~~
\end{eqnarray}
Note that $\Gamma_{s}$ contains  $r=\min\{s, |\Gamma_+|\}$ indices corresponding to the first $r$ largest values in $\{\|(\Z_{:i})^+\|: i\in\Gamma_+\}$, i.e., $\Gamma_{s}=\{i_1,i_2,\ldots,i_{r}\}$, and it may not be unique. For any $T\in\T({\Z};s)$, it has
\beq\label{z0+-Gamma+} 
\eqspace{1.35} 
\begin{array}{ll}
\overline T={\N}\setminus T={\Gamma_{s} \cup  \Gamma_-}.
\end{array}
\eeq 
In addition, if $\|\Z\|^+_0\leq s$, then $ \|\Z\|^+_0 =|\Gamma_+| = r$ and thus $\Gamma_{s}=\Gamma_+$, thereby $\T(\Z;s)=\{\Gamma_0\}$ and $\overline T=\Gamma_{+}\cup \Gamma_-$.  
Taking the following instance as an example,
\begin{eqnarray} \label{point-Z}
\Z=\left[\begin{array}{rrrr}
2&2&0&-1\\ 
0&-1&-2&-3   
\end{array}\right],\qquad M=2, \qquad N=4.\end{eqnarray}
One can check that $\Gamma_{+}=\{1,2\}, \Gamma_0=\{3\}, \Gamma_-=\{4\}$, and thus $\|\Z\|^+_0=|\Gamma_{+}|=2$. If $s\geq 2$, then $\Gamma_s=\{1,2\}$ and  $\T(\Z;s)=\{\{3\}\}$. If $s=1$, then $\Gamma_s=\{1\}$ or $ \{2\}$  and  $\T({\Z};1)=\{\{1,3\},\{2,3\}\}$.

\subsection{Projection}
 Projection $\bpi_{\rmo}({\Z})$ of ${\Z}$ onto a nonempty and closed set $\rmo$ is defined by
\begin{eqnarray}
\bpi_\rmo({\Z}) = {\rm argmin}_{\W\in{\rmo}} ~   \|{\Z}-\W\|.
\end{eqnarray}
It is well known that the solution set of the right-hand side problem is a singleton when ${\rmo}$ is convex and might have multiple elements otherwise.  The following property shows that the projection onto $\S$ has a closed form.
\begin{proposition}\label{pro-1} Define $\T({\Z};s)$ as (\ref{T-z}). Then
\begin{eqnarray}\label{psz}
\bpi_\S({\Z})=\Big\{\left[ ({\Z}_{:T})^-~~{\Z}_{: \overline T  } \right]: T\in\T({\Z};s) \Big\}.
\end{eqnarray}
\end{proposition}
\proof{Proof} 
Denote $\Phi$ the set of the right-hand side of \eqref{psz}. If $\Z\in\S$, then $\bpi_\S({\Z})=\{\Z\}$ and $T=\Gamma_0$ for any $T\in\T({\Z};s)$ which implies $({\Z}_{:T})^-={\Z}_{:T}$, leading to $\Phi=\{\Z\}$. As a result, $\bpi_\S({\Z})=\Phi$.  If $\Z\notin\S$, then $|\Gamma_+|>s$. One can observe that $\N=\Gamma_+\cup\Gamma_0\cup\Gamma_{-}$.  By the definitions in \eqref{notation-z} of $\Gamma_0$ and $\Gamma_{-}$, for any  $\Z^*\in\bpi_\S({\Z}) = {\rm argmin}_{\|\W\|_0^+\leq s } ~  \|{\Z}-\W\|$, we must have $Z^*_{:n}=Z_{:n}$ for any $n\in\Gamma_0\cup\Gamma_{-}$ because $\|\Z\|_0^+$ only counts the number of positive values in $\{\Z_{:1}^{\max},\ldots,\Z_{:N}^{\max}\}$. Therefore, to preserve the feasibility while minimizing $ \|{\Z}-\W\|$, $\Z^*$ should keep $s$ columns of ${\Z}$ but set the remaining to be non-positive. In other words, we need to pick an index set $\Gamma\subseteq\Gamma_+$ with $|\Gamma|=s$ such that
\begin{eqnarray*} 
 \Z_{:n}^*= \left\{
\eqspace{1.35}\begin{array}{ll}
 \Z_{:n}  ,&n\in\Gamma ,\\ 
 (\Z_{:n})^-,&n\in\Gamma_+\setminus\Gamma,\\
  \Z_{:n},&n\in\Gamma_0\cup\Gamma_{-}.
\end{array}
\right.
\end{eqnarray*}
By the definition of $\Gamma_s$ in \eqref{T-z}, the best choice of $\Gamma$ is  $\Gamma_{s}$. Then it follows from $({\Z}_{:\Gamma_0})^-={\Z}_{:\Gamma_0}$ that
\begin{eqnarray*} 
 \Z_{:n}^*= \left\{
\eqspace{1.35}\begin{array}{ll}
 (\Z_{:n})^-,&n\in(\Gamma_+\setminus\Gamma_s)\cup\Gamma_0,\\
  \Z_{:n},&n\in\Gamma_s\cup\Gamma_{-},
\end{array}
\right.
\end{eqnarray*}
and thus $\Z^*\in\Phi$, showing $\bpi_\S({\Z})\subseteq\Phi$. The above arguments can also show that $\bpi_\S({\Z})\supseteq\Phi$, which finishes the proof.\Halmos
\endproof
 We provide an example to illustrate \eqref{psz}. Again consider an example in \eqref{point-Z}. For $s \geq2$, ${\bpi_\S}({\Z})=\{{\Z}\}$ due to $\T({\Z};s)=\{\{3\}\}$ and $({\Z}_{:3})^-={\Z}_{:3}$. For $s=1$,  $\T({\Z};1)=\{\{1,3\},\{2,3\}\}$ and thus 
\begin{eqnarray*}{\bpi_\S}({\Z})=\left\{\left[\begin{array}{rrrr}
0&2&0&-1\\ 
0&-1&-2&-3  
\end{array}\right],\left[\begin{array}{rrrr}
2&0&0&-1\\ 
0&-1&-2&-3  
\end{array}\right]\right\}.\end{eqnarray*}
With the help of the closed form of   ${\bpi_\S}({\Z})$, we establish the following fixed point inclusion.
\begin{proposition}\label{pro-eta} Given  $\W \in\R^{M\times N}$ and $\tau >0$, a point ${\Z}\in\R^{M\times N}$  satisfies
\begin{eqnarray}\label{uPu}
{\Z} \in {\bpi_\S}({{\Z}}+\tau \W )
\end{eqnarray}
if and only if it satisfies  
\begin{eqnarray}\label{uPu-equ}
\left\{\begin{array}{ll}
\W ={\bf0},&\text{if}~\|{\Z}\|^+_0 < s,\\[1ex]
\W _{:\overline \Gamma_0}={\bf0} ,~~ {\bf0}\geq {\Z} _{:\Gamma_0} \perp \W _{:\Gamma_0} \geq {\bf0},~~  \tau\|\W _{:n}\|\leq  \zs , ~\forall n\in {\Gamma_0},&\text{if}~\|{\Z}\|^+_0 = s.
\end{array}\right.
\end{eqnarray} 
 Moreover, if a point ${\Z}$ satisfies (\ref{uPu}), then $ \T({{\Z} }+\tau \W ; s)\ni\{\Gamma_0\}$. 
\end{proposition}
\proof{Proof}
It is easy to verify the following fact: for any $n\in{\N}$,
\begin{eqnarray}\label{eq-1}
\eqspace{1.35}
\begin{array}{rrl}
{\bf0}~\geq~{\Z}_{:n}~\bot~\W _{:n}~\geq~ {\bf0}
&\Longleftrightarrow& ~({\Z}_{:n}+\tau \W _{:n})^-~~={\Z}_{:n}\\ 
&\Longrightarrow& \|({\Z}_{:n}+\tau \W_{:n})^+\|=\tau\|\W _{:n}\|.
\end{array}\end{eqnarray}
\underline{Sufficiency.}  Relation \eqref{uPu} is clearly true when $\|{\Z} \|^+_0 < s$ as $\W ={\bf0}$. When $\|{\Z} \|^+_0 = s$, we have
\begin{eqnarray} \label{y-tau-lam-0}
\eqspace{1.35}
\begin{array}{lll}
&& \left[
{\Z} _{:\Gamma_0}+\tau \W_{:\Gamma_0},~
{\Z} _{:\Gamma_+}+\tau \W_{:\Gamma_+},~
{\Z} _{:\Gamma_-}+\tau \W_{:\Gamma_-} 
\right] \\
& =&   {{\Z} }+\tau \W 
 \overset{\eqref{uPu-equ}}{=} \left[
{\Z} _{:\Gamma_0}+\tau \W_{:\Gamma_0},~
{\Z} _{:\Gamma_+},~
{\Z} _{:\Gamma_-}
\right],\end{array}
\end{eqnarray}
where $\Gamma_0,\Gamma_+,$ and $\Gamma_-$ are defined for ${\Z}$ as \eqref{notation-z}. Therefore,
\begin{eqnarray} \label{y-tau-lam}
({\Z}_{:n}+\tau \W_{:n})^{\max}={\Z}_{:n}^{\max}\left\{
\eqspace{1.35}\begin{array}{ll}
>0,&n\in\Gamma_+,\\ 
< 0,&n\in\Gamma_-.
\end{array}
\right.
\end{eqnarray}
Moreover,   we have
\begin{eqnarray} \label{y-tau-lam-1}
\forall~ n\in {\Gamma_0},~~\|({\Z} _{:n}+\tau \W_{:n})^+\| \overset{(\ref{uPu-equ},\ref{eq-1})}{=}  \tau\| \W _{:n} \| \overset{\eqref{uPu-equ}}{\leq} \zs.
\end{eqnarray}
Since $\|{\Z} \|^+_0 = s$, there is $|\Gamma_+|=s$. This together with \eqref{y-tau-lam-0} and \eqref{y-tau-lam-1} suffices to $\zs = ({\Z} +\tau \W   )^\downarrow_s$, which by   \eqref{y-tau-lam} and \eqref{y-tau-lam-1} implies ${\Gamma_0}\in \T({{\Z} }+\tau \W ; s)$. Then it follows from Proposition  \ref{pro-1} that
\begin{eqnarray*}\eqspace{1.35}
\begin{array}{lcl}
{\bpi_\S}({{\Z} }+\tau \W )&\ni&\left[
({\Z} _{:\Gamma_0}+\tau \W_{:\Gamma_0})^-,~
{\Z} _{ :\overline\Gamma_0}+\tau \W_{ :\overline\Gamma_0}
\right]\\
&\overset{(\ref{uPu-equ},\ref{eq-1})}{=}&\left[
 {\Z} _{:\Gamma_0},~
{\Z} _{:\overline \Gamma_0}
\right]={\Z} .\end{array}\end{eqnarray*}
\underline{Necessity.} It follows from \eqref{uPu} that $\|{\Z} \|^+_0 \leq s$. We claim the conclusion by two cases.
 \begin{itemize}[leftmargin=12pt]
\item {Case 1: $\|{\Z} \|^+_0< s$.}  Condition \eqref{uPu}  implies $\|{\Z} +\tau \W  \|^+_0 < s$,  leading to
\begin{eqnarray*}
{\Z}\in {\bpi_\S}({\Z} +\tau \W  ) = \{{\Z} +\tau \W  \},
\end{eqnarray*}
deriving $\W ={\bf0}$ and  hence $\T({{\Z} }+\tau \W ; s)=\T({{\Z} }; s)=\{\Gamma_0\}$ due to $\|{\Z} \|^+_0< s$. 
\item  {Case 2: $\|{\Z} \|^+_0 =s$.}
Let $\Gamma'_+, \Gamma'_-$,  and $ \Gamma'_0$ be defined as \eqref{notation-z}, and  $\Gamma'_{s}$ be defined as in \eqref{T-z}  but for ${\Z} +\tau \W  $. For any  $T \in\T({\Z} +\tau \W  ;s)$, we have
\beq\label{overline-T}\overline T  \overset{\eqref{z0+-Gamma+}}{=}{\Gamma'_{s} \cup  \Gamma'_-}.\eeq
 It follows from Proposition  \ref{pro-1} that there is a set $T \in\T({\Z} +\tau \W  ;s)$ such that
\begin{eqnarray}\label{fact-z-z-w}
 \eqspace{1.35}
\begin{array}{lcl}
[ 
{\Z} _{:T  },~
{\Z} _{:\Gamma'_{s}},~
{\Z} _{:\Gamma'_{-}} ]&\overset{\eqref{overline-T}}{=}&\left[ 
{\Z} _{:T  },~
{\Z} _{:\overline  T  } 
\right]={\Z}\\
& \overset{\eqref{psz}}{=}&\left[ 
({\Z} _{:T  }+\tau \W_{:T  })^-,~
{\Z} _{:\overline  T  }+\tau \W_{:\overline  T  } 
\right]\\
&\overset{\eqref{overline-T}}{=}& [ 
({\Z} _{:T  }+\tau \W_{:T  })^-,~
{\Z} _{:\Gamma'_{s}}+\tau \W_{:\Gamma'_{s}},~
{\Z} _{:\Gamma'_{-}}+\tau \W_{:\Gamma'_{-}} 
 ],\end{array} 
\end{eqnarray}
 which by \eqref{eq-1} suffices to
\beq\label{pro-z-T}
\begin{array}{lll}
\W _{:\overline{T }}&=&{\bf0},~ {\bf0} \geq  {\Z} _{:T}  \perp  \W _{:T}  \geq {\bf0}.
\end{array} \eeq
Conditions \eqref{fact-z-z-w} and the definitions of $\Gamma'_{s}$ and $\Gamma'_{-}$ enable us to obtain
\beq\label{zimax}
\mz_{:n}  \left\{
\begin{array}{lll}
\leq 0, ~&n\in  T, \\
>0,~~& n\in \Gamma'_{s}, \\
<0,& n\in \Gamma'_{-}.
\end{array}\right.
\eeq
Recalling $\Gamma_0=\{n\in{\N}: \mz_{:n}=0\} $, it follows 
\beq\label{Gamma0-T}
\Gamma_0\subseteq T\in \T({\Z} +\tau \W  ;s). 
\eeq
We next prove $:\Gamma=T \setminus \Gamma_0   = \emptyset$. In fact, \eqref{zimax} means $\Gamma \subseteq \{n\in{\N}: \mz_{:n}<0\}$, resulting in ${\Z}_{:\Gamma}<{\bf0}$ and thus $\W _{:\Gamma}={\bf0}$ from ${\Z} _{:\Gamma}  \perp  \W _{:\Gamma}$ by \eqref{pro-z-T}. Therefore, $(\Z+\tau \W)^{\max} _{:n}= \mz_{:n} <0 $ for any $n\in \Gamma$. However, $\Gamma\subseteq T\in \T({\Z} +\tau \W  ;s)$ manifests $(\Z+\tau \W)^{\max} _{:n}\geq 0$ for any $n\in \Gamma$ from \eqref{T-z}. This contradiction shows $\Gamma = \emptyset$, which by \eqref{Gamma0-T} delivers  $\Gamma_0= T$, thereby  $\T({\Z} +\tau \W  ;s)\ni\{\Gamma_0\}$.   Overall, 
\begin{eqnarray}\label{mz-w}
\eqspace{1.35}
\mz_{:n}  \left\{
\begin{array}{lll}
= 0, &n\in  \Gamma_0, \\
>0,& n\in \Gamma'_{s}, \\
<0,& n\in \Gamma'_{-}, \\
\end{array}\right. ~~\text{and}~~\begin{array}{lcl} 
\W   &=&[
 \W _{:\Gamma_0},~
\W_{:\overline \Gamma_0} 
 ] \\
 &=&\left[ 
 \W _{:\Gamma_0},~
\W_{:\overline T} 
\right]  \\
& \overset{\eqref{pro-z-T}}{=}&\left[
 \W _{:\Gamma_0},~
{\bf0} 
\right]. \end{array} \end{eqnarray}
These conditions also indicate
 \beq\label{z-s-lam}
  \eqspace{1.35}
 \begin{array}{lclcl}
  \zs&\overset{\eqref{mz-w}}{=}&\min_{n'\in\Gamma'_s} \|({\Z} _{:n'})^+\|\\&\overset{\eqref{fact-z-z-w}}{=}& \min_{n'\in\Gamma'_s} \|({\Z} _{:n'}+\tau \W_{:n'})^+\|\\
 &\geq& \|({\Z} _{:n}+\tau \W_{:n})^+\| \\&{\overset{(\ref{pro-z-T},\ref{eq-1})}{=}}&\tau\|\W _{:n}\|, ~~\forall~n\in  \Gamma_0,
 \end{array}  \eeq
where `$\geq$' is due to the definition of $\{\Gamma_0\} = \T({\Z} +\tau \W  ;s)$  in \eqref{z0+-Gamma+}.   Finally, \eqref{mz-w}, \eqref{pro-z-T}, and \eqref{z-s-lam} enable us to conclude  \eqref{uPu-equ}. \Halmos  
 \end{itemize}
\endproof

\subsection{Tangent and Normal cones}
For a nonempty and closed  set $\rmo\subseteq \mathbb{R}^{K}$, its  Bouligand tangent cone  $\rT_{\rmo}(\bfx)$  and  Fr$\acute{e}$chet normal cone  $\widehat{\rN}_{\rmo}(\bfx)$  at point $\bfx\in\rmo$ are defined as \cite{RW1998}:
\begin{eqnarray}\label{Tangent Cone}
\rT _{\rmo}(\bfx)&:=&  \left\{~\bfd\in\mathbb{R}^{K}:
 \begin{array}{r}
\exists~t_{\ell}\geq0,   \bfx^{\ell}\overset{\rmo}{\rightarrow}\bfx ~\text{such~that}~t_{\ell}(\bfx^{\ell}-\bfx)\rightarrow\bfd
\end{array}\right\},~~~~~~\\[1ex]
\label{Normal Cone} \widehat{\rN}_{\rmo}(\bfx)&:=& \Big\{~\bfu\in\mathbb{R}^{K}:~\langle \bfu, \bfd\rangle\leq0,~\forall~\bfd\in \rT_{\rmo}(\bfx)~\Big\},
\end{eqnarray}
where $\bfx^{\ell} {\rightarrow}\bfx$ represents $\lim_{\ell\rightarrow\infty}\bfx^{\ell}=\bfx$ and $\bfx^{\ell}\overset{\rmo}{\rightarrow}\bfx$ stands for $\bfx^{\ell}\in\rmo$ for {every $\ell$} and $ \bfx^{\ell} \rightarrow\bfx$. Let ${\rmo}_1\cup\ldots\cup\rmo_N$ be the union of finitely many nonempty and closed subsets ${\rmo}_n$.
Then by \cite[Proposition 3.1]{ban2011Lipschitzian}, for any $\bfx\in {\rmo}_1\cup\ldots\cup\rmo_N$ we have 
\begin{eqnarray}\label{cup-tangent}\begin{array}{ccl}
  \rT_{{\rmo}_1\cup\ldots\cup\rmo_N}(\bfx)=\bigcup_{n : \bfx\in {\rmo}_n}\rT_{{\rmo}_n}(\bfx). 
\end{array}\end{eqnarray} 
Note that set $\CF$ can be rewritten as
\begin{eqnarray*}\begin{array}{ccl}
  \CF=\bigcup_{\Gamma\in\P ({\M},s)}
  \Big\{\bfx\in \mathbb{R}^{K}: ~(\G (\bfx))_{:n}^{\max}\leq 0,~ n\in \overline{\Gamma}\Big\},\end{array}
\end{eqnarray*}
 where $\P(\cdot,\cdot)$ is defined as
\begin{eqnarray}\label{index-u}
\begin{array}{ccl}
\P(\Gamma,s):=\{  \Gamma' \subseteq \Gamma:~ |\Gamma'|\leq s\}.
\end{array}
\end{eqnarray}
In the subsequent analysis, we let
$$\Z:=\G(\bfx),$$
and corresponding index sets $\Gamma_+, \Gamma_0, \Gamma_-$, and $V_{\Gamma}$ in \eqref{notation-z} are defined for $\G(\bfx)$. Therefore, we have
\begin{eqnarray}\label{V-Gamma0}
 \eqspace{1.35}
\begin{array}{ccl}
\Gamma_+&=&\left\{n\in{\N} :~(\G (\bfx))_{:n}^{\max}>0\right\},\\
\Gamma_-&=&\left\{n\in{\N} :~(\G (\bfx))_{:n}^{\max}<0\right\},\\
\Gamma_0&=&\left\{n\in{\N} :~(\G (\bfx))_{:n}^{\max}=0\right\},\\
V_{\Gamma}&=&\left\{(m,n)\in\M\times\Gamma:~{G}_{mn}(\bfx) = 0\right\}. 
\end{array}
\end{eqnarray}
Based on set $\P(\Gamma_0, s- |\Gamma_+|)$ and the above notation, we derive the Bouligand tangent cones and corresponding  Fr$\acute{e}$chet normal cone of $\CF$ explicitly by the following theorem.
\begin{proposition}
  \label{tangent-S}
Suppose $\nabla_{V_{\Gamma_0}} \G (\bfx)$ is full column rank. Then Bouligand tangent cone $\rT_{\CF}(\bfx)$  and  Fr$\acute{e}$chet normal cone
$\widehat{\rN}_{\CF}(\bfx)$ at $\bfx\in \CF$ are given by
\begin{eqnarray}
\label{TSG}&\rT_{\CF}(\bfx)~{=}&  
{\bigcup}_{\Gamma \in \P(\Gamma_0,s-|\Gamma_+|)}\left\{\bfd\in\mathbb{R}^{K}:~
 \langle \bfd, \nabla G_{mn} (\bfx) \rangle  \leq { 0 },~ (m,n)\in V_{\Gamma_0\setminus\Gamma } \right\},\\[1ex]
\label{NSG-exp}
&\widehat{\rN}_{\CF}(\bfx)~{=}&
\left\{\eqspace{1.35}
\begin{array}{lll}
\left\{\sum_{V_{\Gamma_0}} W_{mn} \nabla  G_{mn}(\bfx) :~
W_{mn}\geq 0, (m,n)\in V_{\Gamma_0} 
\right\},~& \mathrm{if}~\|\G(\bfx)\|^+_0=s,   \\
\hspace{3cm}\{ {\bf0} \}, &\mathrm{if}~\|\G(\bfx)\|^+_0<s.
\end{array}
\right. ~~~~
\end{eqnarray}
\end{proposition} 
\proof{Proof}   For any fixed $\bfx\in \CF$, it follows from \eqref{V-Gamma0} that
\begin{eqnarray*}\label{SGGamma}
\bfx\in  \CF_{\Gamma}:=\left\{{\bf z}\in\mathbb{R}^{K}:
  (\G({\bf z}))_{:n}^{\max}\left\{\begin{array}{ll}
  >0,&n\in \Gamma_+,\\
  <0,& n\in \Gamma_-,\\
  \leq 0,&    n\in\Gamma_0\setminus\Gamma
  \end{array}
  \right.
\right\}, ~~\forall~ \Gamma\in\P:=\P(\Gamma_0,s-|\Gamma_+|).
\end{eqnarray*}
Since $\left(\bigcup_{\Gamma \in \P} \CF_{\Gamma}\right)\subseteq \CF$, the above condition and \eqref{cup-tangent} yield 
 \begin{eqnarray}\label{x-SGGamma}
\rT_{ \CF }(\bfx) = \rT_{\left(\bigcup_{\Gamma \in \P} \CF_{\Gamma}\right)}(\bfx).
\end{eqnarray}
We note that the active index set of $\CF_{\Gamma}$ at $\bfx$ is $$  V_{\Gamma_0\setminus\Gamma}=\{(m,n)\in{\M} \times(\Gamma_0\setminus\Gamma): G_{mn}(\bfx)=0\} \subseteq V_{\Gamma_0}.$$
 Since $\nabla_{V_{\Gamma_0}} \G(\bfx)$  is full column rank, so is $\nabla_{V_{\Gamma_0\setminus\Gamma}} \G(\bfx)$.  Then by \cite[Lemma 12.2]{nocedal2006quadratic}, for any $\Gamma \in \P $,
the Bouligand tangent cone of $\CF_{\Gamma}$ at $\bfx$ is
\begin{eqnarray} \label{TSG-Gamma}
  \rT_{\CF_{\Gamma}}(\bfx)=\left\{\bfd\in\mathbb{R}^{K}:~\langle \bfd, \nabla G_{mn} (\bfx) \rangle  \leq { 0 }, ~(m,n)\in V_{\Gamma_0\setminus\Gamma } \right\}.
\end{eqnarray}
This together with \eqref{x-SGGamma} and \eqref{cup-tangent} shows \eqref{TSG}. 

 Next, we calculate Fr$\acute{e}$chet normal cone $\widehat{\rN}_{\CF}(\bfx)$ of $\bfx\in \CF$  by two cases.
 \begin{itemize}[leftmargin=12pt]
\item  {Case 1: $\|\G(\bfx)\|_0^+=s$.} By \eqref{0+norm}, we have $|\Gamma_+|=s$ and thus  $\P =\P(\Gamma_0,s-|\Gamma_+|)=\{\emptyset\}$, yielding 
\begin{eqnarray*}
 \rT_{\CF}(\bfx) \overset{\eqref{TSG}}{=}  \left\{\bfd\in\mathbb{R}^{K}:~ \langle \bfd, \nabla G_{mn} (\bfx) \rangle  \leq { 0 }, (m,n)\in V_{\Gamma_0} \right\}.
\end{eqnarray*}
Then direct verification by definition \eqref{Normal Cone} allows us to derive \eqref{NSG-exp}.\\
\item  {Case 2: $\|\G(\bfx)\|_0^+<s$.}
If $|\Gamma_0|\leq s-|\Gamma_+|$, then $\Gamma_0\in \P $.
From \eqref{TSG}, we have $\rT_{\CF}(\bfx)=\R^{K}$ and then $\widehat{\rN}_{\CF}(\bfx)=\{{\bf0}\}.$ If  $|\Gamma_0|> s-|\Gamma_+|$, then $|\Gamma_0|\geq 2$ due to $|\Gamma_+|=\|\G(\bfx)\|_0^+<s$. Let 
\begin{eqnarray}\label{def-Gamma-ell}
\Gamma_0:=\{t_1,t_2,\ldots,t_{|\Gamma_0|}\},\qquad\Gamma_{\ell}:= \{t_{\ell}\}, ~\ell=1,2,\ldots,|\Gamma_0|.
\end{eqnarray} 
In addition,  it follows from \eqref{TSG-Gamma}  and \eqref{TSG} that $\rT_{\CF}(\bfx)=\bigcup_{\Gamma\in \P}\rT_{\CF_{\Gamma}}(\bfx)$, thereby  $\rT_{\CF_{\Gamma_{\ell}}}(\bfx)\subseteq \rT_{\CF}(\bfx)$ due to $\Gamma_{\ell}\in\P.$ We note from \eqref{TSG-Gamma} that for any $\bfd\in \rT_{\CF_{\Gamma}}(\bfx)$, any vector $\bfu$ satisfying $\langle \bfu, \bfd\rangle \leq 0$ takes the form of
\begin{eqnarray}\label{u-express}
\begin{array}{lll}
\bfu= \sum W_{mn}  \nabla  G_{mn}(\bfx),\qquad  W_{mn} \left\{ \begin{array}{lll}
\geq 0,~&(m,n)\in V_{\Gamma_0\setminus\Gamma},\\[1ex]
= 0, ~&(m,n)\notin V_{\Gamma_0\setminus\Gamma}.
\end{array}\right.\end{array}
\end{eqnarray}
To show \eqref{NSG-exp}, it suffices to show $\bfu={\bf0}$. By considering $\bfd^{\ell}\in \rT_{\CF_{\Gamma_{\ell}}}(\bfx), \ell=1,\ldots,|\Gamma_0|$ and any $\bfu\in \widehat{\rN}_{\CF}(\bfx)$, we have $\bfd^{\ell}\in \rT_{\CF}(\bfx)$ and $\langle \bfu, \bfd^{\ell}\rangle \leq 0$ for each $\ell$. This and \eqref{u-express} suffice to
\begin{eqnarray} \label{u-Gamma1-Gammak}
\begin{array}{lll}
&&\bfu =  \sum W^\ell_{mn}  \nabla  G_{mn}(\bfx), \qquad  W^\ell_{mn} \left\{ \begin{array}{lll}
\geq 0,~&(m,n)\in V_{\Gamma_0\setminus\Gamma_{\ell}}=:D_\ell,\\[1ex]
= 0, ~&(m,n)\notin V_{\Gamma_0\setminus\Gamma_{\ell}}.
\end{array}\right.
\end{array}
\end{eqnarray}
for all $ \ell=1,2,\ldots, |\Gamma_0|,$ which results in
\begin{eqnarray*}
\eqspace{1.35}
\begin{array}{lll}
0&=&\sum W^1_{mn}  \nabla  G_{mn}(\bfx) -  \sum W^\ell_{mn}  \nabla  G_{mn}(\bfx)\\
&=& \sum_{D_1\cup D_\ell} W^1_{mn}  \nabla  G_{mn}(\bfx) - \sum_{D_1\cup D_\ell} W^\ell_{mn}  \nabla  G_{mn}(\bfx)\\
&=&  \sum_{ D_1 \cup D_\ell} (W^1_{mn} -     W^\ell_{mn})  \nabla  G_{mn}(\bfx) , ~~ \ell=2,3,\ldots, |\Gamma_0|.
\end{array}
\end{eqnarray*}
Since $ (D_1 \cup D_\ell)= V_{\Gamma_0}$ from \eqref{def-Gamma-ell}, the
 above condition and the full column rankness of $\nabla_{V_{\Gamma_0}} \G(\bfx)$ enable us to derive that $$W^1_{mn}= W^\ell_{mn}, ~(m,n)\in V_{\Gamma_0}$$
However, for each $\ell$, it has $W^\ell_{m n}=0$ for any $(m,n)\in V_{\Gamma_\ell}$ from \eqref{u-Gamma1-Gammak}, thereby $W^1_{mn}= 0$ for any $(m,n)\in V_{\Gamma_\ell}$. Overall,  $W^1_{mn}= 0$ for any $(m,n)\in V_{\Gamma_0},$ which by \eqref{u-Gamma1-Gammak} proves $\bfu={\bf0}$. \Halmos
 \end{itemize}
\endproof
\begin{remark} The Bouligand tangent  and  Fr$\acute{e}$chet normal cones of $\S$ at ${\Z}\in \S$ can be given as
\begin{eqnarray}
\label{TS-s}\rT_{\S}({\Z})&=&  {\bigcup}_{\Gamma \in \P(\Gamma_0,s-|\Gamma_+|)} \Big\{\W\in\mathbb{R}^{M\times N}:~W_{mn} \leq { 0 }, (m,n)\in V_{\Gamma_0\setminus\Gamma} 
  \Big\}, \\[1ex]
\label{NS-exp}
\widehat{\rN}_{\S}({\Z})&=&\left\{\eqspace{1.35}
\begin{array}{cl}
\left\{\W\in\mathbb{R}^{M\times N}:~ \begin{array}{l}
W_{mn} \geq { 0 }, (m,n)\in V_{\Gamma_0} \\
W_{mn} = { 0 }, (m,n)\notin V_{\Gamma_0} \\
\end{array}  
\right\}, ~~&\|{\Z}\|^+_0=s,~\\
 \{ {\bf0} \},&\|{\Z}\|^+_0<s.
\end{array}
\right.
\end{eqnarray}
 It should be noted that if $\nabla_{V_{\Gamma_0}} \G (\bfx)$ is full column rank, together with \eqref{NSG-exp} and \eqref{NS-exp},
the Fr$\acute{e}$chet normal cone of $\CF$ at $\bfx\in \CF\underline{}$ can be written as
\begin{eqnarray}\label{SG-Normal Cone}
\begin{array}{l}
 \widehat{\rN}_{\CF}(\bfx)=\left\{ \sum W_{mn}  \nabla  G_{mn}(\bfx) :~\W \in \widehat{\rN}_{\S}({\Z}),~
{\Z}=\G(\bfx)\right\}. 
\end{array}\end{eqnarray}
\end{remark}
Finally, we calculate the tangent and normal cones of $\S=\{{\Z}\in\R^{2\times 4}:~\|{\Z}\|^+_0\leq 2\}$ at
 \begin{eqnarray} \label{point-Z1}
{\Z}=\left[\begin{array}{rrrr}
2 & 2& 0& -1\\
0&-1&-2&-3 
\end{array}\right],~{\Z}'=\left[\begin{array}{rrrr}
2 & 0& 0& -1\\
0&-1&-2&-3 
\end{array}\right].\end{eqnarray}
It can be checked that
\begin{eqnarray*}\eqspace{1.35}
\begin{array}{llllll}
\rT_{S}({\Z})&=& \{\W\in\mathbb{R}^{2\times 4}:W_{13} \leq 0 \},\\
\widehat{\rN}_{S}({\Z}) &=&  \left\{\W \in\mathbb{R}^{2\times 4}:
W_{13}\geq 0,~W_{mn}=0, \forall (m,n)\neq (1,3)
 \right\},\\
 \rT_{S}({\Z}')&=& \{\W\in\mathbb{R}^{2\times 4}:W_{12} \leq 0 ~\text{or}~W_{13} \leq 0 \} ,\\
\widehat{\rN}_{S}({\Z}')&=&\left\{  {\bf0}  \right\}.
\end{array}
\end{eqnarray*}
\section{Optimality Analysis} \label{sec:opt} 
 In this section, we aim at establishing the first order necessary or sufficient optimality conditions of (\ref{SCP}). Hereafter, we always assume its feasible set  is non-empty if no additional information is provided and let 
   \begin{eqnarray}\label{z*-Gx*}
 \begin{array}{rlc}
{\Z}^*:=\G(\bfx^*). 
 \end{array} 
  \end{eqnarray}
We point out that ${\Z}^*$ depends on $\bfx^*$ and we drop this dependence for notational simplicity. Similar to \eqref{V-Gamma0}, we always denote $\Gamma_+^*, \Gamma_0^*, \Gamma_-^*$   for ${\Z}^*$ as well as 
\beq\label{gamma*}   \J_*:=V_{\Gamma^*_0}=\left\{(m,n)\in\M\times\Gamma^*_0:~Z_{mn}^* = 0\right\}.\eeq 
 
\subsection{KKT points}
We call  $\bfx^*\in\R^{K}$ a KKT point of   (\ref{SCP})  if there is a $\W ^*\in\R^{M\times N}$ such that 
  \begin{eqnarray}\label{KKT-point}
 \left\{\begin{array}{rlc}
  -\nabla f(\bfx^*)- \sum W^*_{mn}  \nabla G_{mn}(\bfx^*)  &\in&\widehat{\rN}_{\Omega}(\bfx^*),\\[1ex]
\W ^*&\in& \widehat{\rN}_\S(\Z ^*),\\[1ex]
\bfx^*&\in& \CF\cap\Omega.
 \end{array}\right. 
  \end{eqnarray} 
To derive the relationship between a KKT point and a local minimizer of   (\ref{SCP}), we need the following assumptions. 
\begin{assumption}\label{ass-cq}For point $\bfx^*\in{ \Omega}\cap \CF$,  suppose that $\nabla_{\J_*} \G (\bfx^*)$  is full column rank and 
\begin{equation} \label{CQ}
\widehat{\rN}_{\cal F\cap\Omega}(\bfx^*)=\widehat{\rN}_{\cal F}(\bfx^*)+\widehat{\rN}_{\Omega}(\bfx^*).
\end{equation}
\end{assumption}
One can note that condition \eqref{CQ} is a constraint qualification. 
This condition can be removed if $\Omega=\R^{K}$ or $\bfx^*$ lies in the interior of $\Omega$ due to $\widehat\rN_{\Omega}(\bfx^*)=\{\bf0\}$.
\begin{theorem}[KKT points and local minimizers]\label{nec-suff-opt-con-KKT}
~
\begin{itemize}
\item[a)]  A local minimizer $\bfx^*$ of (\ref{SCP}) is a KKT point under Assumption \ref{ass-cq}.
\item[b)] A KKT point $\bfx^*$ of (\ref{SCP}) is  a local minimizer  if functions $f$  and $G_{mn}$ for all $m\in\M, n\in\N$ are  locally convex around $\bfx^*$.
\end{itemize}
\end{theorem}
\proof{Proof} 
a) According to \cite[6.12 Theorem]{RW1998}, a local minimizer $\bfx^*$ of problem \eqref{SCP} is $$-\nabla f(\bfx^*)\in\widehat{\rN}_{\cal F\cap\Omega}(\bfx^*)=\widehat{\rN}_{\cal F}(\bfx^*)+\widehat{\rN}_{\Omega}(\bfx^*),$$
by Assumption \ref{ass-cq}, which using \eqref{SG-Normal Cone} yields condition \eqref{KKT-point}.


  {
b) Let $(\bfx^*, \W ^*)$ be a KKT point satisfying  (\ref{KKT-point}). We prove the conclusion by two cases.}
\begin{itemize}[leftmargin=12pt]
\item  { {Case $\|\Z ^*\|^+_0 <s$.} Condition \eqref{KKT-point} and $\widehat \rN_{\S}(\Z ^*)=\{{\bf0}\}$
from \eqref{NS-exp} suffice to $\W ^*={\bf0}$ and $-\nabla f(\bfx^*)\in \widehat{\rN}_{\Omega}(\bfx^*)$, which by the local convexity of $f$ around $\bfx^*$ yields
\begin{eqnarray*} 
f(\bfx) \geq f( \bfx^*)+\langle \nabla f( \bfx^*),\bfx-\bfx^*\rangle \geq f(\bfx^*).
\end{eqnarray*}
for any $\bfx\in \CF \cap\Omega$ around $\bfx^*$, 
where the second inequality is due to $-\nabla f(\bfx^*)\in \widehat{\rN}_{\Omega}(\bfx^*)$.} This shows the local optimality of $\bfx^*$.
 
\item  {Case $\|\Z ^*\|^+_0 = s$.}
Consider a  local region $\mathbb N(\bfx^*,\epsilon)$ of $\bfx^*$ for a given sufficiently small radius $\epsilon>0$ and
$\bfx\in \CF\cap\Omega \cap\mathbb N(\bfx^*,\epsilon)$.
To derive the results, we claim several facts.
First, the local convexity of $G_{mn}$ around $\bfx^*$ leads to  
\begin{eqnarray}\label{yyAxx}
\eqspace{1.35}
\begin{array}{lll}
Z_{mn} -Z_{mn}^* = 
G_{mn}(\bfx)-G_{mn}(\bfx^*) 
 \geq  \langle \nabla G_{mn}(\bfx^*), \bfx - \bfx^* \rangle,\end{array}
\end{eqnarray}
for any $(m,n)\in{\M}\times {\N}$. It follows from $\W ^*\in \widehat{\rN}_\S(\Z ^*)$ in \eqref{KKT-point},   \eqref{gamma*}, and  \eqref{NS-exp} that
\begin{eqnarray}\label{yTlT}
\eqspace{1.35}
\begin{array}{lll}
 \W ^*_{\J_*}\geq0,\qquad \W ^*_{\overline \J_*}={\bf0},\qquad \Z ^*_{\J_*}= {\bf0}.
\end{array}
\end{eqnarray}
 Similar to \eqref{z*-Gx*}, let  $\Z =\G(\bfx).$  For any $\bfx\in \CF\cap\Omega \cap \mathbb N(\bfx^*,\epsilon)$, $\|\Z\|^+_0 \leq s$. Moreover, for a sufficiently small $\epsilon$,  we note from the continuity of $\G(\cdot)$ that $Z_{mn}>0$ if $Z_{mn}^*>0$, which indicates 
 $\|\Z\|^+_0 \geq \|\Z^* \|^+_0 = s$. Overall, $\|\Z\|^+_0  = s$ for any $\bfx\in \CF\cap\Omega \cap \mathbb N(\bfx^*,\epsilon)$. This means that  
 \begin{eqnarray*}
 \eqspace{1.35}
 \begin{array}{lll}
 \Gamma_+=\left\{n\in{\N}:\Z_{:n}^{\max}>0\right\}=\left\{n\in{\N}: (\Z^*)_{:n}^{\max}>0\right\}=\Gamma_+^*.
 \end{array}
 \end{eqnarray*}
Since $\Gamma_0^*\cap \Gamma_+^*=\emptyset $ and $\|\Z\|^+_0 =| \Gamma_+| = s$, the above condition indicates $\Z_{:\Gamma_0^*}\leq{\bf0}$, which combining $\J_*=V_{\Gamma_0^*}\subseteq {\M}\times\Gamma_0^*$ suffices to
\begin{eqnarray}\label{yTyT}
\Z_{\J_*}\leq{\bf0}.
\end{eqnarray}
{Finally, the above facts and the local convexity of $f$ can conclude  that
\begin{eqnarray*} 
 \eqspace{1.35}
 \begin{array}{lcl}
f(\bfx)&\geq&f( \bfx^*)+\langle \nabla f( \bfx^*),\bfx-\bfx^*\rangle \\
&{=} &  f( \bfx^*)+\langle \nabla f( \bfx^*)+\sum  W^*_{mn}\nabla G_{mn}(\bfx^*) , \bfx-\bfx^*\rangle-  \langle \sum  W^*_{mn}\nabla G_{mn}(\bfx^*) , \bfx-\bfx^*\rangle\\
&\overset{(\ref{KKT-point})}{\geq} &  f( \bfx^*)- \sum  W^*_{mn} \langle \nabla G_{mn}(\bfx^*) , \bfx-\bfx^*\rangle\\
& \overset{(\ref{yyAxx},\ref{yTlT})}{\geq}& f( \bfx^*)-\sum W^*_{mn}( Z_{mn} -Z_{mn}^*)\\
 &\overset{(\ref{yTlT})}{=}&  f( \bfx^*)-\sum_{  \J_*}  W^*_{mn} Z_{mn} \\
 &\overset{(\ref{yyAxx},\ref{yTyT})}{\geq}& f( \bfx^*), \end{array}
 \end{eqnarray*}
which demonstrates the local optimality of $\bfx^*$ to   problem (\ref{SCP}).}  \Halmos
 \end{itemize}
\endproof

\subsection{$\tau $-stationary points}
Our next result is about the $\tau $-stationary point of  (\ref{SCP}) defined as follows: A point $\bfx^*\in\R^{K}$  is called a $\tau $-stationary point of  (\ref{SCP}) for some $\tau >0$ if there is a $\W ^*\in\R^{M\times N}$ such that
\begin{eqnarray}\label{eta-point}
\eqspace{1.35}
\left\{
 { \begin{array}{lll}
\bfx^*&=&{\bpi_\Omega}\left(\bfx^* - \tau [\nabla f(\bfx^*)+\sum  W^*_{mn}  \nabla G_{mn}(\bfx^*)] \right),\\
\Z^*&\in&{\bpi_\S}\left( \Z^* +\tau  \W ^*\right).
 \end{array}}\right.
  \end{eqnarray}
The following result shows that a $\tau $-stationary point also has a close relationship with the local minimizer of problem (\ref{SCP}).
\begin{theorem}[$\tau $-stationary points and local minimizers]\label{nec-suff-opt-con-sta} ~
\begin{itemize}[leftmargin=17pt]
\item[a)]
Under Assumption \ref{ass-cq}, a local minimizer $\bfx^*$ is also a $\tau $-stationary point  
\begin{itemize}[leftmargin=17pt]
\item either for any $\tau>0$ if $\|\Z^* \|^+_0<s$,  
\item or for any $0<\tau\leq\tau_*:= {(\Z^*)_{s}^{\downarrow}}
  /\max_{n\in\Gamma^*_0}\| \W ^*_{:n}\|$ if $\|\Z^* \|^+_0=s$, where $\W ^*$ satisfies (\ref{KKT-point}).
    \end{itemize}
\item[b)]  A  $\tau $-stationary point with $\tau>0$ is a KKT point and thus a local minimizer if functions $f$ and $G_{mn}$ for all $m\in\M, n\in\N$ are locally convex around $\bfx^*$.
\end{itemize}
\end{theorem}
\proof{Proof}
a) It follows from  Theorem  \ref{nec-suff-opt-con-KKT} that a local minimizer $\bfx^*$  is also a KKT point. Therefore, we have condition \eqref{KKT-point}. {Since $\Omega$ is convex, the first condition in \eqref{KKT-point} is equivalent to
\begin{eqnarray}\label{Z*-pi-S-0}
\eqspace{1.35}
  \begin{array}{lll}
\bfx^*&=&{\rm argmin}_{\bfx\in\Omega} \|\bfx-(\bfx^*-\tau [\nabla f(\bfx^*)+\sum  W^*_{mn}  \nabla G_{mn}(\bfx^*)])\|^2\\[1ex]
&=&{\bpi_\Omega}\left(\bfx^* - \tau [\nabla f(\bfx^*)+\sum  W^*_{mn}  \nabla G_{mn}(\bfx^*)] \right).
 \end{array}
  \end{eqnarray}
 We next show}
\begin{eqnarray}\label{Z*-pi-S}
\Z ^*\in{\bpi_\S}\left( {\Z}^*+\tau  \W ^*\right).
\end{eqnarray}  
If $\|\Z ^*\|^+_0<s$, then  \eqref{KKT-point} and $\widehat \rN_{\S}(\Z ^*)=\{{\bf0}\}$
from \eqref{NS-exp} yield $\W ^*={\bf0}$, resulting in \eqref{Z*-pi-S}  for any $\tau>0$.
Now consider the case of $\|\Z ^*\|^+_0=s$. Under such a case, conditions  \eqref{yTlT} hold, 
 which by  $\J_*=V_{\Gamma_0^*}\subseteq {\M}\times\Gamma_0^*$ allows us to  derive that
\begin{eqnarray}\label{eqp4}
  \W ^*_{:\overline \Gamma^*_0 }={\bf0}, ~~{\bf0}\geq \Z ^*_{:\Gamma^*_0}~\bot ~\W ^*_{:\Gamma^*_0}\geq{\bf0}.
\end{eqnarray}
By $0<\tau\leq\tau_*$, we have 
\begin{eqnarray*}
\eqspace{1.35}
\begin{array}{lcl}\forall~n\in \Gamma_0^*,~~\tau \|\W _{:n}^*\|\leq\tau_* \max_{{n\in \Gamma^*_0}}\|\W _{:n}\| = \tau_* r_* = (\Z^*)^{\downarrow}_s.
\end{array}\end{eqnarray*}
The above condition and \eqref{eqp4} show that \eqref{Z*-pi-S} by Proposition \ref{pro-eta}.

b)  We only prove that a \ts\ is a KKT point because Theorem \ref{nec-suff-opt-con-KKT} b) enables us to conclude the conclusion immediately. {We note that a \ts\ satisfies \eqref{Z*-pi-S-0} and \eqref{Z*-pi-S}. The former implies the first condition in \eqref{KKT-point} and $\bfx^*\in\Omega$, and the latter yields $\|\Z ^*\|^+_0\leq s $.} Comparing \eqref{eta-point} and \eqref{KKT-point}, we only need to prove $\W ^*\in \widehat{\rN}_\S(\Z ^*)$. If $\|\Z ^*\|^+_0< s $, then Proposition \ref{pro-eta} shows $\W^*={\bf0}\in \widehat{\rN}_\S(\Z ^*)$ by \eqref{NS-exp}. If  $\|\Z ^*\|^+_0 = s $, then Proposition \ref{pro-eta} shows \eqref{eqp4},  
which by the definition of $\J_*=V_{\Gamma_0^* }$ in \eqref{gamma*} indicates 
\begin{eqnarray*}    
\W_{\J_*}^*\geq 0,~~\W_{\overline \J_*}^*={\bf0}, 
\end{eqnarray*}
contributing to $\W ^*\in \widehat{\rN}_\S(\Z ^*)$ by \eqref{NS-exp}.
  \Halmos\endproof
 \subsection{Relationships to the binary integer programming} 
In this part, we study two BIP reformulations for problem \eqref{SCP}. It is easy to see that  problem \eqref{SCP} can be equivalently reformulated as the following binary integer programming,
\begin{equation}  \label{BIP}
 \eqspace{1.35}
\begin{array}{l}
 {\min}_{\bfx,\bfy}~f(\bfx),~~
{\rm s.t.}~\bfy\in \{0,1\}^N \cap {\cal B},~ \bfx  \in \Omega \cap \CD(\bfy), 
\end{array}
\tag{BIP}\end{equation}
where ${\bf1}$ is a vector with all entries and
\begin{equation*} 
 \eqspace{1.35}
\begin{array}{ll}
 {\cal B}  &:=  \left\{ \bfy\in \R^N: \langle {\bf1}, \bfy\rangle \geq N-s \right\}, \\
\CD(\bfy)&:= \left\{  \bfx  \in \R^{K}: ~ 
 y_n\cdot{\max}_{m\in\M} G_{mn}(\bfx)  \leq 0,~\forall~n\in\N \right\}. 
\end{array}
\end{equation*}
One can rewrite the feasible set of \eqref{BIP} as
 \begin{eqnarray*}
 \label{BIP-feas-1}
 \eqspace{1.35}
 \begin{array}{lll}
 {\CE := \left\{(\bfx,\bfy)\in \R^{K} \times \{0,1\}^N : \bfx\in\Omega \cap \CD(\bfy) , \bfy\in{\cal B} \right\}.}
 \end{array}
 \end{eqnarray*}
For a point $(\bfx^*, \bfy^*)\in  \CE$,  since  $\{0,1\}^N$ includes a unique $\bfy^*$, the  tangent cone of $ \CE$  at point $(\bfx^*, \bfy^*)$ can be calculated by 
\begin{equation} \label{tangent-decom}
 \eqspace{1.35}  {\begin{array}{lcl}
{\rT}_{ \CE }(\bfx^*, \bfy^*)  &=&  {\rT}_{\bigcup_{ \bfy\in\{0,1\}^N  } ( \Omega \cap \CD(\bfy) \times {\cal B} ) }(\bfx^*, \bfy^*) \\
&\overset{\eqref{cup-tangent}}{=} &\bigcup_{\left(\bfy\in\{0,1\}^N,\bfy=\bfy^*\right) } {\rT}_{\left(\Omega \cap \CD(\bfy) \times  {\cal B}\right)}(\bfx^*, \bfy^*) \\
&= &   {\rT}_{ (\Omega \cap\CD(\bfy^*)) \times  (\{\bfy^*\}\cap{\cal B}) }(\bfx^*, \bfy^*)\\
&= &   {\rT}_{ (\Omega \cap\CD(\bfy^*)) \times  \{\bfy^*\}}(\bfx^*, \bfy^*)\\
& \subseteq &   {\rT}_{ \Omega \cap\CD(\bfy^*)}(\bfx^*) \times \rT_{  \{\bfy^*\} }( \bfy^*)\\
&=&  {\rT}_{\Omega \cap\CD(\bfy^*)  }(\bfx^*) \times  \{0\},
\end{array}}  \end{equation}
where  $\subseteq$ is from \cite[6.41 Proposition]{RW1998}, which turns to an equality if  $\CD(\bfy^*)$ is convex. This can be ensured if $\{G_{mn}: m\in\M, n\in \Gamma_1\}$ are locally convex  around $\bfx^*$. Here $\Gamma_1:=\{n\in\N: y^*_n=1\}$.  By denoting the active set of $\bfy^*$ as
 $$\J^B_*:=\{(m,n)\in\M\times\N: G_{mn}(\bfx^*)=0, ~n\in\Gamma_1\},$$
one can calculate that  
 \begin{equation} \label{TD-*} 
 \eqspace{1.35}
 \begin{array}{lll}
 \rT_{\CD(\bfy^*)} (\bfx^*) &\subseteq&\{{\bfd\in\R^K}:~\langle\bfd, \nabla  G_{mn}(\bfx^*)\rangle \leq 0, (m,n)\in \J^B_*\}.
 \end{array}\end{equation}  
 The equation holds  if  $\nabla_{\J^B_*}\G(\bfx^*)$ is full column rank from \cite[Lemma 12.2]{nocedal2006quadratic}. Based on these facts we define the following binary KKT (BKKT) point. We say  $(\bfx^*, \bfy^*)$ is a BKKT point of problem \eqref{SCP}  if there is a $\W ^*\in\R^{M\times N}$ such that 
  \begin{eqnarray}\label{BKKT-point}
 { \left\{\begin{array}{rll}
 -\nabla f(\bfx^*)- \sum W^*_{mn}  \nabla G_{mn}(\bfx^*) &\in&\widehat{\rN}_{\Omega}(\bfx^*),\\[1ex]
  W ^*_{mn}\geq 0, ~(m,n)&\in& \J^B_*,\\[1ex]
  ~ W ^*_{mn} = 0,~ (m,n)&\notin &\J^B_*,\\[1ex]
 (\bfx^*, \bfy^*)&\in&\CE.
 \end{array}\right.}
  \end{eqnarray} 
Similarly, we have the following first-order necessary optimality conditions for \eqref{BIP} or \eqref{SCP}. 
\begin{lemma} Let $(\bfx^*, \bfy^*)$ be a local  minimizer of (\ref{BIP}) or (\ref{SCP}). If  $\{G_{mn}: m\in\M, n\in \Gamma_1\}$ are locally convex  around $\bfx^*$, $\nabla_{\J^B_*}\G(\bfx^*)$ is full column rank, and $\Omega$ and $\CD(\bfy^*)$ can not be separated,  then it satisfies (\ref{BKKT-point}).
 \end{lemma}
 \proof{Proof} By \cite[Theorem 6.12]{RW1998}, a necessary
condition for $(\bfx^*, \bfy^*)$ to be locally optimal is  
 \begin{eqnarray*}\eqspace{1.35}
\begin{array}{lll}
 \langle \nabla f(\bfx^*),  \bfd \rangle = \langle (\nabla f(\bfx^*),{\bf0}), (\bfd, \bfd') \rangle \geq 0,~~\forall  (\bfd, \bfd') \in \rT_{\CE}(\bfx^*, \bfy^*).
\end{array} \end{eqnarray*}  
Since $\{G_{mn}: n\in \Gamma_1\}$ are locally convex around $\bfx^*$,  $\CD(\bfy^*)$ is convex, which combining with the convexity of $\Omega$ and \eqref{tangent-decom} contributes to $\rT_{\CE}(\bfx^*, \bfy^*)  = ({\rT}_{ \Omega \cap \CD(\bfy^*) }(\bfx^*)) \times  \{{\bf0}\}$. As a result, 
\begin{equation} \label{B-Normal cone}
\begin{array}{lcl}
 \widehat{\rN}_{ \CE }(\bfx^*, \bfy^*) =  \widehat{\rN}_{\Omega \cap \CD(\bfy^*) }(\bfx^*) \times \R^N=  (\widehat{\rN}_{\Omega  }(\bfx^*) +\widehat{\rN}_{\CD(\bfy^*) }(\bfx^*) )\times \R^N,
\end{array}  \end{equation}
where the second equality holds because $\CD(\bfy^*)$ and $\Omega$ are convex and can not be separated, and \cite[6.42 Theorem]{RW1998}. 
The the  full column rankness of $\nabla_{\J^B_*}\G(\bfx^*)$ implies the equation holds in \eqref{TD-*}, thereby leading to 
\begin{equation*}
\begin{array}{lcl}
 \widehat{\rN}_{\CD(\bfy^*) }(\bfx^*, \bfy^*) =  \left\{ \sum W^*_{mn}  \nabla G_{mn}(\bfx^*): \begin{array}{l}
W ^*_{mn}\geq 0, ~(m,n)\in \J^B_,\\
 W ^*_{mn} = 0,~ (m,n) \notin  \J^B_*  
\end{array}  \right\}.
\end{array}  \end{equation*}
Then from \cite[6.12 Theorem]{RW1998}, a  minimizer of problem \eqref{BIP} satisfies
$-(\nabla f(\bfx^*);{\bf0})\in \widehat{\rN}_{ \CE }(\bfx^*, \bfy^*)$, which by \eqref{B-Normal cone} shows the desired result.
 \Halmos\endproof 
 One can easily see that a KKT point of problem (\ref{SCP}) must be a  BKKT  point. In fact, we can let $\bfy^*$ satisfy $y^*_n=1$ if $n\in \Gamma_0^*$ and $y^*_n=0$  otherwise, leading to   $\J^B_*=\J_*$. Then conditions \eqref{BKKT-point} are satisfied due to  \eqref{KKT-point}. Based on this assertion, we have the following relationships.
\begin{corollary}\label{coro-relation} Let $\bfx^*\in\CF$ be one of the $\tau$-stationary point, KKT point, BKKT point, and local minimizer. Define the following cases and conditions: 
\begin{itemize}
\item {\bf Case 1:} Any $\tau>0$ if $\|\Z^*\|_0^+ {<} s$ ; 
\item {\bf Case 2:} Any $\tau  \in (0,\tau_*]$ if $\|\Z^*\|_0^+ {=} s$,   where $\tau_*$ is defined in Theorem \ref{nec-suff-opt-con-sta};
\item {\bf Cond 1:}  $\nabla_{\J_*} \G (\bfx^*)$  is full column rank, and when $\|\Z^*\|_0^+ {=} s$ there is condition (\ref{CQ}); 
\item {\bf Cond 2:} $f$  and each $G_{mn}$  are  locally convex around $\bfx^*$;
\item {\bf Cond 3:}  $\nabla_{\J^B_*} \G (\bfx^*)$  is full column rank, $\{G_{mn}: m\in\M, n\in \Gamma_1\}$ are  locally convex around $\bfx^*$, and $\Omega$ and $\CD(\bfy^*)$ can not be separated. 
\end{itemize} 
{Then we have the following relationships (also shown in (\ref{relation})):
\begin{itemize}
\item[a)] A local minimizer is a KKT point under Cond 1. If Cond 3 holds, then the opposite holds true;
\item[b)] A $\tau$-stationary point is a KKT point.  The opposite holds true for either Case 1 or Case 2;
\item[c)] A KKT point is a BKKT point;
\item[d)] A local minimizer is a BKKT point under Cond 2.
\end{itemize}}
\end{corollary} 

\vspace{2mm}
\begin{remark}The above theorem means that a $\tau$-stationary point for any $\tau>0$ is a KKT point. For some $\tau$ in a particular range $(0,\tau_*]$, a KKT point is also a $\tau$-stationary point. Therefore, $\tau$-stationary points are equivalent to KKT points under \textbf{Case 2}. However, in general,  being a $\tau$-stationary point is a stronger optimality condition than being a KKT point which is also better than being a  BKKT point. In addition, it is worth mentioning that even under \textbf{Cond 3}, a BKKT point may not be a local minimizer. We illustrate this by giving the following example. Let 
$$(M, N, K, s)=(1,2,2,1),~~ f(\bfx)=(x_1-2)^2,~~ \G(\bfx)=(x_1^2-x_2, x_2-1)\in\R^{1\times 2},~~ \Omega=\R^2.$$ 
Clearly, $f$  and each $G_{mn}$  are convex. Problems \eqref{SCP} and \eqref{BIP} are
 \begin{eqnarray} \label{example-bip}
  \eqspace{1.35}
     \begin{array}{rllrl}
\min& ~(x_1-2)^2,&&\min& ~(x_1-2)^2,\\
{\rm s.t.}& ~x_1^2-x_2\leq 0,&~\qquad\qquad&{\rm s.t.}&~ y_1+y_2\geq 1,~y_1\in\{0,1\},~ y_2\in\{0,1\},\\
 &~\text{or}~x_2-1\leq 0 ,&\qquad&& ~y_1(x_1^2-x_2)\leq 0, ~y_2(x_2-1)\leq 0.
     \end{array} 
  \end{eqnarray} 
For $\bfx^*=(1,1)^\top$ and $\bfy^*=(1,1)^\top$, it has $\|\G(\bfx^*)\|_0^+=0<1,~ (\bfx^*,\bfy^*)\in\CE$, and
 \begin{eqnarray*}  
     \begin{array}{lllll}
\J_* = \J^B_* =\{(1,1),(1,2)\},~~\nabla_{\J_*}\G(\bfx^*)=\left[\begin{matrix}
  2& 0\\
  -1 &1
  \end{matrix} \right].
     \end{array} 
  \end{eqnarray*}   
Clearly, $\nabla_{\J_*}\G(\bfx^*)$ is full column rank and thus from  \eqref{NSG-exp} and   \eqref{B-Normal cone}, $$\widehat{\rN}_{\S}(\Z^*)=\{{\bf0}\},~~~~ \widehat{\rN}_{ \CE }(\bfx^*, \bfy^*) =  \left\{ \left[\begin{matrix}
  2W '_{11}\\
  -W '_{11}+W '_{12}
  \end{matrix} \right]: W '_{11}\geq 0, ~W '_{12}\geq 0 \right\}\times \R^2.$$
These further imply that $\bfx^*$ does not satisfy conditions \eqref{KKT-point} (namely $\bfx^*$  is not a KKT point)  but satisfies \eqref{BKKT-point} (namely $\bfx^*$  is a BKKT point) as there exists $W '_{11} = W '_{12}=1$ such that 
 \begin{eqnarray*}  
 \eqspace{1.35}
     \begin{array}{lllll}
  \nabla f(\bfx^*) +  \sum W_{mn}^*   \nabla  G_{mn}(\bfx^*)  =(-2,0)^\top \neq {\bf0}, ~~\forall~ \W^* \in \widehat{\rN}_\S(\Z ^*),\\
  \nabla f(\bfx^*) + (2W '_{11}, -W '_{11}+ W '_{12})^\top={\bf0}. 
  \end{array} 
  \end{eqnarray*} 
  However,  $\bfx^*$ is not a local minimizer. To see this, for any small radius $\epsilon\in(0,1)$, in  neighbourhood ${\mathbb N}( \bfx^* ,\epsilon)$ of $\bfx^*$, we could find a point $\bfx'=(1+\epsilon/2,1)$ and $\bfy'=(0,1)$ to be feasible to \eqref{example-bip} but $f(\bfx')=(1-\epsilon/2)^2<1=f(\bfx^*)$.
  \end{remark}
 \section{Semismooth Newton Method}\label{sec:Newton}
 As shown in \eqref{relation},  a  \ts\ is a better solution than a KKT or BKKT point. Therefore,  in this section, we aim to find a \ts\ to problem \eqref{SCP} by developing a Newton-type algorithm.  Hereafter, for   $\tau>0$, we always denote
 \begin{eqnarray}\label{U-Gamma}
  \eqspace{1.35}
     \begin{array}{lll}
     \bfw&:=&(\bfx; {\rm vec}(\W))\in\R^{K+MN},\\
     \Z&:=&\G(\bfx),\\
     \L&:=&\Z+\tau  \W,\\
U_{T}&:=&\{(m,n)\in\M \times T:  \Lambda _{mn}\geq 0\},~\text{where}~T\in \T( \L; s),\\
{\varphi(\bfw;\J)}&{:=}&{\bfx-\Pi_{\Omega}\left(\bfx-\tau[\nabla f(\bfx)+ \sum_{ \J }W_{mn}   \nabla  G_{mn}(\bfx)]\right),}
     \end{array} 
  \end{eqnarray} 
where $(\bfx;\bf{y})=(\bfx^\top~\bf{y}^\top)^\top$ and ${\rm vec}(\W)$ transforms matrix $\W$ into a column vector by vertically stacking its columns. Similar definitions to \eqref{U-Gamma} are also applied to $\bfw^*:=(\bfx^*; \ve(\W^*))$ and $\bfw^{\ell}:=(\bfx^{\ell}; \ve(\W^{\ell}))$, where the former is a \ts\ and the latter is the point generated by our proposed algorithm at the $\ell$th step.  Additionally, we shall point out the difference between the definitions of $U_{T}$ and $V_{\Gamma}$ in \eqref{notation-z}. The former is given upon $\L$ while the latter is defined based on $\Z$. However, for a $\tau$-stationary point, they two are identical (see Theorem \ref{sta-eq}).

{ Let $\partial\varphi(\cdot;\J)$ be the Clarke generalized Jacobian of $\varphi(\cdot;\J)$ for fixed $\J$. Decompose $\H\in\partial\varphi(\bfw;\J)$ as $\H:=[  \P~\Q~{\bf0}]$ with $\P\in\R^{K\times K}, \Q\in\R^{K\times |\J|}$, and ${\bf0}\in\R^{K\times |\overline\J|}$,} and denote a system of equations as
   \begin{eqnarray}\label{equation-F-mu}
 \F(\bfw;\J) := \left[
  \eqspace{1.25}
     \begin{array}{c}
   {\varphi(\bfw;\J) } \\
  {\rm vec}(\Z_{\J})\\
     {\rm vec}(\W_{\overline \J})\\
     \end{array}
     \right]
  \end{eqnarray}
 and an associated matrix as
  \begin{eqnarray}\label{F-mu}
 \nabla \F_{\mu}(\bfw;{\J}):=\left[  \eqspace{1.25}
     \begin{array}{ccccc}
       {\P} &~~&  {\Q}&~~&{\bf0}\\
     \nabla_{ \J} \G(\bfx)^\top&&-\mu {\bf I}_{|  \J|}&~~&{\bf0}\\
     {\bf0}&&{\bf0}&~~&{\bf I}_{|\overline \J|}\\
     \end{array}
     \right],
  \end{eqnarray}
{ where    $\H=[  \P~\Q~{\bf0}]\in  \partial \varphi(\bfw;\J) $ and ${\bf I}_n$ is the $n$th order of identity matrix. Therefore, $\nabla \F_0 (\bfw;{\J})$  is an element of   $ \partial \F (\bfw;{\J})$ for given $\J$. Thanks to this, hereafter we always  let}
    \begin{eqnarray}\label{sta-eq-2}
  \nabla \F (\bfw;{\J}):=\nabla \F_{0}(\bfw;{\J}).
  \end{eqnarray} 
 \subsection{Stationary equations} 
To employ the Newton method, we need to convert a \ts\ satisfying  \eqref{eta-point} to a system of equations, stated as the following theorem. 
\begin{theorem}[$\tau$-stationary equations]\label{sta-eq} A point $\bfx^*$  is a \ts\ with  $\tau>0$  of (\ref{SCP}) if and only if there is a $\W ^*\in\R^{M\times N}$ such that
\begin{eqnarray}\label{sta-eq-1}
\eqspace{1.35}
\begin{array}{lll}
 \T(\L^*;s)&\ni&\Gamma_0^*,\\
 U_{\Gamma_0^*}&=&\J_*,\\
  \F(\bfw^*;\J_*)&=&{\bf0},
 \end{array} \end{eqnarray}
where $\J_*$ is defined as (\ref{gamma*}). 
\end{theorem}
\proof{Proof} 
First of all, we denote 
\begin{eqnarray*} 
\eqspace{1.35}
\begin{array}{ll}
 \J_-:=( {\M}\times\Gamma_0^*) \setminus \J_*=\{(m,n):~Z_{mn}^* \neq 0,~ \forall~ n\in \Gamma^*_0\}. 
\end{array}
\end{eqnarray*}
 The definition of $\Gamma^*_0$ means $  \Z^*_{:\Gamma^*_0} \leq {\bf0}$, which by $\J_*\subseteq {\M}\times\Gamma_0^*$  derives
\begin{eqnarray}\label{Z-J*-J-} 
\Z^*_{\J_*} = {\bf0}, ~~ \Z^*_{\J_-} < {\bf0}.
\end{eqnarray} 

\underline{Necessity.} Since  $\bfx^*$  is a \ts\   of (\ref{SCP}), there is a $\W ^* $ satisfying   $\Z^*  \in {\bpi_\S}( \L^*)$ from \eqref{eta-point}, which together with Proposition \ref{pro-eta} suffices to $\T(\L^*;s)=\{\Gamma_0^*\}$.   Therefore, $\Gamma_0^*\in \T(\L^*;s)$. If $\|\Z^*\|_0^+<s$, then $\W^*={\bf0}$ by Proposition \ref{pro-eta}, which immediately shows $\J_*=U_{\Gamma_0^*}$ and \eqref{sta-eq-1}. Next, we prove the conclusion for case $\|\Z^*\|_0^+=s.$ It follows from \eqref{uPu-equ} that
\begin{eqnarray}\label{uPu-equ-*}
\W ^*_{:\overline \Gamma_0^*}={\bf0} ,~~ {\bf0}\geq \Z ^*_{:\Gamma_0^*} \perp \W ^*_{:\Gamma_0^*} \geq{\bf0}.
\end{eqnarray}
To show $\J_*=U_{\Gamma_0^*}$, we only need to show $Z_{mn}^*=0 \Leftrightarrow (Z^*+\tau  W ^*)_{mn}\geq 0$ for any $ n\in\Gamma_0^*$. This is clearly true due to the second condition in \eqref{uPu-equ-*}. The conditions in  \eqref{uPu-equ-*}  and \eqref{Z-J*-J-} indicate  $\W_{\J_-}^*={\bf0}$, thereby resulting in  $$\W_{\overline \J_*}^*=\W_{ \J_-\cup( {\M}\times \overline \Gamma_0^*)}^*={\bf0}.$$  This leads to the following condition and hence displays \eqref{sta-eq-1},
\begin{eqnarray*}  \eqspace{1.35}
\begin{array}{lll}
{ {\bf0}\overset{\eqref{eta-point} }{=} \varphi(\bfw^*;\M\times\N)  =\varphi(\bfw^*;\J_* ).}
\end{array} 
\end{eqnarray*} 

\underline{Sufficiency.} We aim  to prove \eqref{eta-point}. Condition $\varphi(\bfw^*;\M\times\N) ={\bf0}$ follows from the first and third equations in \eqref{sta-eq-1} immediately. We next show  $\Z^*\in{\bpi_\S}(\L^*)$  in \eqref{eta-point}. Condition $\J_*=U_{\Gamma_0^*}$ implies that
 \begin{eqnarray}  \label{gamma-T-*-Z-W}
\Z_{:\Gamma_0^*}^*={\bf0}, ~~~ \W ^* _{:\Gamma_0^*}\geq {\bf0},  
\end{eqnarray}
 which together with $\Gamma_0^*  \in \T(\L^*;s)$ and \eqref{psz} derives 
\begin{eqnarray}\label{eqp3} 
\eqspace{1.35} 
\Big[ (\L^*)^-_{ :\Gamma_0^*}~
( \L^*)_{:\overline\Gamma_0^*} 
\Big] \in  {\bpi_\S}(\L^*).
\end{eqnarray}
We finally show the left-hand side of \eqref{eqp3} is  $\Z^*$. Since $\J_*\subseteq {\M}\times\Gamma_0^*$, we have ${\M} \times \overline\Gamma_0^*   \subseteq \overline \J_*$, thereby leading to $\W ^*_{:\overline\Gamma_0^*}={\bf0}$ from \eqref{sta-eq-1}. Hence, 
\begin{eqnarray} \label{eqp3-1} ( \L^*)_{:\overline\Gamma_0^*} = (\Z^* +\tau  \W ^*)_{:\overline\Gamma_0^*} =\Z^* _{:\overline\Gamma_0^*}.\end{eqnarray}  
Condition \eqref{sta-eq-1} means that $W ^*_{\J_-} ={\bf0}$ due to $  \J_-\subseteq\overline \J_*$. As a result,
\begin{eqnarray*}   \eqspace{1.35}
\begin{array}{cccccccc}
\L^*_{\J_*} &=& (\Z^* +\tau  \W ^*)_{\J_*} &\overset{\eqref{Z-J*-J-} }{=}& \tau \W_{\J_*}^* &\overset{\eqref{gamma-T-*-Z-W}}{\geq}&  {\bf0},\\
\L^*_{\J_-} &=& (\Z^* +\tau  \W ^*)_{\J_-} &= &  \Z^*_{\J_-} &\overset{\eqref{Z-J*-J-} }{<}&{\bf0}.  
\end{array} 
\end{eqnarray*} 
Using the above conditions and \eqref {Z-J*-J-} enables us to show $(\L^* _{\J_*})^- = \Z^*_{\J_*}$ and $(\L^* _{\J_-})^- = \Z^*_{\J_-}$,  which combining $\J_-\cup \J_*=({\M}\times \Gamma_0^*)$ and conditions \eqref{eqp3} and \eqref{eqp3-1} proves $\Z^*\in  {\bpi_\S}(\L^*)$.   
\Halmos\endproof
 
\subsection{Algorithmic design}  
We note that equation $ \F(\bfw^*;\J_*)=0$ in \eqref{sta-eq-1} involves an unknown set $\J_*$. Therefore, to proceed with the semismooth Newton method \cite{kojima1986extension,qi1993nonsmooth}, we have to find $\J_*$, which will be adaptively updated by using the approximation of $\bfw^*$. More precisely,  let $\bfw^\ell$ be the current point, we first select  
\begin{align}\label{choice-of-H}
\begin{array}{l}
T_{\ell}\in\T(\L ^{\ell} ;s)~~\text{and} {~~\H^{\ell}\in  \partial \varphi(\bfw ^{\ell};U_{T_{\ell}} )},\end{array}
 \end{align}
 based on which we  find Newton direction $\bfd^{\ell}\in\R^{K+MN}$ by solving the following linear equations:
 \beq \label{newton-dir-0}
 \nabla \F_{\mu_\ell}(\bfw^{\ell};U_{T_{\ell}})~\bfd^{\ell} = - \F(\bfw^{\ell};U_{T_{\ell}}),
 \eeq
 where $\nabla \F_{\mu}(\bfw;{\J})$ is defined as \eqref{F-mu} and  $\mu_\ell$ is updated by
 \beq \label{update-mu}
\mu_\ell=\min\{\nu\mu_{\ell-1}, \rho\|\F(\bfw^{\ell};U_{T_{\ell}})\|\},
 \eeq
with $\nu\in(0,1)$ and $\rho>0$. The framework of our proposed method is presented in Algorithm \ref{Alg-NL01}.

\begin{algorithm}
	\caption{\nscp: Semismooth Newton method for (\ref{SCP})}
	\begin{algorithmic}[1] \label{Alg-NL01}
		\STATE Initialize $\bfw^0{=}(\bfx^0;{\rm vec}(\W)^0)$ positive parameters $ \texttt{maxIt}, ~\texttt{tol}, ~\rho, ~\tau, ~\underline\mu, ~\gamma$, and $\nu, \pi \in (0,1)$. 
		\STATE {Select $T_0\in\T(\L ^0;s)$ and $\H^0\in  \partial \varphi(\bfw ^0;U_{T_0} )$}, compute $\mu_0{=}\min\{\underline\mu, \rho\|\F(\bfw^{0};U_{T_0})\|\}$, and set $\ell{=}0$.
\IF{ $\ell\leq\texttt{maxIt}$ and $\|\F(\bfw^{\ell};U_{T_{\ell}})\|\geq\texttt{tol}$}
\STATE ~~~ {\bf if} \eqref{newton-dir-0} is solvable {\bf then}		
\STATE ~~~ \qquad Update $\bfd^{\ell}$ by solving \eqref{newton-dir-0}.
\STATE ~~~ {\bf else} 
\STATE ~~~ \qquad Update $\bfd^{\ell}=- \F(\bfw^{\ell};U_{T_{\ell}}) $.  
\STATE ~~~ {\bf endif} 	
	\STATE ~~~ Find the minimal integer $t_\ell\in\{0,1,2,\ldots \}$ such that
	\begin{eqnarray}  \label{gradually-feasible}
\|\G(\bfx^{\ell} + \pi^{t_\ell}\bfd^{\ell}_\bfx)\|_0^+\leq (\gamma+1) s,
\end{eqnarray}  
~~~ where $\bfd^{\ell}_\bfx$ is the subvector formed by the first $K$ entries in $\bfd^{\ell}$.
		\STATE ~~~ Update  $ \bfw^{\ell+1}~=~\bfw^{\ell}+\pi^{t_\ell}\bfd^{\ell}$.
		\STATE ~~~ {Update $U_{T_{\ell+1}}\in\T(\L ^{\ell+1} ;s)$ and $\H^{\ell+1}\in  \partial \varphi(\bfw ^{\ell+1};U_{T_{\ell+1}} )$.}
		\STATE  ~~~ Update $\mu_{\ell+1}$ by	\eqref{update-mu} and set $\ell = \ell+1$.   		
\ENDIF		
\RETURN $\bfw^{\ell}$.
	\end{algorithmic}
\end{algorithm}

\begin{remark}\label{rem:complexity}
Regarding  Algorithm \ref{Alg-NL01}, we have some observations.
\begin{itemize}[leftmargin=17pt]
\item[i)] One of the halting conditions makes use of $\|\F(\bfw^{\ell};U_{T_{\ell}})\|$. The reason behind this is that if   point $\bfw^\ell$ satisfies $\|\F(\bfw^{\ell};U_{T_{\ell}})\|=0$, then it is a \ts\ of  (\ref{SCP}) by Theorem \ref{sta-eq}.
\item[ii)] Recalling (\ref{sta-eq-1}), we are expected to update $\bfd^{\ell}$ by solving
\beq \label{newton-dir-0-0}
 \nabla \F (\bfw^{\ell};U_{T_{\ell}})~\bfd^{\ell} = - \F(\bfw^{\ell};U_{T_{\ell}}),
 \eeq
 instead of (\ref{newton-dir-0}). However,  the major concern is made on the existence of $\bfd^{\ell}$ by solving (\ref{newton-dir-0-0}). To overcome such a drawback, we add a smoothing term $-\mu_{\ell}{\bf I}_{|\J|}$ to increase the possibility of the non-singularity of  $\nabla \F_{\mu_{\ell}}(\bfw^{\ell};U_{T_{\ell}})$.  This idea has been adopted in literature, e.g., \cite{chen1998global,zhou2021quadratic}
\item[iii)]  When the algorithm derives a direction $\bfd^{\ell}$, we use condition \eqref{gradually-feasible} to decide the step size. This condition allows the next point to be chosen in a larger region to some extent by setting $\gamma>0$. However, it can ensure that the next point does not step far away from the feasible region by setting a small value of $\gamma$ (e.g., $\gamma=2/s$). In this way, the algorithm performs relatively steadily. In addition, we will show that if the starting point is chosen close to a stationary point, condition \eqref{gradually-feasible} can be always satisfied with $\pi^{t_\ell}=1$ from Theorem \ref{the:quadratic}.
\item[iv)]  Finally, when projecting a point onto $\Omega$, we expect its explicit form for numerical computing. Examples of such $\Omega$ include the unit ball, box, non-negative orthant, specific affine subspaces, and so on. However, in the next subsection, we will show that the closed form of the projection is unnecessary to establish the local convergence rate of the proposed algorithm. 
\end{itemize}
\end{remark}
\subsection{Local convergence rate}
 Given a \ts\ $\bfx^*$ of \eqref{SCP}, there is $\W^*$ satisfying condition (\ref{eta-point}). Hereafter, we always denote $\bfw^*:=(\bfx^*; \ve(\W^*))$ and  $\bfw:=(\bfx; \ve(\W))$. 
\begin{lemma}\label{lemma-neighbour} Let $\bfx^*$ be a \ts\ with $0<\tau<\tau_*$ of (\ref{SCP}). Then there is a neighbourhood ${\mathbb N}^*$ of $\bfw^*$ such that,
 \beq\label{gw*-0} \F(\bfw^*;U_T)={\bf 0}~~\text{and}~~U_T\subseteq\J_*,~~\forall \bfw \in {\mathbb N}^*,~\forall~T\in\T(\L ;s). \eeq
 \end{lemma}
 To establish the local convergence performance, we need the following assumptions.
\begin{assumption}\label{ass-1} 
Suppose that $\varphi(\cdot;\J_*)$ is semismooth at $\bfw^*:=(\bfx^*; \ve(\W^*))$  and 
\begin{eqnarray}\label{second-order-cond}
 \left[  \eqspace{1.25}
     \begin{array}{ccc}
       \P^* & ~~& \Q^*\\
    \nabla_{ \J} \G(\bfx^*)^\top&~~&{\bf0}
     \end{array}
     \right],
  \end{eqnarray}
is nonsingular for any $\J\subseteq\J_*$ and any $\H^*=[  \P^*~\Q^*~{\bf0}]\in  \partial \varphi(\bfw^*;\J) $. 
\end{assumption}
These assumptions are related to the regularity conditions \cite{qi1993nonsmooth, robinson1980strongly, dontchev2010newton, zhou2021quadratic} usually used to achieve the convergence results for Newton-type methods. 
  \begin{lemma}\label{lemma-semismooth} Let $\bfw^*$ be a \ts\ with $0<\tau<\tau_*$ of (\ref{SCP}). If  $\varphi(\cdot;\J_*)$ is semismooth at $\bfw^*$, then so is $\F(\cdot;\J)$ for any $\J\subseteq\J_*$. If  $\varphi(\cdot;\J_*)$ and $G_{mn}, (m,n)\in\J_*$ are strongly semismooth at $\bfw^*$, then so is $\F(\cdot;\J)$ for any $\J\subseteq\J_*$.
 \end{lemma}
 \proof{Proof} Since $G_{mn}$ for any $(m,n)$ are continuously differentiable, they are semismooth everywhere on $\R^K$.  Then the results follow from the definition of $\F$ in \eqref{equation-F-mu} and  \cite[Proposition 1.73]{izmailov2014newton}.
\Halmos\endproof
\begin{remark} We have the following comments on the assumptions in Lemma \ref{lemma-semismooth}.   As projection $\Pi_{\Omega}$ is semismooth,  $\varphi(\cdot;\J_*)$ is semismooth at $\bfx^*$ if $\nabla f$ and $\nabla G_{mn}, (m,n)\in\J_*$ are semismooth based on \cite[Proposition 1.74]{izmailov2014newton}. In this regard, the assumption on the semismoothness of $\varphi(\cdot;\J_*)$ is mild. Moreover, let ${\bf z}^*:=\bfx^*-\tau[\nabla f(\bfx^*)+ \sum_{ \J_* }W^*_{mn}   \nabla  G_{mn}(\bfx^*)]$, if  $\nabla f$ and $\nabla G_{mn}, (m,n)\in\J_*$ are strongly semismooth at $\bfx^*$, and $\Pi_{\Omega}$ is strongly  semismooth at ${\bf z}^*$ (e.g., when  $\Omega$ is a symmetric cone or polyhedral set, or ${\bf z}^*$ belongs to the interior of $\Omega$),  then $\varphi(\cdot;\J_*)$  is  strongly semismooth at $\bfx^*$, resulting in the strongly semismoothness of $\F(\cdot;\J)$ at $\bfx^*$ for any $\J\subseteq\J_*$ by \cite[Proposition 1.74]{izmailov2014newton}. 
\end{remark}
 \begin{lemma}\label{lemma-neighbour-bd}  If Assumptions \ref{ass-1} holds, then there is $\mu_*>0$ and a neighbourhood $\mathbb{N}^*$ of $\bfw^*$ such that given any $\mu\in[0, \mu_*)$ and $\J\subseteq\J_*$,  matrix
  \begin{eqnarray*}
 \nabla \F_{\mu}(\bfw;{\J})=\left[  \eqspace{1.25}
     \begin{array}{ccccc}
       \P &~~&   \Q&~~&{\bf0}\\
    \nabla_{ \J} \G(\bfx)^\top&&-\mu {\bf I}_{|  \J|}&~~&{\bf0}\\
     {\bf0}&&{\bf0}&~~&{\bf I}_{|\overline \J|}\\
     \end{array}
     \right]
  \end{eqnarray*}
is non-singular. Moreover, both $\|(\nabla \F_\mu(\bfw;\J))^{-1}\|$ and $\|\nabla \F_\mu(\bfw;\J)\|$ are bounded for any $\H=[  \P~\Q~{\bf0}]\in  \partial \varphi(\bfw;\J)$ and any $\bfw\in \mathbb{N}^*$. 
\end{lemma}
\proof{Proof} We first claim that for any fixed $\J\subseteq\J_*$,
  \begin{eqnarray*}
{\bf\Phi}(\bfw):=\left[  \eqspace{1.25}
     \begin{array}{ccc}
       \P &~~&   \Q\\
  \nabla_{ \J} \G(\bfx)^\top&& {\bf 0}\\
     \end{array}
     \right],
  \end{eqnarray*}
is non-singular  for any $\H=[  \P~\Q~{\bf0}]\in  \partial \varphi(\bfw;\J)$ and any $\bfw\in \mathbb{N}^*$.  Suppose this is not true, then there is a sequence $\bfw^k\to\bfw^*$, $\H^k=[  \P^k~\Q^k~{\bf0}]\in  \partial \varphi(\bfw^k;\J)$,  all $\{{\bf\Phi}(\bfw^k)\}$ are singular. Since $\varphi(\cdot;\J)$  is semismooth around $\bfw^*$,  it is locally Lipschitz continuous. This implies that $\partial \varphi(\cdot;\J)$ in bounded in  a neighbourhood $\mathbb{N}^*$ of $\bfw^*$. Therefore, sequence $\{\H^k\}$  is bounded. By passing to a subsequence, we may assume $  \H^k \to \H^\infty=[  \P^\infty~\Q^\infty~{\bf0}]$. By the closedness of the generalized Jacobian, we have $\H^\infty\in\partial \varphi(\bfw^*;\J)$, which together with $\nabla_{ \J} \G(\bfx^k)\to\nabla_{ \J} \G(\bfx^*)$ means that 
  \begin{eqnarray*}
{\bf\Phi}(\bfw^k)\to \left[  \eqspace{1.25}
     \begin{array}{ccc}
       \P^\infty &~~&   \Q^\infty\\
     \nabla_{ \J} \G(\bfx^*)^\top&& {\bf 0}\\
     \end{array}
     \right]=:{\bf\Phi}(\bfw^*)
  \end{eqnarray*}
and thus ${\bf\Phi}(\bfw^*)$ is singular, contradicting with Assumption  \ref{ass-1}. Now for fixed $\J$, let $\sigma(\J)$ be the smallest singular value of ${\bf\Phi}(\bfw)$ over all $\H\in  \partial \varphi(\bfw;\J)$ and all $\bfw\in \mathbb{N}^*$. The above claim indicates $ \sigma(\J)>0$. Then for any $\mu\in[0,\sigma(\J))$, the following matrix
  \begin{eqnarray*}
\left[  \eqspace{1.25}
     \begin{array}{ccc}
       \P &~~&   \Q\\
     \nabla_{ \J} \G(\bfx)^\top&& -\mu {\bf I}_{|  \J|}\\
     \end{array}
     \right]=\left[  \eqspace{1.25}
     \begin{array}{ccc}
       \P &~~&   \Q\\
     \nabla_{ \J} \G(\bfx)^\top&&  {\bf 0}\\
     \end{array}
     \right]+\left[  \eqspace{1.25}
     \begin{array}{ccc}
      {\bf 0} &~~&  {\bf 0}\\
     {\bf 0 }&& -\mu {\bf I}_{|  \J|}\\
     \end{array}
     \right]
  \end{eqnarray*}
  is still nonsingular, so is matrix $\nabla \F_{\mu}(\bfw;{\J})$. Let $\mu_*:=\min_{\J\subseteq\J_*}\sigma(\J)$. As the choices of $\J\subseteq\J_*$  are finitely many, $\mu_*>0$ is well defined. Then we can conclude that for any $\mu\in[0, \mu_*)$ and any fixed $\J\subseteq\J_*$,  matrix $\nabla \F_{\mu}(\bfw;{\J})$ is non-singular for any $\H\in  \partial \varphi(\bfw;\J)$ and any $\bfw\in \mathbb{N}^*$.  Finally, one can observe that the smallest singular value of $\nabla \F_{\mu}(\bfw;{\J})$  is no less than $ \min\{1, \sigma(\J)-\mu\}$ for each fixed $\J$. Hence for all $\J\subseteq\J_*$, the smallest singular value is no less than  $\min\{1, \mu_*-\mu\}$, this implies  $\|(\nabla \F_\mu(\bfw;\J))^{-1}\|\leq 1/\min\{1, \mu_*-\mu\}$, namely, it is bounded. The boundedness of $\|\nabla \F_\mu(\bfw;\J)\|$ follows from the boundedness of $\partial \varphi(\bfw;\J)$  and $ \nabla_{ \J} \G(\bfx)$ for any $\bfw\in \mathbb{N}^*$. \Halmos\endproof
We are ready to establish the local convergence rate of the proposed smooth Newton method, \nscp, before which we point out that the establishment is not trivial, because differing from the standard system of equations, the $\tau$-stationary equations, \eqref {sta-eq-1}, involve an unknown set $\J_*=U_{ \Gamma_0^*}$. If this set is available in advance, then one can follow the standard way for the Newton-type methods to build the local convergence rate. However, set $U_{ T_\ell}$ may change from one iteration to another. A different set leads to a different
system of equations $\F(\bfw; U_{ T_\ell}) = {\bf0}$. Hence, in each step, the algorithm finds a Newton direction for a
different system of equations instead of a fixed system. This is where the standard proof for quadratic
convergence fails to fit our case. 
\begin{theorem}[Local convergence rate]\label{the:quadratic} Suppose Assumption \ref{ass-1} holds and set $\gamma\geq|\Gamma^*_0|/s$. Then in a neighborhood of $\bfw^*$, Algorithm \ref{Alg-NL01} admits full Newton steps and converges to $\bfw^*$ superlinearly.  If further assume $\varphi(\cdot;\J)$  and $G_{mn}, (m,n)\in\J_*$ are strongly semismooth at $\bfw^*$, then the convergence is quadratic.
\end{theorem}
 \proof{Proof}  
Consider a neighbourhood  ${\mathbb N}^*(\bfw^*,\eta_*)$ with a sufficiently small radius $\eta_*>0$. For $\bfw^\ell\in {\mathbb N}^*(\bfw^*,\eta_*)$, we have $\J_{\ell} := U_{T_{\ell}}\subseteq\J_*$ due to  \eqref{gw*-0}, where $T_\ell\in\T_\tau(\L ^\ell;s)$. By Lemma \ref{lemma-neighbour}, we have
\beq\label{gw*-bd-h}\F (\bfw^{*};\J_\ell) \overset{\eqref{gw*-0}}{=}{\bf 0}.
\eeq
which indicates that $\F (\bfw^\ell;\J_\ell)$ is near zero as $\bfw^\ell$ close to $\bfw^*$. This is because $\F (\cdot;\J_{\ell})$ is locally Lipschitz continuous around $\bfx^*$ for given $\J_{\ell}$. Now from Lemma \ref{lemma-neighbour-bd}, ${\bf\Phi}_{\mu_\ell}^\ell:=\nabla \F_{\mu_\ell} (\bfw^{\ell};\J_\ell)$ is non-singular and $\|({\bf\Phi}_{\mu_\ell}^\ell)^{-1}\|$ is bounded, and thus \eqref{newton-dir-0} is solvable, namely, $\bfd^{{\ell}}=-({\bf\Phi}_{\mu_\ell}^\ell)^{-1} \F (\bfw^\ell;\J_\ell) $.  
 Therefore, $\bfd^{{\ell}}$  is close to zero. This allows us to derive the following conditions
 \begin{eqnarray*} 
  \eqspace{1.35}
     \begin{array}{rll}
\forall~n\in\Gamma_+^*:~~(\G(\bfx^*))_{:n}^{\max} >0~~\Longrightarrow ~~(\G(\bfx^{\ell} +  \bfd^{\ell}_\bfx))_{:n}^{\max}>0,\\
\forall~n\in\Gamma_-^*:~~(\G(\bfx^*))_{:n}^{\max} <0~~\Longrightarrow ~~(\G(\bfx^{\ell} +  \bfd^{\ell}_\bfx))_{:n}^{\max}<0,
     \end{array} 
  \end{eqnarray*}
which further result in
 \begin{eqnarray*} 
  \eqspace{1.35}
     \begin{array}{rll}
     \|\G(\bfx^{0} +  \bfd^{0}_\bfx)\|^+_0  \leq |\Gamma_+^*| + |\Gamma_0^*| 
      =  \|\G(\bfx^*)\|^+_0 + |\Gamma_0^*| 
      \leq  s + \gamma s. 
     \end{array} 
  \end{eqnarray*}
The above condition indicates \eqref{gradually-feasible} is satisfied with $ \pi^{t_\ell}=1$. Hence, the full Newton is admitted, namely, $\bfw^{\ell+1} = \bfw^{\ell} +  \bfd^{\ell}.$ 
  By the definition of ${\bf\Phi}_{\mu_\ell}^\ell$ and \eqref{sta-eq-2}, it follows ${\bf\Phi}_{0}^\ell= \nabla \F (\bfw^{\ell};\J_\ell) $ and
   \begin{eqnarray}\label{super-rate}
\eqspace{1.35}
\begin{array}{lcl}
&&\|\bfw^{\ell+1}-\bfw^{*}\|  =\|\bfw^{\ell}+\bfd^{\ell}-\bfw^{*}\| \\
& = & \|({\bf\Phi}_{\mu_\ell}^\ell)^{-1} \cdot {\bf\Phi}_{\mu_\ell}^\ell \cdot(\bfw^{\ell}+\bfd^{\ell}-\bfw^{*})\|\\
  &\overset{\eqref{newton-dir-0}}{=} &  \|({\bf\Phi}_{\mu_\ell}^\ell)^{-1} \cdot[{\bf\Phi}_{\mu_\ell}^\ell \cdot(\bfw^{\ell} -\bfw^{*}) - \F (\bfw^{\ell};\J_\ell)]\|\\
   &\overset{\eqref{gw*-0}}{=} &  \|({\bf\Phi}_{\mu_\ell}^\ell)^{-1}  \cdot[\F (\bfw^{\ell};\J_\ell) - \F (\bfw^*;\J_\ell)  - {\bf\Phi}_{\mu_\ell}^\ell \cdot(\bfw^{\ell} -\bfw^{*})]\|\\
  &=&  \|({\bf\Phi}_{\mu_\ell}^\ell)^{-1}  \cdot[\F (\bfw^{\ell};\J_\ell) - \F (\bfw^*;\J_\ell)  - {\bf\Phi}_{0}^\ell \cdot(\bfw^{\ell} -\bfw^{*}) + ({\bf\Phi}_{0}^\ell -{\bf\Phi}_{\mu_\ell}^\ell ) \cdot(\bfw^{\ell} -\bfw^{*})]\|\\
 & =& o(\|\bfw^{\ell}-\bfw^{*}\|),
\end{array} 
 \end{eqnarray}
 where the last equation holds due to three facts:   semismoothness of $\F (\cdot;\J_\ell)$ at $\bfw^*$ for a fixed $\J_{\ell}$, boundedness of $({\bf\Phi}_{\mu_\ell}^\ell)^{-1}$, and 
    \begin{eqnarray}\label{mu-w-w}
\eqspace{1.35}
\begin{array}{lcl}
\|({\bf\Phi}_{0}^\ell -{\bf\Phi}_{\mu_\ell}^\ell ) \cdot(\bfw^{\ell} -\bfw^{*})\|= \mu_\ell \|\W^{\ell}_{\J_{\ell}}-\W^{*}_{\J_{\ell}}\| \leq   \mu_\ell\|\bfw^{\ell}-\bfw^{*}\| =  o(\|\bfw^{\ell}-\bfw^{*}\|).
\end{array} 
 \end{eqnarray}
 Here the first equation is from definition \eqref{F-mu} of $\nabla \F_{\mu} (\bfw;\J)$. Equation \eqref{super-rate} shows the superlinear rate. To see the quadratic rate, from \eqref{mu-w-w}, we obtain
  \begin{eqnarray*} 
\eqspace{1.35}
\begin{array}{lcl}
\|({\bf\Phi}_{0}^\ell -{\bf\Phi}_{\mu_\ell}^\ell ) \cdot(\bfw^{\ell} -\bfw^{*})\|  &\leq&   \mu_\ell \|\bfw^{\ell}-\bfw^{*}\| \\
&\overset{\eqref{update-mu}}{\leq} &  \rho \|\F (\bfw^\ell;\J_\ell)\|  \cdot\|\bfw^{\ell}-\bfw^{*}\|  \\
&\overset{\eqref{newton-dir-0}}{=} &  \rho \|{\bf\Phi}_{\mu_\ell}^\ell\cdot \bfd^{{\ell}}\| \cdot \|\bfw^{\ell}-\bfw^{*}\|  \\ 
&\leq&  \rho \|{\bf\Phi}_{\mu_\ell}^\ell\|\cdot \|\bfw^{\ell+1}-\bfw^{\ell}\| \cdot \|\bfw^{\ell}-\bfw^{*}\|  \\ 
&\leq&  \rho \|{\bf\Phi}_{\mu_\ell}^\ell\|\cdot (\|\bfw^{\ell+1}-\bfw^{*}\|+ \|\bfw^{\ell}-\bfw^{*}\|) \cdot \|\bfw^{\ell}-\bfw^{*}\| \\
&\leq& 1/(2\|({\bf\Phi}_{\mu_\ell}^\ell)^{-1}\|)  \|\bfw^{\ell+1}-\bfw^{*}\|+ \rho \|{\bf\Phi}_{\mu_\ell}^\ell\|\cdot \|\bfw^{\ell}-\bfw^{*}\|^2,
\end{array} 
 \end{eqnarray*} 
 the last inequality is ensured for sufficiently small $\eta_*\leq 1/(2\rho \|{\bf\Phi}_{\mu_\ell}^\ell\|\|({\bf\Phi}_{\mu_\ell}^\ell)^{-1}\|)$ and $\bfw^\ell\in {\mathbb N}^*(\bfw^*,\eta_*)$. The above condition, strong semismoothness of $\F (\cdot;\J_\ell)$ at $\bfw^*$ for a fixed $\J_{\ell}$,  and \eqref{super-rate} deliver
 \begin{eqnarray*} 
\eqspace{1.35}
\begin{array}{lcl}
 \|\bfw^{\ell+1}-\bfw^{*}\| 
  &\leq &  O(\|\bfw^{\ell}-\bfw^{*}\|^2) + \| ({\bf\Phi}_{\mu_\ell}^\ell)^{-1}  \cdot ({\bf\Phi}_{0}^\ell -{\bf\Phi}_{\mu_\ell}^\ell ) \cdot(\bfw^{\ell} -\bfw^{*})]\|\\
  &\leq &  O(\|\bfw^{\ell}-\bfw^{*}\|^2) + (1/2 ) \|\bfw^{\ell+1}-\bfw^{*}\|+ \rho \|{\bf\Phi}_{\mu_\ell}^\ell\|\cdot \|\bfw^{\ell}-\bfw^{*}\|^2,
\end{array} 
 \end{eqnarray*}
 which suffices to $ \|\bfw^{\ell+1}-\bfw^{*}\|  = O(\|\bfw^{\ell}-\bfw^{*}\|^2)$, showing the quadratic convergence rate. 
  \Halmos\endproof
%
%

\section{Numerical Experiments}\label{sec:numerical}

In this section, we will conduct some numerical experiments of  \nscp\ ((available at \url{https://github.com/ShenglongZhou/SNSCO})) using MATLAB (R2023b) on a laptop of  $64$GB memory and Core i9.  

\subsection{Test example}\label{subsec:nop}
 We use the norm optimization problem described in \cite{hong2011sequential,adam2016nonlinear} to demonstrate the performance of the selected algorithms. The problem takes the following form,
\begin{equation}\label{NOP-orig}
\eqspace{1.35}
\begin{array}{cll}
\underset{\bfx\in \R^{K}}{\min}&&  -\langle{\bf 1}, \bfx\rangle,\\
{\rm s.t.}&& \mathbb{P}\Big\{\frac{1}{2} \langle \bfxi^m\circ\bfxi^m, \bfx\circ\bfx\rangle  \leq b, ~m\in\M\Big\}\geq 1-\alpha,~ \bfx\geq  {\bf0}, 
\end{array}\end{equation}
where    $\circ$ represents the Hadamard product,  ${\bf 1}$ is the vector with all entries being ones,  $b>0$, and $\bfxi^m=(\xi_{1}^m,\ldots,\xi_{K}^m)^\top\in\R^K$. 
However, in the sequel, we aim to solve the following problem,
\begin{equation}\label{NOP}
\eqspace{1.35}
\begin{array}{cll}
\underset{\bfx\in \R^{K}}{\min}&& -\langle{\bf 1}, \bfx\rangle+\frac{\lambda}{2}\|\bfx\|^2,\\
{\rm s.t.}&& \mathbb{P}\Big\{\frac{1}{2}\sum_{k=1}^K(\xi_{k}^m)^2x_k^2\leq b, ~m\in\M\Big\}\geq 1-\alpha,~ \bfx\geq  {\bf0}, 
\end{array}\end{equation} 
where $\lambda$ is a positive penalty scalar. The above two problems admit the same optimal solution if $\lambda$ is smaller than a threshold and   $\xi_{k}^m,m\in\M,k\in\K$ are independent and identically distributed (i.i.d.) standard normal random variables.  To see this, define
\begin{equation} \label{opt-solution} 
\eqspace{1.35}
\begin{array}{cll}
\bfx^{\rm opt}:=  c  {\bf 1}~~ \text{with}~~ c:=\Big[\frac{2b}{F_{\chi_K^2}^{-1}\left((1-\alpha)^{1/M}\right)}\Big]^{1/2},
\end{array}\end{equation}
 where $F_{\chi_K^2}^{-1}$ denotes the inverse distribution function of a chi-square distribution with $K$ degrees of freedom. Then it follows from \cite{hong2011sequential} that the optimal solution to problem \eqref{NOP-orig} is $\bfx^{\rm opt}$. Similar reasoning to \cite{hong2011sequential} allows us to derive that the optimal solution  to problem \eqref{NOP} is $\min \left\{  {1}/{\lambda}, c   \right\}{\bf 1}$, which is identical to $\bfx^{\rm opt}$ if $0<\lambda\leq 1/c$. Therefore, for simplicity, in the subsequent numerical experiments, we set $\lambda=1/{(2c)}$ and $b=5.$  

 Let $\{\bfxi_n^m:=(\xi_{n1}^m,\ldots,\xi_{nK}^m)^\top\in\R^K: n\in\N\}$ be $N$ realizations of $\bfxi^m$ and denote
\begin{equation*}
 T_+:=\{k\in\K: z_k>0\},\qquad T_0:=\{k\in\K: z_k=0\}, \qquad T_-:=\{k\in\K: z_k<0\},
\end{equation*}
where ${\bf z}:=\bfx-\tau[\nabla f(\bfx)+ \sum_{ \J }W_{mn}   \nabla  G_{mn}(\bfx)]$.  Define
\begin{eqnarray*}
{\bf \Theta}&:=&\left[
\begin{array}{rrrr} 
  \bfxi_{1}^{1}\circ\bfxi_{1}^{1}, \ldots, \bfxi_{1}^{M}\circ\bfxi_{1}^{M}, &\bfxi_{2}^{1}\circ\bfxi_{2}^{1}, \ldots, \bfxi_{2}^{M}\circ\bfxi_{2}^{M}, &\ldots, \bfxi_{N}^{1}\circ\bfxi_{N}^{1} , \ldots, \bfxi_{N}^{M}\circ\bfxi_{N}^{M}
\end{array} \right]^\top\in\R^{K\times MN},\\[1.5ex]
{\cal I}(\J)&:=&\{i\in\{1,2,\ldots,MN\}:~i=m+(n-1)M,~\forall~(m,n)\in\J \}.
\end{eqnarray*}
 As $G_{mn}(\bfx)=(1/2)\langle \bfxi_n^m \circ\bfxi_n^m, \bfx\circ\bfx \rangle -b$, one can obtain
\begin{eqnarray*}
 \nabla_{\J}\G(\bfx)= {\rm D}(\bfx){\bf \Theta}_{: {\cal I}(\J)}.
\end{eqnarray*} 
Moreover,  any $\H=[\P~\Q~{\bf 0}]\in\partial \varphi(\bfw,\J)$ takes the form of
\begin{eqnarray*}
\begin{array}{l}
\P=\left[
\begin{array}{ccc} 
\P_{T_+T_+} & {\bf 0}& {\bf 0}\\[1ex]
{\bf 0} & \P_{T_0T_0}({\bf t}) & {\bf 0}\\[1ex]
{\bf 0} & {\bf 0} &  {\bf I}_{|T_-|}\\ 
\end{array} 
\right],~
 \Q =  \left[
\begin{array}{r} 
 \tau  {\rm D}(\bfx_{T_+}){\bf \Theta}_{T_+{\cal I}(\J)} \\[1ex]
\tau  {\rm D}({\bf t}\circ\bfx_{T_0}){\bf \Theta}_{T_0{\cal I}(\J)} \\[1ex]
 {\bf 0}~\\ 
\end{array} \right],\\[5ex]
\P_{T_+T_+}=\tau{\rm D}\left( {\bf \Theta}_{T_+{\cal I}(\J)} {\rm vec}(\W_{\J})\right) ,\\[2ex]
\P_{T_0T_0}({\bf t})={\rm D}\left({\bf 1}-{\bf t}+\tau{\bf t}\circ [{\bf \Theta}_{T_0{\cal I}(\J)} {\rm vec}(\W_{\J}) ]+\lambda{\bf1}\right),
\end{array}\end{eqnarray*} 
where ${\bf t}\in[{\bf 0},{\bf 1}]\subset\R^{|T_0|}$. Based on these calculations, we have the following corollary.
\begin{corollary}\label{cor-norm-opt} Let $\bfw^*$  be a \ts\ of problem (\ref{SCP}) in the case of model (\ref{NOP}) and ${\bf z}^*:=\bfx^*+\tau[{\bf 1}-{\rm D}(\bfx^*){\bf \Theta}_{:{\cal I}(\J_*)}{\rm vec}(\W^*_{\J_*})]$ with $T^*_+:=\{k\in\K: z_k^*>0\}$.  Suppose ${\bf \Theta}_{T^*_+{\cal I}(\J_*)}$ is full column rank. Then in a neighborhood of $\bfw^*$, Assumption \ref{ass-1} holds and Algorithm \ref{Alg-NL01} converges to $\bfw^*$  quadratically.
\end{corollary}
\proof{Proof} As $f(\bfx)=-\langle{\bf 1}, \bfx\rangle+(\lambda/2)\|\bfx\|^2$, $G_{mn}(\bfx), (m,n)\in\M\times\N$, and $\Pi_{\Omega}$  are all strongly semismooth, $\varphi(\cdot;\J_*)$ is  strongly semismooth. Therefore, we only need to verify the nonsingularity of the matrix in \eqref{second-order-cond} for any $\J\subseteq\J_*$ and any $\H^*\in  \partial \varphi(\bfw^*;\J)$. By the definition of $T^*_+$, we have $\bfx^*_{T^*_+} >{\bf 0}$ and $\bfx^*_{\overline T^*_+}={\bf 0}$ due to $\bfx^*=\max\{{\bf 0}, {\bf z}^*\}$ from \eqref{eta-point}. Hence the matrix in \eqref{second-order-cond} reduces to
\begin{eqnarray*}
\left[
\begin{array}{ccccc} 
\P^*_{T^*_+T^*_+}  & {\bf 0}& {\bf 0}&\tau  {\rm D}(\bfx^*_{T^*_+}){\bf \Theta}_{T^*_+{\cal I}(\J)} \\[1ex]
{\bf 0} & \P^*_{T^*_0T^*_0}({\bf t})  & {\bf 0}&{\bf 0} \\[1ex]
{\bf 0} & {\bf 0} &  {\bf I}_{|T^*_-|} & {\bf 0}  \\ [1ex]
({\rm D}(\bfx^*_{T^*_+}){\bf \Theta}_{T^*_+{\cal I}(\J)})^\top&{\bf 0}&{\bf 0} &   {\bf 0}\\[1ex]
\end{array} 
\right],
\end{eqnarray*} 
where  $\P^*_{T^*_+T^*_+}:=\tau{\rm D} ( {\bf \Theta}_{T^*_+{\cal I}(\J)} {\rm vec}(\W^*_{\J})+ \lambda {\bf 1} ) $, $
\P^*_{T^*_0T^*_0}({\bf t}):={\rm D} ({\bf 1}-{\bf t}+\tau{\bf t}\circ [{\bf \Theta}_{T^*_0{\cal I}(\J)} {\rm vec}(\W^*_{\J})+\lambda {\bf 1}] )$, and $T^*_0:=\{k\in\K: z_k^*=0\}, T^*_-:=\{k\in\K: z_k^*<0\}$. It follows from  \eqref{eta-point} and \eqref{uPu-equ} that $\W^*\geq0$, resulting in the positive definiteness of two diagonal matrices $\P^*_{T^*_+T^*_+}$ and $\P^*_{T^*_0T^*_0}({\bf t})$. With the help of the full column rankness of ${\bf \Theta}_{T^*_+{\cal I}(\J_*)}$,  its sub-matrix ${\bf \Theta}_{T^*_+{\cal I}(\J)}$ is full column rank for any $\J\subseteq\J_*$, leading to the nonsingularity of the above matrix. The remaining proof follows Theorem \ref{the:quadratic}.
  \Halmos\endproof
 
 \begin{figure}[!th]
 \centering
 \includegraphics[width=1.0\textwidth]{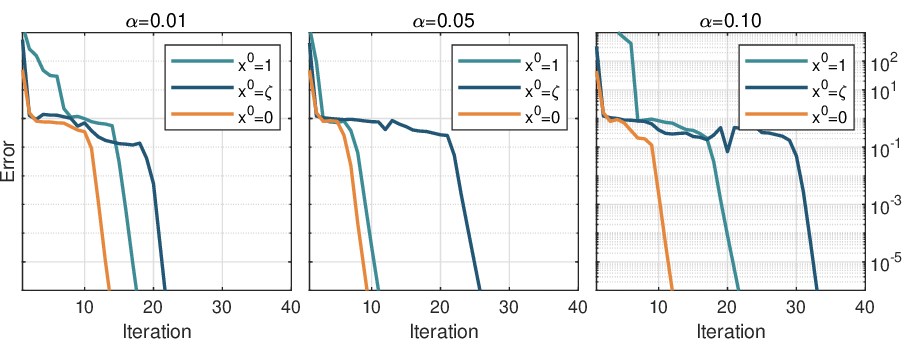}\\
 \includegraphics[width=1.0\textwidth]{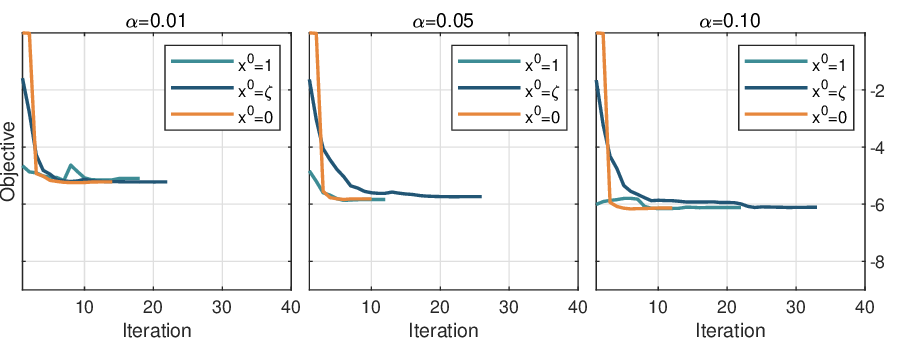} 
  \caption{Effect of the starting points.\label{fig:qcr}}\vspace{-0mm}
 \end{figure}
\subsection{Implementations} The parameters for \nscp\ are set as follows:  $\texttt{maxIt}=10^5, \texttt{tol}= 10^{-6}MN, \rho=100,  \underline\mu=0.125$, $\nu=0.9995,$   $\gamma=1$ if $s=1$ and $\gamma=1+2/s$ if $s>1$,  $\pi=0.5$, and   $\W^0=50\cdot{\bf1}$.  In addition, we set $s=\lceil \alpha N\rceil$ where $\alpha$ is chosen from $\{0.01,0.05,0.1\}$.  To observe the impact of starting points on the proposed algorithm, we use three choices of $\bfx^0$ when solving \eqref{NOP} with $(K,M,N)=(10,1,100)$. The three starting points are presented in Fig. \ref{fig:qcr}, where  ${\boldsymbol \zeta}$ has entries randomly sampled from $[0,1]$.  It can be seen from that all error $\|\F(\bfw^{\ell},U_{T_{\ell}})\|$ declines quickly once the iteration surpasses a certain threshold, demonstrating a fast convergence rate as stated by Corollary \ref{cor-norm-opt}. Since  $\bfx^0={\bf1}$ appears to yield the best overall performance, we fix this initialization for all subsequent numerical experiments.
 
\subsection{Benchmarks}
 We will compare \nscp\ with an {\tt NLP} algorithm developed in \cite{pena2020solving}, two algorithms proposed in \cite{adam2016nonlinear}: regularized algorithm with a convex start  (\reg) and relaxed algorithm (\rel), and  {\tt GUROBI}. All algorithms use the same initial point, $\bfx^0={\bf1}$. Other relevant parameters and implementations are described as follows: 
 \begin{itemize}
\item For {\tt GUROBI}, we  employ it to solve problem \eqref{BIP}. However, it is challenging for {\tt GUROBI} to operate the constraints in $\CD(\bfy)$. So we replace it with the so-called ``big-M" constraints,
\begin{eqnarray*} 
 \eqspace{1.35}
 \begin{array}{lll}
  \left\{   \bfx \in \R^{K} : ~ 
  G_{mn}(\bfx)  \leq (1-y_n) M_n,~\forall~m\in\M,~ n\in\N \right\},
 \end{array}
 \end{eqnarray*}
 where $M_n,n\in\N$ are big positive scalars.  ``Big-M" constraints has been widely used in  \cite{ahmed2008solving, song2014chance, luedtke2014branch}. This value significantly impacts the computational speed of {\tt GUROBI}, and a tighter bound of $M_n$ may lead to a better result \cite{song2014chance, adam2016nonlinear}. Therefore, we set $M_n$ using three different choices.  Moreover, similar to \cite{pena2020solving} to avoid too long computational time, we set a time limit, namely, we terminate {\tt GUROBI} if the consumed time exceeds 10 minutes.  Overall, we include three following variants of  {\tt GUROBI} in the numerical comparison:
 \begin{itemize}
 \item  {\tt GUROBI}${10^2}$:  {\tt GUROBI} under $M_n=10^2$ for any $n\in\N$;
 \item  {\tt GUROBI}${10^3}$: {\tt GUROBI} under $M_n=10^3$ for any $n\in\N$;
 \item  {\tt GUROBI}${10^4}$: {\tt GUROBI} under $M_n=10^4$ for any $n\in\N$.
\end{itemize}  
 \item For {\tt NLP},   there is an approximation parameter $\epsilon$ that has a big influence on the solution quality.   As suggested by \cite{pena2020solving}, this parameter should be adaptively updated by a binary search algorithm. However, we adopt a more effective continuation approach as follows: First, initialize   $P$ values of $\{ \epsilon_1,\epsilon_1,\ldots,\epsilon_{P}\}$ and denote $\bfx(\epsilon_1):=\bfx^0=\bf1$. Then for each step $t=1,2,3,\ldots,P$, under $\epsilon_t$ and initial point $\bfx(\epsilon_{t})$, we employ {\tt NLP} to solve the problem to obtain a solution $\bfx^{t+1}$ and denote $\bfx(\epsilon_{t+1}):=\bfx^{t}$ which is used for step $t+1$.  Finally, we pick the best solution in terms of the lowest objective function values, namely, $(\bfx^{\rm best},\epsilon^{\rm best})={\rm argmin} \{ f(\bfx(\epsilon_{1})),\ldots,f(\bfx(\epsilon_{P}))\}$. This approach is known as the algorithm with continuation \cite{ma2011fixed, hale2008fixed, huang2018robust, adam2016nonlinear}. Therefore, we denote {\tt NLP} with continuation by {\tt NLPC} and report the results of the following two algorithms:
   \begin{itemize}
 \item  {\tt NLPC}: {\tt NLP} with continuation, where  $\epsilon_1=0.1,\epsilon_P=0.001$, and $\epsilon_{t+1}=\epsilon_t(\epsilon_P/\epsilon_1)^{1/(P-1)}$. To accelerate the computation, we set $P=20$ if $M=1$ and $P=10$ otherwise.
 \item  {\tt NLP}: {\tt NLP} under the best approximation parameter, i.e., $\epsilon^{\rm best}$ selected by {\tt NLPC}.
\end{itemize} 
 \item  For \nscp, our main theory involves a parameter $\tau$, which plays a critical role in deciding the solution quality.  Similar to {\tt NLPC}, we set a series of values of $\tau$. However, differing from {\tt NLPC} using $\bfx(\epsilon_{t+1})=\bfx^{t}$ as the initial point for next step $t+1$, we always use $(\bfx^0,\W^0)$ as the initial point for each $\tau_t$. Overall, we report the results of the following algorithms:
    \begin{itemize}
 \item  {\tt SNSCOC}: \nscp\ with continuation. We let $\tau_1=0.5, \tau_P=1.75$, and $\tau_{t+1}=\tau_t(\tau_P/\tau_1)^{1/(P-1)}$, and set $P=25$ if $\alpha=0.01$ and $P=50$ otherwise.
 \item  \nscp: \nscp\ under the best $\tau$ selected by  {\tt SNSCOC}.
\end{itemize} 
 \end{itemize} 
 In the sequel, we choose $\alpha\in\{0.01,0.05,0.10\}$. For each fixed combination of $(\alpha, K, N, M)$, we run 20 trials and record two metrics: objective function value $f(\bfx)$ and computational time in seconds. Each reported factor represents the median of 20 values.

\subsection{Single CCP with i.i.d. data}
 We first employ all algorithms to solve problem  \eqref{NOP} with  $M=1$ and i.i.d. standard normal random variables $\xi_{k}^m,m\in\M,k\in\K$. We have also tested all algorithms to solve problems with non-i.i.d. random variables $\xi_{k}^m$, but omitted the comparison results as they are similar to these for i.i.d. cases. 
 
    \begin{table}[H]
	\renewcommand{\arraystretch}{1.0}\addtolength{\tabcolsep}{0.8pt}
	\caption{Effect of $K$ for the single CCP with $N=100$ and $K\in\{10,30,50\}$.\label{table:k-single}} 
	\centering
		\begin{tabular}{lrrrrrrrrrrrr}
			\hline 

	 	&	\multicolumn{3}{c}{$\alpha=0.01$} &&	\multicolumn{3}{c}{$\alpha=0.05$}	 	 	&&	\multicolumn{3}{c}{$\alpha=0.10$}	 		&		\\\cline{2-4}\cline{6-8}\cline{10-12}
$K$	&	$10$	&	$30$	&	$50$	&&	$10$	&	$30$	&	$50$	&&	$10$	&	$30$	&	$50$	\\\hline 
	&	\multicolumn{11}{c}{$f(\bfx)$}\\\hline
  
  {\tt GUROBI}$10^2$	&	-5.281 	&	-10.821 	&	-14.462 	&&	-5.781 	&	-11.380 	&	-15.112 	&&	-6.206 	&	-11.765 	&	-15.553 	\\
{\tt GUROBI}$10^3$	&	-5.281 	&	-10.821 	&	-14.462 	&&	-5.781 	&	-11.380 	&	-15.112 	&&	-6.206 	&	-11.765 	&	-15.553 	\\
{\tt GUROBI}$10^4$	&	-5.281 	&	-10.821 	&	-14.462 	&&	-5.781 	&	-11.380 	&	-15.112 	&&	-6.206 	&	-11.765 	&	-15.553 	\\
{\tt RegAC}	&	-5.281 	&	-10.814 	&	-14.462 	&&	-5.744 	&	-11.380 	&	-15.112 	&&	-6.202 	&	-11.763 	&	-15.553 	\\
{\tt RelA}	&	-5.281 	&	-10.769 	&	-14.421 	&&	-5.606 	&	-11.217 	&	-14.907 	&&	-5.877 	&	-11.519 	&	-15.422 	\\
{\tt NLPC}	&	-5.145 	&	-10.814 	&	-14.427 	&&	-5.744 	&	-11.326 	&	-14.972 	&&	-6.182 	&	-11.693 	&	-15.475 	\\
{\tt NLP}	&	-5.145 	&	-10.814 	&	-14.427 	&&	-5.744 	&	-11.326 	&	-14.972 	&&	-6.182 	&	-11.693 	&	-15.475 	\\
{\tt SNSCOC}	&	-5.281 	&	-10.821 	&	-14.462 	&&	-5.781 	&	-11.348 	&	-15.112 	&&	-6.193 	&	-11.765 	&	-15.553 	\\
{\tt SNSCO}	&	-5.281 	&	-10.821 	&	-14.462 	&&	-5.781 	&	-11.348 	&	-15.112 	&&	-6.193 	&	-11.765 	&	-15.553 	\\\hline

&       \multicolumn{11}{c}{Time (seconds)}   																			\\\hline
{\tt GUROBI}$10^2$	&	0.329 	&	0.417 	&	0.436 	&&	0.658 	&	0.626 	&	0.641 	&&	2.084 	&	5.488 	&	14.17 	\\
{\tt GUROBI}$10^3$	&	0.288 	&	0.351 	&	0.306 	&&	0.624 	&	0.565 	&	0.600 	&&	1.429 	&	1.551 	&	9.417 	\\
{\tt GUROBI}$10^4$	&	0.197 	&	0.261 	&	0.298 	&&	0.516 	&	0.879 	&	0.692 	&&	0.975 	&	5.153 	&	8.891 	\\
{\tt RegAC}	&	0.230 	&	0.561 	&	0.825 	&&	0.381 	&	0.945 	&	1.224 	&&	0.651 	&	1.528 	&	1.671 	\\
{\tt RelA}	&	0.086 	&	0.301 	&	0.466 	&&	0.172 	&	0.558 	&	0.847 	&&	0.365 	&	0.792 	&	1.162 	\\
{\tt NLPC}	&	0.750 	&	1.026 	&	1.038 	&&	0.837 	&	1.128 	&	0.525 	&&	0.801 	&	0.999 	&	0.429 	\\
{\tt NLP}	&	0.028 	&	0.060 	&	0.079 	&&	0.035 	&	0.040 	&	0.086 	&&	0.032 	&	0.071 	&	0.300 	\\
{\tt SNSCOC}	&	0.210 	&	0.197 	&	0.248 	&&	0.086 	&	0.222 	&	0.423 	&&	0.123 	&	0.312 	&	0.204 	\\
{\tt SNSCO}	&	0.005 	&	0.012 	&	0.004 	&&	0.002 	&	0.003 	&	0.004 	&&	0.002 	&	0.005 	&	0.004 	\\\hline

		\end{tabular}
\end{table}

\begin{table}[b]
	\renewcommand{\arraystretch}{1.0}\addtolength{\tabcolsep}{2.0pt}
	\caption{Effect of $N$ for the single CCP with $K=10$ and $N\in\{200,600,1000\}$.\label{table:n-single}} 
	\centering
		\begin{tabular}{lrrrrrrrrrrrr}
			\hline 

 	&	\multicolumn{3}{c}{$\alpha=0.01$} &&	\multicolumn{3}{c}{$\alpha=0.05$}	 	 	&&	\multicolumn{3}{c}{$\alpha=0.10$}	 		&		\\\cline{2-4}\cline{6-8}\cline{10-12}
 	$N$ &	$200$	&	$600$	&	$1000$	&&	$200$	&	$600$	&	$1000$	&&	$200$	&	$600$	&	$1000$	\\\hline
 	 &	\multicolumn{11}{c}{$f(\bfx)$}\\\hline

{\tt GUROBI}$10^2$	&	-5.178 	&	-5.084 	&	-5.064 	&&	-5.725 	&	-5.668 	&	-5.609 	&&	-6.140 	&	-6.010 	&	-6.021 	\\
{\tt GUROBI}$10^3$	&	-5.178 	&	-5.084 	&	-5.064 	&&	-5.725 	&	-5.668 	&	-5.610 	&&	-6.140 	&	-6.010 	&	-6.029 	\\
{\tt GUROBI}$10^4$	&	-5.178 	&	-5.084 	&	-5.064 	&&	-5.725 	&	-5.668 	&	-5.609 	&&	-6.140 	&	-6.010 	&	-6.028 	\\
{\tt RegAC}	&	-5.178 	&	-5.083 	&	$--$	&&	-5.713 	&	-5.668 	&	$--$	&&	-6.140 	&	-5.998 	&	$--$	\\
{\tt RelA}	&	-4.995 	&	-4.930 	&	-4.876 	&&	-5.542 	&	-5.316 	&	-5.122 	&&	-5.761 	&	-5.428 	&	-5.216 	\\
{\tt NLPC}	&	-5.178 	&	-5.072 	&	-5.044 	&&	-5.645 	&	-5.623 	&	-5.589 	&&	-6.122 	&	-5.984 	&	-5.991 	\\
{\tt NLP}	&	-5.178 	&	-5.072 	&	-5.044 	&&	-5.645 	&	-5.623 	&	-5.589 	&&	-6.122 	&	-5.984 	&	-5.991 	\\
{\tt SNSCOC}	&	-5.178 	&	-5.084 	&	-5.064 	&&	-5.713 	&	-5.660 	&	-5.606 	&&	-6.140 	&	-6.002 	&	-6.031 	\\
{\tt SNSCO}	&	-5.178 	&	-5.084 	&	-5.064 	&&	-5.713 	&	-5.660 	&	-5.606 	&&	-6.140 	&	-6.002 	&	-6.031 	\\\hline

&       \multicolumn{11}{c}{Time (seconds)}   																			\\\hline
{\tt GUROBI}$10^2$	&	0.425 	&	2.195 	&	6.363 	&&	2.034 	&	57.67 	&	600.5 	&&	3.384 	&	600.6 	&	600.5 	\\
{\tt GUROBI}$10^3$	&	0.615 	&	1.212 	&	6.180 	&&	2.111 	&	59.11 	&	600.6 	&&	6.007 	&	600.8 	&	600.7 	\\
{\tt GUROBI}$10^4$	&	0.516 	&	2.280 	&	9.203 	&&	2.112 	&	36.35 	&	600.7 	&&	3.239 	&	601.1 	&	600.9 	\\
{\tt RegAC}	&	1.806 	&	60.51 	&	$--$	&&	5.496 	&	177.3 	&	$--$	&&	7.015 	&	248.1 	&	$--$	\\
{\tt RelA}	&	0.971 	&	12.07 	&	53.83 	&&	2.493 	&	16.73  	&	59.32  	&&	3.053 	&	17.06 	&	61.06 	\\
{\tt NLPC}	&	1.219 	&	3.201 	&	7.294 	&&	1.352 	&	3.984 	&	8.927 	&&	1.459 	&	3.816 	&	8.518 	\\
{\tt NLP}	&	0.050 	&	0.133 	&	0.238 	&&	0.058 	&	0.141 	&	0.277 	&&	0.057 	&	0.140 	&	0.280 	\\
{\tt SNSCOC}	&	0.253 	&	0.302 	&	0.421 	&&	0.154 	&	0.443 	&	0.746 	&&	0.218 	&	0.703 	&	0.914 	\\
{\tt SNSCO}	&	0.008 	&	0.013 	&	0.015 	&&	0.002 	&	0.005 	&	0.008 	&&	0.003 	&	0.014 	&	0.013 	\\\hline

		\end{tabular}
\end{table}

\begin{table}[b]
	\renewcommand{\arraystretch}{1.0}\addtolength{\tabcolsep}{4.0pt}
	\caption{Effect of $N$ for the single CCP with $K=10$ and $N\in\{2000,6000\}$.\label{table:n-single-large}} 
	\centering
		\begin{tabular}{lrrrrrrrrrrrrr}
			\hline 
  	 &	\multicolumn{5}{c}{$f(\bfx)$}&&	\multicolumn{5}{c}{Time (seconds)}\\\cline{2-6}\cline{8-12} 	 
  	 	&\multicolumn{2}{c}{$\alpha=0.05$}	 	 	&&	\multicolumn{2}{c}{$\alpha=0.10$} &&\multicolumn{2}{c}{$\alpha=0.05$}	 	 	&&	\multicolumn{2}{c}{$\alpha=0.10$}	 \\\cline{2-3}\cline{5-6}\cline{8-9}\cline{11-12}
 	$N$ &	$2000$	&	$6000$	&&	$2000$	&	$6000$ 
 	&&	$2000$	&	$6000$	&&	$2000$	&	$6000$ \\\hline	 
{\tt GUROBI}$10^2$	&	-5.601 	&	-5.552 	&&	-5.983 	&	-5.942 	&&	600.4 	&	600.5 	&&	600.4 	&	600.5 	\\
{\tt GUROBI}$10^3$	&	-5.595 	&	-5.551 	&&	-5.983 	&	-5.940 	&&	600.5 	&	600.5 	&&	600.4 	&	600.5 	\\
{\tt GUROBI}$10^4$	&	-5.593 	&	-5.547 	&&	-5.985 	&	-5.943 	&&	600.5 	&	600.6 	&&	600.5 	&	600.5 	\\
{\tt NLPC}	&	-5.581 	&	-5.548 	&&	-5.949 	&	-5.943 	&&	44.20 	&	212.3 	&&	28.94 	&	291.6 	\\
{\tt NLP}	&	-5.581 	&	-5.548 	&&	-5.949 	&	-5.943 	&&	1.628 	&	5.963 	&&	0.883 	&	12.14 	\\
{\tt SNSCOC}	&	-5.600 	&	-5.559 	&&	-5.987 	&	-5.945 	&&	1.082 	&	6.308 	&&	1.380 	&	6.831 	\\
{\tt SNSCO}	&	-5.600 	&	-5.559 	&&	-5.987 	&	-5.945 	&&	0.011 	&	0.123 	&&	0.024 	&	0.136 	\\\hline

		\end{tabular}
\end{table}

   {\it a) Effect of $K$}. To see this, we fix $N=100$ and choose $K\in\{10,30,50\}$. Results are reported in Table \ref{table:k-single}. Concerning objective function values,  {\tt GUROBI} always obtains the lowest ones, followed by  {\tt SNSCO} and \reg. In contrast, \rel\  delivers the worst results. As for the computational speed, under the best choice of parameters $\epsilon$ and $\tau$, both {\tt SNSCO} and {\tt NLP} run much faster than the other algorithms. Without these best choices, {\tt SNSCO} is the winner. 
   
{\it b) Effect of $N$}. To observe this effect, by fixing $K=10$, we select $N\in\{200,600,1000\}$ and present the relevant results in Table \ref{table:n-single}, where symbol `$--$' means that we do not report the results for \reg\ as it consumes a long time to solve the problem when $N=1000$.  In terms of solution quality,  {\tt GUROBI}  and  {\tt SNSCO} achieve the lowest objective function values, followed by \reg\ and {\tt NLP}. Regarding computational speed, {\tt SNSCO} and {\tt NLP} always run the fastest.  Three {\tt GUROBI} algorithms cost about 600 seconds when $N\geq 600$ and $\alpha=0.1$ since they reach the maximal time limit (i.e., 10 minutes). We note that when $N=1000$ and $\alpha=0.1$, {\tt SNSCO} generates slightly better solutions than {\tt GUROBI}  as the latter reaches the maximal time limit without obtaining optimal solutions.  This indicates that {\tt SNSCO} is more advantageous for large scale computation compared to {\tt GUROBI}. Such as assertion is supported by the results in Table \ref{table:n-single-large}, where {\tt SNSCO} produces the lowest objective values and runs the fastest among all algorithms.

\subsection{Joint CCP with non i.i.d. data}
In the subsequent numerical experiments, we compare all algorithms for solving the joint CCP (namely, $M>1$),  but exclude \reg\ and \rel\  as they were designed to address the single CCP. Moreover, to test the algorithms solving problems with more complicated instances,  we now employ them to solve \eqref{NOP} with non i.i.d. standard normal random variables $\xi_{k}^m,m\in\M,k\in\K$. In this scenario, the optimal solution to \eqref{NOP} remains elusive, and $\bfx^{\rm opt}$ as defined in \eqref{opt-solution} is no longer valid. We adopt the data generation from \cite{hong2011sequential}, that is,  $\xi_{k}^m,m\in\M,k\in\K$ are normal random variables with mean $k/K$ and variance $1$, and $ {\rm Cov}(\xi_{k}^m,\xi_{k}^{m'}) = 0.5$ for if $m\neq m'\in\M$ and  $ {\rm Cov}(\xi_{k}^m,\xi_{k'}^{m'}) = 0$ if $k\neq k'$ for any $m,m'\in\M$ and $k,k'\in\K$.

 \begin{table}[t]
	\renewcommand{\arraystretch}{1.0}\addtolength{\tabcolsep}{2.0pt}
	\caption{Effect of $K$ for the joint CCP with $M=10, N=100$ and $K\in\{10,20,30\}$.\label{table:k-joint}} 
	\centering
		\begin{tabular}{lrrrrrrrrrrrr}
			\hline 

 	&	\multicolumn{3}{c}{$\alpha=0.01$} &&	\multicolumn{3}{c}{$\alpha=0.05$}	 	 	&&	\multicolumn{3}{c}{$\alpha=0.10$}	 		&		\\\cline{2-4}\cline{6-8}\cline{10-12}
 	$K$ &	$10$	&	$20$	&	$30$	&&	$10$	&	$20$	&	$30$	&&	$10$	&	$20$	&	$30$	\\\hline
 	 &	\multicolumn{11}{c}{$f(\bfx)$}\\\hline
{\tt GUROBI}$10^2$	&	-4.199 	&	-6.778 	&	-8.947 	&&	-4.539 	&	-7.122 	&	-9.264 	&&	-4.699 	&	-7.356 	&	-9.454 	\\
{\tt GUROBI}$10^3$	&	-4.199 	&	-6.778 	&	-8.947 	&&	-4.539 	&	-7.122 	&	-9.264 	&&	-4.699 	&	-7.356 	&	-9.454 	\\
{\tt GUROBI}$10^4$	&	-4.199 	&	-6.778 	&	-8.947 	&&	-4.539 	&	-7.122 	&	-9.264 	&&	-4.699 	&	-7.356 	&	-9.454 	\\
{\tt NLPC}	&	-4.141 	&	-6.775 	&	-8.898 	&&	-4.527 	&	-7.000 	&	-9.211 	&&	-4.631 	&	-7.312 	&	-9.380 	\\
{\tt NLP}	&	-4.141 	&	-6.775 	&	-8.898 	&&	-4.527 	&	-7.000 	&	-9.211 	&&	-4.631 	&	-7.312 	&	-9.380 	\\
{\tt SNSCOC}	&	-4.199 	&	-6.778 	&	-8.947 	&&	-4.539 	&	-7.122 	&	-9.264 	&&	-4.693 	&	-7.356 	&	-9.454 	\\
{\tt SNSCO}	&	-4.199 	&	-6.778 	&	-8.947 	&&	-4.539 	&	-7.122 	&	-9.264 	&&	-4.693 	&	-7.356 	&	-9.454 	\\\hline

&       \multicolumn{11}{c}{Time (seconds)}   																			\\\hline
{\tt GUROBI}$10^2$	&	0.575 	&	1.039 	&	1.372 	&&	1.690 	&	1.891 	&	3.154 	&&	5.798 	&	6.760 	&	14.42 	\\
{\tt GUROBI}$10^3$	&	0.599 	&	1.056 	&	1.541 	&&	1.621 	&	2.517 	&	4.623 	&&	6.430 	&	9.322 	&	45.67 	\\
{\tt GUROBI}$10^4$	&	0.618 	&	0.895 	&	1.080 	&&	2.070 	&	2.727 	&	3.778 	&&	6.854 	&	18.31 	&	10.89 	\\
{\tt NLPC}	&	19.16 	&	27.80 	&	148.0 	&&	13.86 	&	30.98 	&	172.1 	&&	14.57 	&	38.14 	&	158.4 	\\
{\tt NLP}	&	0.420 	&	0.972 	&	1.822 	&&	0.280 	&	1.340 	&	4.264 	&&	0.711 	&	0.609 	&	5.561 	\\
{\tt SNSCOC}	&	0.358 	&	0.411 	&	0.628 	&&	0.577 	&	0.545 	&	0.946 	&&	0.473 	&	0.541 	&	1.058 	\\
{\tt SNSCO}	&	0.013 	&	0.015 	&	0.016 	&&	0.015 	&	0.010 	&	0.016 	&&	0.013 	&	0.006 	&	0.017 	\\\hline

		\end{tabular}
\end{table}

{\it c) Effect of $K$.}  We set $(M,N)=(10,100)$ and vary $K$ among $\{10,20,30\}$. In these experiments, we opt for smaller values of $K$ due to the considerable time required by {\tt NLPC} to solve the problem when $K>30$ and $\alpha\geq0.05$. Table \ref{table:k-joint} illustrates that the three {\tt GUROBI} solvers and two \nscp\ algorithms yield the lowest objective function values. Regarding computational speed,   {\tt NLPC} consumes quite a long time when $K$ gets increased. Once again, \nscp\  runs much faster than the others.

\begin{table}[t]
	\renewcommand{\arraystretch}{1.0}\addtolength{\tabcolsep}{2.0pt}
	\caption{Effect of $N$ for the joint CCP with $K=10, M=10$ and $N\in\{200,600,1000\}$.\label{table:n-joint}} 
	\centering
		\begin{tabular}{lrrrrrrrrrrrr}
			\hline 

 	&	\multicolumn{3}{c}{$\alpha=0.01$} &&	\multicolumn{3}{c}{$\alpha=0.05$}	 	 	&&	\multicolumn{3}{c}{$\alpha=0.10$}	 		&		\\\cline{2-4}\cline{6-8}\cline{10-12}
 	$N$ &	$200$	&	$600$	&	$1000$	&&	$200$	&	$600$	&	$1000$	&&	$200$	&	$600$	&	$1000$	\\\hline
 	 &	\multicolumn{11}{c}{$f(\bfx)$}\\\hline

{\tt GUROBI}$10^2$	&	-4.198 	&	-4.075 	&	-4.053 	&&	-4.482 	&	-4.408 	&	-4.405 	&&	-4.654 	&	-4.606 	&	-4.591 	\\
{\tt GUROBI}$10^3$	&	-4.198 	&	-4.075 	&	-4.053 	&&	-4.482 	&	-4.407 	&	-4.406 	&&	-4.654 	&	-4.594 	&	-4.593 	\\
{\tt GUROBI}$10^4$	&	-4.198 	&	-4.075 	&	-4.053 	&&	-4.482 	&	-4.408 	&	-4.406 	&&	-4.654 	&	-4.609 	&	-4.587 	\\
{\tt NLPC}	&	-4.181 	&	-4.052 	&	-4.026 	&&	-4.457 	&	-4.383 	&	-4.391 	&&	-4.619 	&	-4.569 	&	-4.579 	\\
{\tt NLP}	&	-4.181 	&	-4.052 	&	-4.026 	&&	-4.457 	&	-4.383 	&	-4.391 	&&	-4.619 	&	-4.569 	&	-4.579 	\\
{\tt SNSCOC}	&	-4.198 	&	-4.075 	&	-4.053 	&&	-4.475 	&	-4.404 	&	-4.406 	&&	-4.652 	&	-4.606 	&	-4.593 	\\
{\tt SNSCO}	&	-4.198 	&	-4.075 	&	-4.053 	&&	-4.475 	&	-4.404 	&	-4.406 	&&	-4.652 	&	-4.606 	&	-4.593 	\\\hline

&       \multicolumn{11}{c}{Time (seconds)}   																			\\\hline
{\tt GUROBI}$10^2$	&	1.552 	&	13.97 	&	73.27 	&&	11.18 	&	526.2 	&	600.3 	&&	37.14 	&	600.3 	&	600.3 	\\
{\tt GUROBI}$10^3$	&	1.489 	&	10.82 	&	53.48 	&&	11.89 	&	600.2 	&	600.3 	&&	36.22 	&	600.3 	&	600.4 	\\
{\tt GUROBI}$10^4$	&	1.471 	&	10.46 	&	112.8 	&&	10.96 	&	600.3 	&	600.4 	&&	61.10 	&	600.3 	&	600.4 	\\
{\tt NLPC}	&	14.51 	&	17.08 	&	51.00 	&&	13.82 	&	23.68 	&	41.92 	&&	10.52 	&	28.65 	&	30.67 	\\
{\tt NLP}	&	0.358 	&	1.060 	&	2.403 	&&	0.304 	&	0.833 	&	2.000 	&&	0.458 	&	1.061 	&	1.797 	\\
{\tt SNSCOC}	&	0.358 	&	0.812 	&	1.235 	&&	0.673 	&	3.186 	&	5.119 	&&	0.750 	&	4.324 	&	6.177 	\\
{\tt SNSCO}	&	0.016 	&	0.022 	&	0.042 	&&	0.010 	&	0.051 	&	0.042 	&&	0.014 	&	0.065 	&	0.069 	\\\hline

		\end{tabular}
\end{table}

\begin{table}[t]
	\renewcommand{\arraystretch}{1.0}\addtolength{\tabcolsep}{2.0pt}
	\caption{Effect of $M$ for the joint CCP with $K=10, N=100$ and $M\in\{20,60,100\}$.\label{table:m-joint}} 
	\centering
		\begin{tabular}{lrrrrrrrrrrrr}
			\hline 

 	&	\multicolumn{3}{c}{$\alpha=0.01$} &&	\multicolumn{3}{c}{$\alpha=0.05$}	 	 	&&	\multicolumn{3}{c}{$\alpha=0.10$}	 		&		\\\cline{2-4}\cline{6-8}\cline{10-12}
 	$M$ &	$20$	&	$60$	&	$100$	&&	$20$	&	$60$	&	$100$	&&	$20$	&	$60$	&	$100$\\\hline
 	 &	\multicolumn{11}{c}{$f(\bfx)$}\\\hline

{\tt GUROBI}$10^2$	&	-4.139 	&	-3.888 	&	-3.868 	&&	-4.355 	&	-4.107 	&	-4.057 	&&	-4.503 	&	-4.284 	&	-4.199 	\\
{\tt GUROBI}$10^3$	&	-4.139 	&	-3.888 	&	-3.868 	&&	-4.355 	&	-4.107 	&	-4.057 	&&	-4.503 	&	-4.284 	&	-4.199 	\\
{\tt GUROBI}$10^4$	&	-4.139 	&	-3.888 	&	-3.868 	&&	-4.355 	&	-4.107 	&	-4.057 	&&	-4.503 	&	-4.284 	&	-4.199 	\\
{\tt NLPC}	&	-4.137 	&	-3.835 	&	-3.856 	&&	-4.286 	&	-4.077 	&	-4.033 	&&	-4.447 	&	-4.226 	&	-4.156 	\\
{\tt NLP}	&	-4.137 	&	-3.835 	&	-3.856 	&&	-4.286 	&	-4.077 	&	-4.033 	&&	-4.447 	&	-4.226 	&	-4.156 	\\
{\tt SNSCOC}	&	-4.139 	&	-3.888 	&	-3.868 	&&	-4.355 	&	-4.107 	&	-4.057 	&&	-4.503 	&	-4.284 	&	-4.199 	\\
{\tt SNSCO}	&	-4.139 	&	-3.888 	&	-3.868 	&&	-4.355 	&	-4.107 	&	-4.057 	&&	-4.503 	&	-4.284 	&	-4.199 	\\\hline

&       \multicolumn{11}{c}{Time (seconds)}   																			\\\hline
{\tt GUROBI}$10^2$	&	1.027 	&	2.557 	&	4.716 	&&	2.775 	&	5.441 	&	9.024 	&&	9.009 	&	22.66 	&	17.73 	\\
{\tt GUROBI}$10^3$	&	0.995 	&	2.792 	&	4.993 	&&	2.762 	&	7.318 	&	9.649 	&&	9.363 	&	30.22 	&	18.23 	\\
{\tt GUROBI}$10^4$	&	1.038 	&	2.579 	&	4.541 	&&	3.999 	&	8.148 	&	14.81 	&&	9.440 	&	27.06 	&	115.4 	\\
{\tt NLPC}	&	16.54 	&	18.56 	&	19.44 	&&	15.65 	&	15.83 	&	25.10 	&&	13.80 	&	13.86 	&	24.24 	\\
{\tt NLP}	&	0.353 	&	0.479 	&	0.708 	&&	0.530 	&	0.474 	&	1.109 	&&	0.402 	&	0.633 	&	0.789 	\\
{\tt SNSCOC}	&	0.354 	&	1.050 	&	1.351 	&&	0.725 	&	1.606 	&	2.520 	&&	0.727 	&	2.187 	&	2.568 	\\
{\tt SNSCO}	&	0.017 	&	0.027 	&	0.037 	&&	0.008 	&	0.023 	&	0.023 	&&	0.014 	&	0.082 	&	0.042 	\\\hline

		\end{tabular}
\end{table}

{\it d) Effect of $N$.} We select $N\in\{200,600,1000\}$ while fixing $K=10$ and $M=10$. The data are summarized in Table \ref{table:n-joint}. When $\alpha\geq0.05$ and $N\geq600$, three {\tt GUROBI} solvers always reach the maximal time limit and thus produce slightly worse solutions than those of  {\tt SNSCO}. Taking the case of $(\alpha, N)=(0.1,1000)$ as an instance,  {\tt SNSCO} spends 0.069 seconds solving the problem, while {\tt NLP} and {\tt GUROBI} need 1.797 and 600 seconds, respectively.   These results highlight the efficiency of {\tt SNSCO} in handling problems with large-scale samples $N$.

{\it e) Effect of $M$.} By fixing $(K,N)=(10,100)$ and choosing $M\in\{20,40,60\}$, we obtain the results in Table \ref{table:m-joint}. Three {\tt GUROBI} solvers and two \nscp\ algorithms obtain identical objective function values. Meanwhile,  two {\tt NLP} algorithms achieve values close to the optimal ones. In terms of computational efficiency, \nscp\ runs the fastest, and {\tt NLP} is the runner-up.

\section{Conclusion}
The $0/1$ loss function ideally characterizes the constraints of SAA. However, due to its discontinuous nature, it has impeded the development of numerical algorithms for solving SAA for a long time. In this paper, we studied a general $0/1$ constrained optimization problem, directly applicable to addressing SAA. One key factor contributing to this success was the derivation of the normal cone to the feasible set. Another crucial factor was the establishment of the $\tau$-stationary equations, a type of optimality condition that enables us to leverage the semismooth Newton-type method. {
The theorems regarding the optimality and the algorithm can be extended to encompass a broader case. For instance, we could incorporate additional equality or inequality constraints \cite{hong2011sequential,curtis2018sequential} into \eqref{SCP}. However, we may need to reformulate the complementarity conditions as equations using nonlinear complementarity problem functions.  
This adjustment would allow us to continue to use semismooth Newton-type methods, worthy of further investigation.}


 \begin{APPENDIX}{Proof of Lemma \ref{lemma-neighbour}}
\proof{Proof} a) Let $\Gamma_+^*, \Gamma_-^*$, and $\Gamma_0^*$ be defined for $\Z^* $, while let   $\Gamma_+, \Gamma_-$, and $\Gamma_0$ be defined for $\L $ as
\begin{eqnarray}\label{Lambda-Gamma0}
 \eqspace{1.35}
\begin{array}{ccl}
\Gamma_+&=&\left\{n\in{\N} :~\L _{:n}^{\max}>0\right\},\\
\Gamma_-&=&\left\{n\in{\N} :~\L _{:n}^{\max}<0\right\},\\
\Gamma_0&=&\left\{n\in{\N} :~\L _{:n}^{\max}=0\right\}.
\end{array}
\end{eqnarray}
Similar to \eqref{T-z}, let $\Gamma_{s}\subseteq  \Gamma_+$ extract  $s$ indices in $\Gamma_+$ that correspond to the first $s$ largest elements in $\{\|{\L}_{:n}^+\|:n\in\Gamma_+\}$.  Moreover, we define $\J_*$  as \eqref{gamma*} and $\J$   as
   \begin{eqnarray} \label{def-J-Lambda}
  \eqspace{1.35}
     \begin{array}{rll}
     \J:=U_T=\{(m,n): \Lambda _{mn}\geq 0, n\in T\},~~T\in\T(\L ,s).
     \end{array} 
  \end{eqnarray}
It follows from \eqref{sta-eq-1} that a \ts\ $\bfw^*$ satisfies $\Gamma_0^*\in \T(\L^*;s)$,  $  U_{\Gamma_0^*}=\J_*$, and 
\begin{eqnarray}\label{sta-eq-1-11}
  \eqspace{1.35}
     \begin{array}{rll}
   \varphi(\bfw^*,\J_*)={\bf 0},~~\Z^*_{\J_*}={\bf 0},~~ \W ^*_{\overline \J_*}={\bf 0}.
     \end{array}
  \end{eqnarray}
Consider any $\bfw\in {\mathbb N} (\bfw^*,\eta_*)=: {\mathbb N}^*$  with a sufficiently small radius $\eta_*>0$. For such $\bfw$, we define $\Z$ and $\L $ by \eqref{U-Gamma} and $ \J=U_T$ by \eqref{def-J-Lambda}.  To show \eqref{gw*-0}, we need to prove
\begin{eqnarray}\label{gw*-0-1}
     \begin{array}{rll}
      \varphi(\bfw^*,\J)={\bf 0},~~
   \Z^*_{\J}={\bf 0},~~
     \W ^*_{\overline \J}={\bf 0}.
     \end{array}
  \end{eqnarray}
In the sequel, we aim to prove  
\begin{eqnarray}\label{gw*-0-2}
\J\subseteq \J_*, ~~\W ^*_{\overline \J}={\bf 0}.
  \end{eqnarray}
This is because if \eqref{gw*-0-2} holds then   \eqref{gw*-0-1} can be ensured due to $\varphi(\bfw^*,\J)=  \varphi(\bfw^*,\J_*)={\bf 0}. $ 
\begin{table}[H]
	\renewcommand{\arraystretch}{1.25}\addtolength{\tabcolsep}{18pt}
	\caption{Index set decomposition.\label{ind-decom}}
		\begin{tabular}{l|ccc}
			\hline
 	   	 &\multicolumn{3}{|c}{${n\in\N}$}	\\\cline{2-4}   
 	   & $n\in\Gamma_0^*$ &$n\in\Gamma_+^*$&$n\in\Gamma_-^*$\\	\hline 
 \multirow{4}{*}{$m\in\M$}	   & $ (m,n)\in \J_*:\left\{\begin{array}{r}
 Z_{mn}^*=0\\
  W_{mn}^*\geq 0\\
  \Lambda _{mn}^*\geq 0	   
\end{array} 	  \right.  $
& \multirow{4}{*}{$\left\{\begin{array}{r}
(\Z^*)_{:n}^{\max}>0\\
 W_{mn}^* = 0\\
  (\L^*)_{:n}^{\max}>0
\end{array}\right.$}
& \multirow{4}{*}{$\left\{\begin{array}{r}
(\Z^*)_{:n}^{\max}<0\\
 W_{mn}^* = 0\\
  (\L^*)_{:n}^{\max}<0
\end{array}\right.$}
\\
&   $ (m,n)\in \J_-:\left\{\begin{array}{r}
 Z_{mn}^*<0\\
  W_{mn}^*= 0\\
  \Lambda _{mn}^*< 0	   
\end{array}\right. 	    $
&&\\
\hline
		\end{tabular}
\end{table}

 Suppose there is an index $(m_0,n_0)\in \J$ but $(m_0,n_0)\notin \J_*$. We decompose the entire index set ${\M}\times {\N}$ as Table \ref{ind-decom}, where we used three facts: $\L^*=\Z^*+\tau W^*$,  \eqref{gamma*}, and Proposition \ref{pro-eta} that \begin{eqnarray*} 
  \eqspace{1.35}
     \begin{array}{rll}\W^*_{:\overline \Gamma_0^*}={\bf 0},\qquad {\bf 0}\geq \Z^* _{:\Gamma_0^*} \perp \W^*_{:\Gamma_0^*} \geq {\bf 0}.   \end{array} 
  \end{eqnarray*} 
  Since $\eta_*>0$ can be set sufficiently small, $\bfw$ can be close to $\bfw^*$, and so is $\L $ to $\L^*$, which shows
 \begin{eqnarray}\label{gamma-gamma-*} 
  \eqspace{1.35}
     \begin{array}{rll}
\forall~n\in\Gamma_+^*:~~(\Z^*)_{:n}^{\max}=(\L^*)_{:n}^{\max}>0~~\Longrightarrow ~~(\L )_{:n}^{\max}>0,\\
\forall~n\in\Gamma_-^*:~~(\Z^*)_{:n}^{\max}=(\L^*)_{:n}^{\max}<0~~\Longrightarrow ~~(\L )_{:n}^{\max}<0.\\
     \end{array} 
  \end{eqnarray}
The definition of $\T(\L ,s)$ in \eqref{T-z} implies $T=\Gamma_0 \cup (\Gamma_+\setminus \Gamma_s)$ for any given $T\in\T(\L ,s)$. Condition \eqref{gamma-gamma-*} suffices to $\Gamma_+^* \subseteq \Gamma_s\subseteq \Gamma_+$ and $\Gamma_-^* \subseteq \Gamma_-$. Therefore, we must have $T\subseteq \Gamma_0^*$. Now combining this condition, Table \ref{ind-decom}, \eqref{def-J-Lambda},  $(m_0,n_0)\in \J=U_T$, and $(m_0,n_0)\notin \J_*$,  we can claim that $(m_0,n_0)\in \J_-$, thereby resulting in $\Lambda_{m_0n_0}^*{<}0$. However, since  $\L $ is relatively close to $\L^*$, we have $\Lambda_{m_0n_0}{<}0$, which contradicts to $(m_0,n_0)\in \J$ in \eqref{def-J-Lambda}. So we prove $\J\subseteq \J_*$, the first condition in \eqref{gw*-0-2} 

Finally, we prove $\W ^*_{\overline \J}={\bf 0}$. If $\|\Z^*\|_0^+<s$, then $\W ^*={\bf 0}$ by Proposition \ref{pro-eta}. The conclusion is clearly true. We focus on  $\|\Z^*\|_0^+=s$.  This indicates $|\Gamma_+^*|=\|\Z^*\|_0^+=s$. Again, by \eqref{gamma-gamma-*}, we can derive
\begin{eqnarray}\label{relation-diff-sets}
 \eqspace{1.35}
\begin{array}{ccl}
\Gamma_+^*=\Gamma_s,~~\Gamma_-^* \subseteq \Gamma_-, ~~\Gamma_0^*\supseteq T=\Gamma_0 \cup (\Gamma_+\setminus \Gamma_s).
\end{array}
\end{eqnarray}
To show $\W ^*_{\overline \J}={\bf 0}$,  we only need to check $\W ^*_{\J_* \setminus \J}={\bf 0}$ owing to Table  \ref{ind-decom} and $\J\subseteq \J_*$. It follows from Table  \ref{ind-decom} that $\W ^*\geq {\bf 0}$. Suppose that there exists $(m_0,n_0)\in \J_* \setminus \J$ such that $W_{m_0n_0}^* > 0$. This implies $\Lambda_{m_0n_0}^* > 0$ and thus $\Lambda_{m_0n_0}>0$ because $\L $ is close to $\L^*$. Then we have $n_0\notin\Gamma_-$. In addition, suppose $n_0\in T$, then \eqref{def-J-Lambda} means $(m_0,n_0)\in \J$, a contradiction. So $n_0\notin T$. Overall, $n_0\in \Gamma_s=\Gamma_+^*$   by \eqref{relation-diff-sets}. However, $ \J_* \setminus \J\subseteq\J_*$, we must have  $n_0\in\Gamma_0^*$ due to $(m_0,n_0)\in \J_*$, which contradicts with $n_0\in  \Gamma_+^*$.  Hence, $\J_* \setminus \J=\emptyset$ and  $\J = \J_*$, there by $W ^*_{\overline \J}={\bf 0}$ from \eqref{sta-eq-1-11}.  
\Halmos\endproof

 \end{APPENDIX}
%
%

\section*{Acknowledgments.}
This work was supported by the National Key R\&D Program of China (2023YFA1011100), the Fundamental Research Funds for the Central Universities, and the National Natural Science Foundation of China (12271309). The authors express their sincere gratitude to the area editor, associate editor, and two referees for their constructive comments, which have significantly improved the quality of the paper. The authors also extend their appreciation to Mr. Shuai Li for his assistance in implementing one of the algorithms used in the numerical experiments.


\bibliographystyle{informs2014} 
\bibliography{references} 

@article{ma2011fixed,
  title={Fixed point and Bregman iterative methods for matrix rank minimization},
  author={Ma, Shiqian and Goldfarb, Donald and Chen, Lifeng},
  journal={Math. Program.},
  volume={128},
  number={1},
  pages={321--353},
  year={2011},
  publisher={Springer}
}

@article{hale2008fixed,
  title={Fixed-point continuation for $\ell_1$-minimization: Methodology and convergence},
  author={Hale, Elaine T and Yin, Wotao and Zhang, Yin},
  journal={SIAM J. Optim.},
  volume={19},
  number={3},
  pages={1107--1130},
  year={2008},
  publisher={SIAM}
}

@article{huang2018robust,
  title={Robust decoding from 1-bit compressive sampling with ordinary and regularized least squares},
  author={Huang, Jian and Jiao, Yuling and Lu, Xiliang and Zhu, Liping},
  journal={SIAM J.  Sci. Comput.},
  volume={40},
  number={4},
  pages={A2062--A2086},
  year={2018},
  publisher={SIAM}
}

@article{song2014chance,
  title={Chance-constrained binary packing problems},
  author={Song, Yongjia and Luedtke, James R and K{\"u}{\c{c}}{\"u}kyavuz, Simge},
  journal={INFORMS J. Comput.},
  volume={26},
  number={4},
  pages={735--747},
  year={2014},
  publisher={INFORMS}
}

@book{izmailov2014newton,
  title={Newton-type methods for optimization and variational problems},
  author={Izmailov, Alexey F and Solodov, Mikhail V},
  volume={3},
  year={2014},
  publisher={Springer}
}

@book{clarke1990optimization,
  title={Optimization and nonsmooth analysis},
  author={Clarke, Frank H},
  year={1990},
  publisher={SIAM}
}

@article{kojima1986extension,
  title={Extension of Newton and quasi-Newton methods to systems of PC\^{} 1 equations},
  author={Kojima, Masakazu and Shindo, Susumu},
  journal={J. Oper. Res. Soc. Jpn.},
  volume={29},
  number={4},
  pages={352--375},
  year={1986},
  publisher={The Operations Research Society of Japan}
}

@article{qi1993nonsmooth,
  title={A nonsmooth version of Newton's method},
  author={Qi, Liqun and Sun, Jie},
  journal={Math. Program.},
  volume={58},
  number={1-3},
  pages={353--367},
  year={1993},
  publisher={Springer}
}

@article{heitsch2019probabilistic,
  title={On probabilistic capacity maximization in a stationary gas network},
  author={Heitsch, Holger},
  journal={Optimization},
  year={2019},
  publisher={Taylor \& Francis}
}

@article{van2022generalized,
  title={Generalized differentiation of probability functions: parameter dependent sets given by intersections of convex sets and complements of convex sets},
  author={Van Ackooij, Wim and P{\'e}rez-Aros, Pedro},
  journal={Appl. Math. Optim.},
  volume={85},
  number={1},
  pages={2},
  year={2022},
  publisher={Springer}
}

@article{berthold2022algorithmic,
  title={On the algorithmic solution of optimization problems subject to probabilistic/robust (probust) constraints},
  author={Berthold, Holger and Heitsch, Holger and Henrion, Ren{\'e} and Schwientek, Jan},
  journal={Math. Methods Oper. Res.},
  volume={96},
  number={1},
  pages={1--37},
  year={2022},
  publisher={Springer}
}

@article{van2020gradient,
  title={Gradient formulae for nonlinear probabilistic constraints with non-convex quadratic forms},
  author={Van Ackooij, Wim and P{\'e}rez-Aros, Pedro},
  journal={J. Optim. Theory Appl.},
  volume={185},
  number={1},
  pages={239--269},
  year={2020},
  publisher={Springer}
}

@article{farshbaf2020optimal,
  title={Optimal Neumann boundary control of a vibrating string with uncertain initial data and probabilistic terminal constraints},
  author={Farshbaf-Shaker, M Hassan and Gugat, Martin and Heitsch, Holger and Henrion, Ren{\'e}},
  journal={SIAM J. Control Optim.},
  volume={58},
  number={4},
  pages={2288--2311},
  year={2020},
  publisher={SIAM}
}

@article{van2019generalized,
  title={Generalized differentiation of probability functions acting on an infinite system of constraints},
  author={Van Ackooij, Wim and P{\'e}rez-Aros, Pedro},
  journal={SIAM J. Optim.},
  volume={29},
  number={3},
  pages={2179--2210},
  year={2019},
  publisher={SIAM}
}

@article{hantoute2019subdifferential,
  title={Subdifferential characterization of probability functions under {G}aussian distribution},
  author={Hantoute, Abderrahim and Henrion, Ren{\'e} and P{\'e}rez-Aros, Pedro},
  journal={Math. Program.},
  volume={174},
  pages={167--194},
  year={2019},
  publisher={Springer}
}

@article{van2018sub,
  title={({S}ub-){D}ifferentiability of probability functions with elliptical distributions},
  author={Van Ackooij, Wim and Aleksovska, Ivana and Munoz-Zuniga, M},
  journal={Set-Valued Var. Anal.},
  volume={26},
  pages={887--910},
  year={2018},
  publisher={Springer}
}

@article{van2017second,
  title={Second-order differentiability of probability functions},
  author={Van Ackooij, Wim and Malick, J{\'e}r{\^o}me},
  journal={Optim. Lett.},
  volume={11},
  number={1},
  pages={179--194},
  year={2017},
  publisher={Springer}
}

@article{gonzalez2017joint,
  title={A joint model of probabilistic/robust constraints for gas transport management in stationary networks},
  author={Gonzalez Grandon, Tatiana and Heitsch, Holger and Henrion, Ren{\'e}},
  journal={Comput. Manag. Sci.},
  volume={14},
  pages={443--460},
  year={2017},
  publisher={Springer}
}

@article{van2017sub,
  title={({S}ub-){G}radient formulae for probability functions of random inequality systems under {G}aussian distribution},
  author={Van Ackooij, Wim and Henrion, Ren{\'e}},
  journal={SIAM-ASA J. Uncertain. Quantif.},
  volume={5},
  number={1},
  pages={63--87},
  year={2017},
  publisher={SIAM}
}

@article{gotzes2016quantification,
  title={On the quantification of nomination feasibility in stationary gas networks with random load},
  author={Gotzes, Claudia and Heitsch, Holger and Henrion, Ren{\'e} and Schultz, R{\"u}diger},
  journal={Math. Methods Oper. Res. },
  volume={84},
  pages={427--457},
  year={2016},
  publisher={Springer}
}

@article{van2014gradient,
  title={Gradient formulae for nonlinear probabilistic constraints with {G}aussian and {G}aussian-like distributions},
  author={Van Ackooij, Wim and Henrion, Ren{\'e}},
  journal={SIAM J. Optim.},
  volume={24},
  number={4},
  pages={1864--1889},
  year={2014},
  publisher={SIAM}
}

@article{henrion2012gradient,
  title={A gradient formula for linear chance constraints under {G}aussian distribution},
  author={Henrion, Ren{\'e} and M{\"o}ller, Andris},
  journal={Math. Oper. Res.},
  volume={37},
  number={3},
  pages={475--488},
  year={2012},
  publisher={INFORMS}
}

@article{erdougan2006ambiguous,
  title={Ambiguous chance constrained problems and robust optimization},
  author={Erdo{\u{g}}an, Emre and Iyengar, Garud},
  journal={Math. Program.},
  volume={107},
  pages={37--61},
  year={2006},
  publisher={Springer}
}

@article{de2004constraint,
  title={On constraint sampling in the linear programming approach to approximate dynamic programming},
  author={De Farias, Daniela Pucci and Van Roy, Benjamin},
  journal={Math. Oper. Res.},
  volume={29},
  number={3},
  pages={462--478},
  year={2004},
  publisher={INFORMS}
}

@article{calafiore2005uncertain,
  title={Uncertain convex programs: randomized solutions and confidence levels},
  author={Calafiore, Giuseppe and Campi, Marco C},
  journal={Math. Program.},
  volume={102},
  pages={25--46},
  year={2005},
  publisher={Springer}
}

@article{massart1990tight,
  title={The tight constant in the {D}voretzky-{K}iefer-{W}olfowitz inequality},
  author={Massart, Pascal},
  journal={Ann. Probab.},
  pages={1269--1283},
  year={1990},
  publisher={JSTOR}
}

@article{nocedal2006quadratic,
  title={Quadratic programming},
  author={Nocedal, Jorge and Wright, Stephen J},
  journal={Numerical optimization},
  pages={448--492},
  year={2006},
  publisher={Springer}
}

@article{adam2016nonlinear,
  title={Nonlinear chance constrained problems: optimality conditions, regularization and solvers},
  author={Adam, Luk{\'a}{\v{s}} and Branda, Martin},
  journal={J. Optim. Theory. Appl.},
  volume={170},
  number={2},
  pages={419--436},
  year={2016},
  publisher={Springer}
}

@incollection{ahmed2008solving,
  title={Solving chance-constrained stochastic programs via sampling and integer programming},
  author={Ahmed, Shabbir and Shapiro, Alexander},
  booktitle={State-of-the-art decision-making tools in the information-intensive age},
  pages={261--269},
  year={2008},
  publisher={Informs}
}

@article{ban2011lipschitzian,
  title={Lipschitzian stability of parametric variational inequalities over generalized polyhedra in Banach spaces},
  author={Ban, Liqun and Mordukhovich, Boris S and Song, Wen},
  journal={Nonlinear Anal. Theory Methods Appl.},
  volume={74},
  number={2},
  pages={441--461},
  year={2011},
  publisher={Elsevier}
}

@article{beraldi2010exact,
  title={An exact approach for solving integer problems under probabilistic constraints with random technology matrix},
  author={Beraldi, Patrizia and Bruni, Maria Elena},
  journal={Ann. Oper. Res.},
  volume={177},
  number={1},
  pages={127--137},
  year={2010},
  publisher={Springer}
}

@inproceedings{BB08,
  title={1-bit compressive sensing},
  author={Boufounos, Petros T and Baraniuk, Richard G},
  booktitle={2008 42nd Annual Conference on Information Sciences and Systems},
  pages={16--21},
  year={2008},
  organization={IEEE}
}

@article{charnes1958cost,
  title={Cost horizons and certainty equivalents: an approach to stochastic programming of heating oil},
  author={Charnes, Abraham and Cooper, William W and Symonds, Gifford H},
  journal={Manage. Sci.},
  volume={4},
  number={3},
  pages={235--263},
  year={1958},
  publisher={INFORMS}
}

@article{chen1998global,
  title={Global and superlinear convergence of the smoothing Newton method and its application to general box constrained variational inequalities},
  author={Chen, Xiaojun and Qi, Liqun and Sun, Defeng},
  journal={Math. Comput.},
  volume={67},
  number={222},
  pages={519--540},
  year={1998}
}

@article{CV95,
  title={Support-vector networks},
  author={Cortes, Corinna and Vapnik, Vladimir},
  journal={Mach. Learn.},
  volume={20},
  number={3},
  pages={273--297},
  year={1995},
  publisher={Springer}
}

@article{curtis2018sequential,
  title={A sequential algorithm for solving nonlinear optimization problems with chance constraints},
  author={Curtis, Frank E and W\"{a}chter, Andreas and Zavala, Victor M},
  journal={SIAM J. Optim.},
  volume={28},
  number={1},
  pages={930--958},
  year={2018},
  publisher={SIAM}
}

@article{dontchev2010newton,
  title={Newton's method for generalized equations: a sequential implicit function theorem},
  author={Dontchev, Asen L and Rockafellar, R Tyrrell},
  journal={Math. Program.},
  volume={123},
  number={1},
  pages={139--159},
  year={2010},
  publisher={Springer}
}

@article{evgeniou2000regularization,
  title={Regularization networks and support vector machines},
  author={Evgeniou, Theodoros and Pontil, Massimiliano and Poggio, Tomaso},
  journal={Adv. Comput. Math.},
  volume={13},
  number={1},
  pages={1},
  year={2000},
  publisher={Springer}
}

@article{friedman1997bias,
  title={On bias, variance, 0/1 loss, and the curse-of-dimensionality},
  author={Friedman, Jerome H},
  journal={Data Min. Knowl. Discov.},
  volume={1},
  number={1},
  pages={55--77},
  year={1997},
  publisher={Springer}
}

@article{geletu2017inner,
  title={An inner-outer approximation approach to chance constrained optimization},
  author={Geletu, Abebe and Hoffmann, Armin and Kl\"{o}ppel, Michael and Li, Pu},
  journal={SIAM J. Optim.},
  volume={27},
  number={3},
  pages={1834--1857},
  year={2017},
  publisher={SIAM}
}

@book{hastie2009elements,
  title={The elements of statistical learning: data mining, inference, and prediction},
  author={Hastie, Trevor and Tibshirani, Robert and Friedman, Jerome},
  year={2009},
  publisher={Springer Science \& Business Media}
}

@article{henrion2007structural,
  title={Structural properties of linear probabilistic constraints},
  author={Henrion, Ren{\'e}},
  journal={Optim.},
  volume={56},
  number={4},
  pages={425--440},
  year={2007},
  publisher={Taylor \& Francis}
}

@article{henrion2003optimization,
  title={Optimization of a continuous distillation process under random inflow rate},
  author={Henrion, R and M{\"o}ller, A},
  journal={Comput. Math. with Appl.},
  volume={45},
  number={1-3},
  pages={247--262},
  year={2003},
  publisher={Elsevier}
}

@article{hong2011sequential,
  title={Sequential convex approximations to joint chance constrained programs: A Monte Carlo approach},
  author={Hong, L Jeff and Yang, Yi and Zhang, Liwei},
  journal={Oper. Res.},
  volume={59},
  number={3},
  pages={617--630},
  year={2011},
  publisher={INFORMS}
}

@article{kataoka1963stochastic,
  title={A stochastic programming model},
  author={Kataoka, Shinji},
  journal={J. Econom.},
  pages={181--196},
  year={1963},
  publisher={JSTOR}
}

@article{lagoa2005probabilistically,
  title={Probabilistically constrained linear programs and risk-adjusted controller design},
  author={Lagoa, Constantino M and Li, Xiang and Sznaier, Mario},
  journal={SIAM J. Optim.},
  volume={15},
  number={3},
  pages={938--951},
  year={2005},
  publisher={SIAM}
}

@article{lejeune2007efficient,
  title={An efficient trajectory method for probabilistic production-inventory-distribution problems},
  author={Lejeune, Miguel A and Ruszczy{\'n}ski, Andrzej},
  journal={Oper. Res.},
  volume={55},
  number={2},
  pages={378--394},
  year={2007},
  publisher={INFORMS}
}

@inproceedings{LL2007,
  title={Optimizing 0/1 loss for perceptrons by random coordinate descent},
  author={Li, Ling and Lin, Hsuan-Tien},
  booktitle={2007 International Joint Conference on Neural Networks},
  pages={749--754},
  year={2007},
  organization={IEEE}
}

@article{luedtke2014branch,
  title={A branch-and-cut decomposition algorithm for solving chance-constrained mathematical programs with finite support},
  author={Luedtke, James},
  journal={Math. Program.},
  volume={146},
  number={1},
  pages={219--244},
  year={2014},
  publisher={Springer}
}

@article{luedtke2008sample,
  title={A sample approximation approach for optimization with probabilistic constraints},
  author={Luedtke, James and Ahmed, Shabbir},
  journal={SIAM J. Optim.},
  volume={19},
  number={2},
  pages={674--699},
  year={2008},
  publisher={SIAM}
}

@article{luedtke2010integer,
  title={An integer programming approach for linear programs with probabilistic constraints},
  author={Luedtke, James and Ahmed, Shabbir and Nemhauser, George L},
  journal={Math. Program.},
  volume={122},
  number={2},
  pages={247--272},
  year={2010},
  publisher={Springer}
}

@article{miller1965chance,
  title={Chance constrained programming with joint constraints},
  author={Miller, Bruce L and Wagner, Harvey M},
  journal={Oper. Res.},
  volume={13},
  number={6},
  pages={930--945},
  year={1965},
  publisher={INFORMS}
}

@article{nemirovski2007convex,
  title={Convex approximations of chance constrained programs},
  author={Nemirovski, Arkadi and Shapiro, Alexander},
  journal={SIAM J. Optim.},
  volume={17},
  number={4},
  pages={969--996},
  year={2007},
  publisher={SIAM}
}

@inproceedings{osuna1998reducing,
  title={Reducing the run-time complexity of support vector machines},
  author={Osuna, Edgar and Girosi, Federico},
  booktitle={International Conference on Pattern Recognition (submitted)},
  year={1998}
}

@article{pagnoncelli2009sample,
  title={Sample average approximation method for chance constrained programming: theory and applications},
  author={Pagnoncelli, Bernardo K and Ahmed, Shabbir and Shapiro, Alexander},
  journal={J. Optim. Theory. Appl.},
  volume={142},
  number={2},
  pages={399--416},
  year={2009},
  publisher={Springer}
}

@article{pena2020solving,
  title={Solving chance-constrained problems via a smooth sample-based nonlinear approximation},
  author={Pe{\~n}a-Ordieres, Alejandra and Luedtke, James R and W$\ddot{a}$chter, Andreas},
  journal={SIAM J. Optim.},
  volume={30},
  number={3},
  pages={2221--2250},
  year={2020},
  publisher={SIAM}
}

@article{prekopa1973contributions,
  title={Contributions to the theory of stochastic programming},
  author={Pr{\'e}kopa, Andras},
  journal={Math. Program.},
  volume={4},
  number={1},
  pages={202--221},
  year={1973},
  publisher={Springer}
}

@book{prekopa2013stochastic,
  title={Stochastic programming},
  author={Pr{\'e}kopa, Andr{\'a}s},
  volume={324},
  year={2013},
  publisher={Springer Science \& Business Media}
}

@article{robinson1980strongly,
  title={Strongly regular generalized equations},
  author={Robinson, Stephen M},
  journal={Math. Oper. Res.},
  volume={5},
  number={1},
  pages={43--62},
  year={1980},
  publisher={INFORMS}
}

@book{RW1998,
  title={Variational analysis},
  author={Rockafellar, R Tyrrell and Wets, Roger J-B},
  volume={317},
  year={2009},
  publisher={Springer Science \& Business Media}
}

@book{shapiro2021lectures,
  title={Lectures on stochastic programming: modeling and theory},
  author={Shapiro, Alexander and Dentcheva, Darinka and Ruszczy{\'n}ski, Andrzej},
  year={2021},
  publisher={SIAM}
}

@article{sun2014asymptotic,
  title={Asymptotic analysis of sample average approximation for
  stochastic optimization problems with joint chance constraints via conditional value at risk and
  difference of convex functions},
  author={Sun, Hailin and Xu, Huifu and Wang, Yong},
  journal={J. Optim. Theory. Appl.},
  volume={161},
  number={1},
  pages={257--284},
  year={2014},
  publisher={Springer}
}

@article{sun2019Convergence,
  title={Convergence Analysis and a DC Approximation Method
  for Data-driven Mathematical Programs with Distributionally Robust Chance Constraints},
  author={Sun, Hailin and Zhang, Dali and Chen, Yannan},
  journal={$http://www.optimization-online.org/DB_HTML/2019/11/7465.html$},
  year={2019},
}

@article{takyi1999surface,
  title={Surface water quality management using a multiple-realization chance constraint method},
  author={Takyi, Andrews K and Lence, Barbara J},
  journal={Water Resour. Res.},
  volume={35},
  number={5},
  pages={1657--1670},
  year={1999},
  publisher={Wiley Online Library}
}

@article{wang2021support,
  title={Support Vector Machine Classifier via $L_{0/1}$ Soft-Margin Loss},
  author={Wang, Huajun and Shao, Yuanhai and Zhou, Shenglong and Zhang, Ce and Xiu, Naihua},
  journal={IEEE Trans. Pattern Anal. and Mach. Intell.},
  volume={44},
  number={10},
  pages={7253--7265},
  year={2021},
  publisher={IEEE}
}

@article{zhou2022computing,
  title={Computing one-bit compressive sensing via double-sparsity constrained optimization},
  author={Zhou, Shenglong and Luo, Ziyan and Xiu, Naihua and Li, Geoffrey Ye},
  journal={ IEEE Trans. Signal Process},
  volume={70},
  pages={1593--1608},
  year={2022},
  publisher={IEEE}
}

@article{zhou2021quadratic,
  title={Quadratic Convergence of Smoothing Newton's Method for 0/1 Loss Optimization},
  author={Zhou, Shenglong and Pan, Lili and Xiu, Naihua and Qi, Hou-Duo},
  journal={SIAM J. Optim.},
  volume={31},
  number={4},
  pages={3184--3211},
  year={2021},
  publisher={SIAM}
}


\end{document}